\documentclass[11pt]{article}

\usepackage[margin=1in]{geometry}
\usepackage{amsmath,amssymb,amsthm,mathtools}
\usepackage{mathrsfs}
\usepackage{enumitem}
\usepackage{array}
\usepackage[colorlinks=true,linkcolor=blue,citecolor=blue,urlcolor=blue]{hyperref}
\hypersetup{
  pdftitle={Response Calculus for Spectral Simplicity and Joint Eigenvalue Densities},
  pdfauthor={Chunhao Cai},
  pdfsubject={Gaussian response calculus, simple spectra, and joint eigenvalue densities},
  pdfkeywords={Gaussian multiplicative chaos, spectral response, Liouville Brownian motion, Liouville--Cauchy operator, simple spectrum, eigenvalue densities}
}

\numberwithin{equation}{section}

\newtheorem{theorem}{Theorem}[section]
\newtheorem{proposition}[theorem]{Proposition}
\newtheorem{lemma}[theorem]{Lemma}
\newtheorem{corollary}[theorem]{Corollary}
\newtheorem{assumption}[theorem]{Assumption}
\theoremstyle{definition}
\newtheorem{definition}[theorem]{Definition}
\theoremstyle{remark}
\newtheorem{remark}[theorem]{Remark}

\newcommand{\LCircle}{\mathbb T}
\newcommand{\R}{\mathbb R}
\newcommand{\E}{\mathcal E}
\newcommand{\Hc}{H}
\newcommand{\Hreg}{H_{\mathrm{reg}}}
\newcommand{\Vd}{V}
\newcommand{\Xspace}{X}
\newcommand{\State}{S}
\newcommand{\GreenOp}{G}
\newcommand{\Pot}{U_D}
\newcommand{\Potential}{U}
\newcommand{\ME}{\mathcal M_{\rm fe}(D)}
\newcommand{\WinProj}{\Pi}
\newcommand{\RieszProj}{R}
\newcommand{\EigProj}{P}
\newcommand{\OrthProj}{\operatorname{pr}}
\newcommand{\Lzero}{L_0}
\newcommand{\Mult}{\operatorname{Mult}}
\newcommand{\Mc}{\mathcal M}
\newcommand{\eps}{\varepsilon}
\newcommand{\dd}{\,d}
\newcommand{\inner}[2]{\left\langle #1,#2\right\rangle}
\newcommand{\norm}[1]{\left\|#1\right\|}
\newcommand{\one}{\mathbf 1}
\newcommand{\Dom}{\operatorname{Dom}}
\newcommand{\Response}{B}
\newcommand{\Compress}{C}
\DeclareMathOperator{\Cov}{Cov}
\DeclareMathOperator{\Tr}{Tr}
\DeclareMathOperator{\supp}{supp}
\DeclareMathOperator{\rank}{rank}
\DeclareMathOperator{\Spec}{Spec}
\DeclareMathOperator{\Disc}{Disc}
\title{Response Calculus for Spectral Simplicity\\
and Joint Eigenvalue Densities}
\author{Chunhao Cai\\
{\small School of Mathematics (Zhuhai), Sun Yat-sen University}\\
{\small \texttt{caichh9@mail.sysu.edu.cn}}}
\date{}

\begin{document}
\maketitle

\begin{abstract}
We develop a perturbative response calculus for spectral problems obtained by
changing the speed measure of a fixed symmetric energy form in a Gaussian
environment.  If Cameron--Martin translation acts by
$\mu_{x+f}=e^{\kappa f}\mu_x$, unitary transport identifies the varying
$L^2$ spaces and produces a common-domain analytic family.  For a positive
eigenvalue $\Lambda$, the first-order operator is $-\kappa\Lambda$ times the
compression of multiplication by $f$ to the $\Lambda$-eigenspace; its
eigenvalues are the derivatives of the analytic branches issuing from
$\Lambda$.  A
countable separation condition and finite-dimensional Gaussian
disintegration then give almost-sure simplicity; a response-transversality
condition and the inverse function theorem give joint densities for all
finite vectors of positive ordered eigenvalues.  We verify these hypotheses
for every $0<\gamma<2$ in two models: Dirichlet Liouville Brownian motion on
an arbitrary bounded connected planar domain, and the Liouville--Cauchy
operator on the circle.  In the first model the whole spectrum is almost
surely simple; in the second the constants form the deterministic zero mode
and the positive spectrum is almost surely simple.  Transversality follows
from a local eigenfunction-square identity in the Brownian case and its
nonlocal jump-form analogue in the Cauchy case.
\end{abstract}



\begingroup
\small
\tableofcontents
\endgroup

\section{Introduction}
\label{sec:introduction}

Liouville Brownian motion (LBM) is the diffusion obtained by time-changing
planar Brownian motion with the Liouville quantum gravity area measure.  Its
construction from Gaussian multiplicative chaos (GMC), together with the
associated Dirichlet form, heat kernel, and resolvent, was developed in
\cite{BerestyckiLQGDiffusion,GarbanRhodesVargasLBM,
GarbanRhodesVargasHeatKernel,AndresKajinoHeatKernel,
MaillardRhodesVargasZeitouniHeatKernel}.  Recent work establishes a Weyl law
\cite{BerestyckiWongWeyl} and a second-order annealed heat-trace expansion
\cite{BerestyckiKleinHeatTrace}; see
\cite{BerestyckiSpectralGeometry} for an overview and open problems.  Such
counting results do not determine whether
individual eigenvalues collide or whether finite collections of indexed
eigenvalues possess joint densities.  In particular, simple spectrum for the
Dirichlet LBM generator was posed explicitly in
\cite[Problem~4.2]{BerestyckiSpectralGeometry}.

The second model is the time change of the symmetric Cauchy process on the
circle by boundary GMC, introduced in
\cite[Chapter~4, especially Section~4.3.4]{BaverezThesis}; see also
\cite{OoiGMCStable} for a related stable-process convergence result.  Its
deterministic zero mode and nonlocal jump form make it a genuinely different
test of the same response mechanism.

We prove almost-sure simplicity and joint absolute continuity for every finite
vector of positive ordered eigenvalues in both models.  In particular, the LBM
simplicity result resolves \cite[Problem~4.2]{BerestyckiSpectralGeometry}.

\subsection{Models and main results}
\label{subsec:intro-models-main-results}

\paragraph{Abstract response principle.}
Both models described above fit the same scheme: a fixed non-negative energy
form is paired with a random finite speed measure, regular Cameron--Martin
translations act by exponential tilting, and unitary transport turns the
varying $L^2$ spaces into an analytic perturbation problem on a fixed Hilbert
space.  We first state the abstract criterion and then identify its two
realizations.

\paragraph{Abstract setting and model inputs.}
Let $\State$ be a locally compact second-countable Hausdorff space,
$\Vd_{\rm e}$ a real linear space of functions on $\State$, and $\E$ a fixed
non-negative symmetric form on $\Vd_{\rm e}$.  Fix a class
$\mathcal M_{\rm adm}$ of finite positive Radon measures and set
\begin{equation}
 \Vd_\mu:=\Vd_{\rm e}\cap L^2(\mu),
 \qquad \mu\in\mathcal M_{\rm adm}.
 \label{eq:abstract-trace-domain}
\end{equation}
Let $(\Xspace,\Hc,\mathbb P)$ be an abstract Wiener space and
for each environment $x\in\Xspace$ let $\mu_x$ be the corresponding finite
positive Radon measure on $\State$.  Thus $\State$ is the state space,
$\Xspace$ the ambient Banach space, $\Hc$ its Cameron--Martin space, and
$x\in\Xspace$ the environment variable.  Let
$\Hreg\subset\Hc$ be a linear space of bounded Borel functions on $\State$,
and fix $\kappa>0$.  This paragraph specifies the ambient objects.  The
following five assumptions are the standing model-dependent inputs; they are
verified for LBM in
Section~\ref{sec:lbm-realization} and for the Liouville--Cauchy operator in
Section~\ref{sec:lcp-realization}.

For $\mu\in\mathcal M_{\rm adm}$, write
\[
 \norm{u}_{\E,\mu}^2:=\E(u,u)+\norm{u}_{L^2(\mu)}^2,
 \qquad u\in\Vd_\mu.
\]

\begin{assumption}
\label{ass:admissible-forms}
For every $\mu\in\mathcal M_{\rm adm}$, the form $(\E,\Vd_\mu)$ is densely
defined and closed on $L^2(\mu)$,
$\dim L^2(\mu)=\infty$, and
\begin{equation}
 j_\mu:(\Vd_\mu,\norm{\cdot}_{\E,\mu})\hookrightarrow L^2(\mu)
 \quad\text{is compact}.
 \label{eq:abstract-compact-form-embedding}
\end{equation}
\end{assumption}

For $\mu\in\mathcal M_{\rm adm}$, let $A_\mu$ be the associated non-negative
self-adjoint operator.  Proposition~\ref{prop:abstract-admissible-spectral-consequences}
shows that $A_\mu$ has compact resolvent; write its positive eigenvalues, with
multiplicity, as
\begin{equation}
 0<\Lambda_1^\mu\le\Lambda_2^\mu\le\cdots\uparrow\infty.
 \label{eq:intro-positive-eigenvalue-convention}
\end{equation}
We call the positive spectrum simple when all inequalities in
\eqref{eq:intro-positive-eigenvalue-convention} are strict; no condition on
$\ker A_\mu$ is included.

\begin{assumption}
\label{ass:regular-form-multipliers}
For every $\mu\in\mathcal M_{\rm adm}$ and $f\in\Hreg$,
\[
 f\Vd_\mu\subset\Vd_\mu.
\]
\end{assumption}

Since $f\in L^\infty(\mu)$, Assumptions~\ref{ass:admissible-forms}--
\ref{ass:regular-form-multipliers} and the closed graph theorem give, for some
$C_{\mu,f}<\infty$,
\begin{equation}
 \Mult_f\in\mathcal L(\Vd_\mu),
 \qquad
 \norm{fu}_{\E,\mu}\le C_{\mu,f}\norm{u}_{\E,\mu},
 \qquad u\in\Vd_\mu.
 \label{eq:abstract-regular-form-multiplier}
\end{equation}

\begin{assumption}
\label{ass:coherent-orbit}
There exist a canonical event $\Omega_{\rm can}\in\mathcal B(\Xspace)$ and a
structural event $\Omega_{\rm str}\in\mathcal B(\Xspace)$ such that
\begin{equation}
 \Omega_{\rm str}\subset\Omega_{\rm can},
 \qquad
 \mathbb P(\Omega_{\rm str})=1,
 \qquad
 \mu_x\in\mathcal M_{\rm adm}\quad(x\in\Omega_{\rm str}).
 \label{eq:abstract-structural-data}
\end{equation}
For every $f\in\Hreg$, the chosen family is shift-covariant on
$\Omega_{\rm can}$:
\begin{equation}
 \mu_{x+f}=e^{\kappa f}\mu_x,
 \qquad
 x,x+f\in\Omega_{\rm can},
 \label{eq:pointwise-abstract-orbit}
\end{equation}
and the admissible class is stable under the same tilt:
\begin{equation}
 \{e^{\kappa f}\mu:\mu\in\mathcal M_{\rm adm}\}
 \subset\mathcal M_{\rm adm}.
 \label{eq:abstract-admissible-orbit-closure}
\end{equation}
\end{assumption}

\begin{assumption}
\label{ass:direction-separation}
There exist a reference measure $\rho$ and a collection $\mathcal Q$ of
regular perturbation directions such that
\begin{equation}
 \rho\in\{0\}\cup\mathcal P(\State),
 \qquad
 0\ne\mathcal Q\subset\Hreg,
 \qquad
 \mathcal Q\text{ is a countable }\mathbb Q\text{-vector space}.
 \label{eq:countable-direction-space}
\end{equation}
The only finite signed Radon measures invisible to every direction in
$\mathcal Q$ are the multiples of $\rho$:
\begin{equation}
 \mathcal Q^\perp
 :=\left\{\sigma\in\Mc_{\rm fin}(\State):\int_{\State}q\,d\sigma=0
          \text{ for all }q\in\mathcal Q\right\}
 =\operatorname{span}\{\rho\}.
 \label{eq:direction-annihilator}
\end{equation}
\end{assumption}

\begin{assumption}
\label{ass:spectral-borel}
For $x\in\Omega_{\rm str}$ and rational $0<a<b$, put
$\WinProj_{a,b}^x:=\one_{(a,b)}(A_{\mu_x})$.  For every such $a,b$, every
$n,k\ge1$, and every $q\in\mathcal Q$, the following functions on
$\Omega_{\rm str}$ are measurable:
\begin{equation}
 x\longmapsto\Lambda_n^{\mu_x},
 \qquad
 x\longmapsto\Tr\!\left[
  (\WinProj_{a,b}^x\Mult_q\WinProj_{a,b}^x)^k\right]
 \label{eq:minimal-borel-interface}
\end{equation}
\end{assumption}

Window ranks require no separate assumption, since
\[
 \rank\WinProj_{a,b}^x
 =\sum_{n\ge1}\one_{(a,b)}(\Lambda_n^{\mu_x}),
 \qquad x\in\Omega_{\rm str}.
\]
When these scalar data are used on all of $\Xspace$, we extend them by zero
on $\Xspace\setminus\Omega_{\rm str}$.

Assumption~\ref{ass:spectral-borel} is used only to make the collision and
submersion events in Sections~\ref{sec:slicing-simplicity}--\ref{sec:density}
measurable.  The operators and their compact resolvents already follow from
Assumption~\ref{ass:admissible-forms} and
Proposition~\ref{prop:abstract-admissible-spectral-consequences}.

For a positive eigenvalue $\Lambda$ of $A_\mu$, let
$\EigProj_\Lambda^\mu$ be the $L^2(\mu)$-orthogonal projection onto
$\ker(A_\mu-\Lambda)$, and let $\Mult_f$ denote multiplication by $f$.
Proposition~\ref{prop:abstract-cluster-response} shows that, under the tilt
$d\mu_t=e^{\kappa t f}\,d\mu$, the first-order cluster response is
\begin{equation}
 \Response_f^\Lambda(\mu)
 =-\kappa\Lambda\,
 \EigProj_\Lambda^\mu\Mult_f|_{\ker(A_\mu-\Lambda)}.
 \label{eq:intro-cluster-response}
\end{equation}
For a simple $L^2(\mu)$-normalized eigenfunction $\phi$, this becomes
\begin{equation}
 \left.\frac{d}{dt}\right|_{t=0}\Lambda(t)
 =-\kappa\Lambda\int_{\State}f\phi^2\,d\mu.
 \label{eq:intro-simple-response}
\end{equation}
Here $\Lambda(0)=\Lambda$ is the corresponding analytic eigenvalue branch.
Accordingly, for every simple positive eigenvalue $\Lambda_n^\mu$ with a real normalized
eigenfunction $\phi_n^\mu$, define
\begin{equation}
 \ell_n^\mu(q):=-\kappa\Lambda_n^\mu
 \int_{\State}q(\phi_n^\mu)^2\,d\mu,
 \qquad q\in\mathcal Q.
 \label{eq:abstract-response-functionals-Q}
\end{equation}

\begin{definition}
\label{def:abstract-response-transversality}
Let $\Lambda_{n_1}^\mu,\ldots,\Lambda_{n_M}^\mu$ be distinct simple positive
eigenvalues.  The family $(\Lambda_{n_i}^\mu)_{i=1}^M$ is called
\emph{response-transverse} if its response functionals in
\eqref{eq:abstract-response-functionals-Q} are linearly independent on
$\mathcal Q$.  Equivalently, for real $a_1,\ldots,a_M$,
\begin{equation}
 \sum_{i=1}^Ma_i\ell_{n_i}^\mu(q)=0
 \quad\text{for every }q\in\mathcal Q
 \quad\Longrightarrow\quad a_1=\cdots=a_M=0.
 \label{eq:abstract-RT}
\end{equation}
\end{definition}

The five assumptions provide the perturbation, measurability, and
direction-separation framework.
Definition~\ref{def:abstract-response-transversality} is the pathwise spectral
condition used only for the joint-density conclusion; the square criteria below
will verify it in the two models.

\begin{theorem}
\label{thm:intro-abstract-response}
In the abstract setup above, suppose
Assumptions~\ref{ass:admissible-forms}--\ref{ass:spectral-borel} hold.
\begin{enumerate}[label=\textnormal{(\roman*)}]
\item The positive spectrum is almost surely simple:
\begin{equation}
 \mathbb P\!\left\{x\in\Omega_{\rm str}:
  \dim\ker(A_{\mu_x}-\Lambda_n^{\mu_x})=1
  \text{ for every }n\ge1
 \right\}=1.
 \label{eq:intro-abstract-simple-spectrum}
\end{equation}
\item Fix $M\ge1$ and $1\le n_1<\cdots<n_M$.  If, for
$\mathbb P$-almost every $x\in\Omega_{\rm str}$, the selected positive
eigenvalues form a response-transverse family in the sense of
Definition~\ref{def:abstract-response-transversality}, then
\begin{equation}
 (\Lambda_{n_1}^{\mu_x},\ldots,\Lambda_{n_M}^{\mu_x})_\#\mathbb P
 \ll\mathcal L^M.
 \label{eq:intro-abstract-density}
\end{equation}
\end{enumerate}
\end{theorem}

Here $F_\#\mathbb P$ denotes the push-forward law of a random vector $F$.
Thus \eqref{eq:intro-abstract-density} means that the selected eigenvalue
vector has a joint density with respect to Lebesgue measure.

Part~\textnormal{(i)} gives simplicity in each model; the model-specific
square identity then verifies the response transversality required in
part~\textnormal{(ii)}.

\paragraph{The two models.}
We now specify the two realizations of the abstract principle.  Fix
$0<\gamma<2$.  Let $D\subset\R^2$ be bounded, open, and connected, with no
regularity assumption on $\partial D$.  Let $h$ be the zero-boundary Gaussian
free field on $D$ with law $\mathbb P$, and let $M_h$ be its full-variance
Gaussian multiplicative chaos with parameter $\gamma$.  The Dirichlet
Liouville operator $A_h$ is the
non-negative self-adjoint operator associated with
\begin{equation}
 \E(u,v)=\frac1{2\pi}\int_D\nabla u\cdot\nabla v\,dz,
 \qquad
 \Vd_h=\{u\in H_0^1(D):\widetilde u\in L^2(M_h)\}.
 \label{eq:intro-lbm-form}
\end{equation}
Write $\Lzero:=-(2\pi)^{-1}\Delta_D$ for the background Dirichlet generator
associated with this normalization.

For the second model, let
\begin{equation}
 \LCircle=\R/(2\pi\mathbb Z),
 \qquad m(d\theta)=\frac{d\theta}{2\pi},
 \qquad A_0e^{in\theta}=|n|e^{in\theta}.
 \label{eq:intro-cauchy-background}
\end{equation}
Let $h$ be the mean-zero circle field with covariance $A_0^{-1}$ on
$L_0^2(m)$ and law $\mathbb P_{\rm C}$, and let $M_h^\partial$ be its boundary
chaos with coefficient $\gamma/2$.  The Liouville--Cauchy operator
$A_h^{\rm C}$ is associated on $L^2(M_h^\partial)$ with
\begin{equation}
 \E(u,v)=\sum_{n\ne0}|n|\widehat u(n)\overline{\widehat v(n)},
 \qquad
 \Vd_h^{\rm C}
 =\{u\in H^{1/2}(\LCircle):\widetilde u\in L^2(M_h^\partial)\}.
 \label{eq:intro-cauchy-form}
\end{equation}
In both displays, $\widetilde u$ is a fixed quasi-continuous representative
for the relevant form.  The superscript ${\rm C}$ distinguishes Cauchy-model
objects; constants form the kernel of the Cauchy form.

In the notation of the abstract theorem, $\kappa$ is the exponential-orbit
coefficient and $\rho$ is the fixed centering measure invisible to the
countable direction core $\mathcal Q$.  The two models correspond to
\begin{equation}
 (\kappa,\rho)=(\gamma,0)
 \qquad\text{for Dirichlet LBM},
 \qquad
 (\kappa,\rho)=\left(\frac\gamma2,m\right)
 \qquad\text{for Liouville--Cauchy},
 \label{eq:intro-model-dictionary}
\end{equation}
where $m$ is normalized Haar measure on the circle.

We now state the resulting model theorems and record the model-specific
identities that verify response transversality.

\paragraph{Dirichlet Liouville Brownian motion.}

\begin{theorem}
\label{thm:intro-lbm-spectrum}
For every $0<\gamma<2$, almost surely, $A_h$ has compact resolvent, and its
positive eigenvalues, listed with multiplicity, satisfy
\begin{equation}
 0<\Lambda_1^h<\Lambda_2^h<\cdots\uparrow\infty.
 \label{eq:intro-lbm-simple-spectrum}
\end{equation}
Moreover, for every $N\ge1$ and $1\le n_1<\cdots<n_N$,
\begin{equation}
 (\Lambda_{n_1}^h,\ldots,\Lambda_{n_N}^h)_\#\mathbb P
 \ll\mathcal L^N.
 \label{eq:intro-lbm-density}
\end{equation}
\end{theorem}

For a real $L^2(M_h)$-normalized eigenpair $A_h\phi=\Lambda\phi$ with
$\Lambda>0$, the key model-specific ingredient for response transversality,
proved in Lemma~\ref{lem:eigenfunction-square-laplacian}, is the
measure identity
\begin{equation}
 \Lzero(\phi^2)
 =2\Lambda\phi^2M_h-2\Gamma_{\E}(\phi)\,\mathcal L^2,
 \label{eq:intro-lbm-square-identity}
\end{equation}
where $\Lzero$ is understood distributionally and $\Gamma_\E$ is the
carr\'e-du-champ density; the precise conventions are
\eqref{eq:lbm-energy-laplacian-convention} and
\eqref{eq:carre-du-champ-definition}.
The first measure on the right is singular with respect to Lebesgue measure,
whereas the second is absolutely continuous.  Together with a Vandermonde
argument, this establishes the response transversality required in
Theorem~\ref{thm:intro-abstract-response}\textnormal{(ii)}.

\paragraph{The Liouville--Cauchy operator.}

\begin{theorem}
\label{thm:intro-lcp-spectrum}
For every $0<\gamma<2$, almost surely, $A_h^{\rm C}$ has compact resolvent, its
kernel consists exactly of the constants, and, with the positive eigenvalues
listed with multiplicity,
\begin{equation}
 0=\Lambda_0^{{\rm C},h}
 <\Lambda_1^{{\rm C},h}<\Lambda_2^{{\rm C},h}<\cdots\uparrow\infty.
 \label{eq:intro-lcp-simple-spectrum}
\end{equation}
Moreover, for every $N\ge1$ and $1\le n_1<\cdots<n_N$,
\begin{equation}
 (\Lambda_{n_1}^{{\rm C},h},\ldots,
  \Lambda_{n_N}^{{\rm C},h})_\#\mathbb P_{\rm C}
 \ll\mathcal L^N.
 \label{eq:intro-lcp-density}
\end{equation}
\end{theorem}

For a real $L^2(M_h^\partial)$-normalized eigenpair
$A_h^{\rm C}\phi=\Lambda\phi$ with $\Lambda>0$, the corresponding
model-specific ingredient is the nonlocal eigenfunction-square identity from
Lemma~\ref{lem:lcp-nonlocal-square-identity}.  The local carr\'e-du-champ term
in \eqref{eq:intro-lbm-square-identity} is replaced by the jump-energy density
$\Gamma_J$ defined in \eqref{eq:lcp-jump-energy-density}:
\begin{equation}
 A_0(\phi^2)
 =2\Lambda\phi^2M_h^\partial-\Gamma_J(\phi)m.
 \label{eq:intro-lcp-square-identity}
\end{equation}
This identity is an equality of finite signed measures.  Its
singular--absolutely-continuous decomposition and the same Vandermonde argument
verify transversality for the centered response measures
\[
 -\frac\gamma2\Lambda
 \bigl(\phi^2M_h^\partial-m\bigr).
\]

\begin{remark}
\label{rem:intro-super-cauchy}
For $1<\alpha<2$, the inverse of the circle form with symbol $|n|^\alpha$
defines an $L^2(m)$-valued Gaussian field, whose exponential weight has an
almost surely strictly positive density relative to $m$; compare
\cite{OoiGMCStable}.  Simplicity should follow once the five model inputs above
are verified, whereas joint densities require a new transversality argument.
\end{remark}

\begin{remark}
\label{rem:intro-sub-cauchy}
For $0<\alpha<1$, the pseudo-Green singularity is of order
$|e^{i\theta}-e^{i\varphi}|^{-(1-\alpha)}$ near the diagonal.  When
$\alpha>1/2$, a Frostman exponent
$\delta>2(1-\alpha)$ is compatible with the present Hilbert--Schmidt route;
for $\alpha\le1/2$, a different compactness argument is needed.
\end{remark}

\begin{remark}
\label{rem:intro-non-gaussian}
The pathwise response formula is deterministic.  Once the corresponding
model inputs are in place, the remaining probabilistic step is to replace
Gaussian slicing by finite coordinate blocks with conditional Lebesgue
densities compatible with the orbit law
$\mu_{x+tf}=e^{\kappa tf}\mu_x$.  Non-Gaussian and L\'evy chaos are natural
testing grounds; see \cite{JunnilaNonGaussianChaos,RhodesSohierVargasLevyChaos}.
\end{remark}

Section~\ref{sec:abstract-setting} collects the functional-analytic
consequences of the abstract hypotheses and constructs analytic Gaussian
slices.
Section~\ref{sec:pathwise-response} derives the cluster-response formula and
the square criterion for transversality.
Sections~\ref{sec:slicing-simplicity} and~\ref{sec:density} use
finite-dimensional Gaussian slicing to prove almost-sure simplicity and joint
absolute continuity.
Finally, Sections~\ref{sec:lbm-realization} and~\ref{sec:lcp-realization}
verify the abstract model inputs and the corresponding local and nonlocal
square identities for Dirichlet LBM and the Liouville--Cauchy operator,
respectively.
Appendices~\ref{app:lbm-finite-energy-response} and~\ref{app:lcp-response}
develop the stronger potential-theoretic response structure summarized above,
including the full Cameron--Martin extensions.  These results are not used in
the proofs of the model theorems.

\section{Functional-analytic preparation and Gaussian slices}
\label{sec:abstract-setting}

\subsection{Basic spectral and tilt consequences}

Apart from the response-transversality hypothesis in
Theorem~\ref{thm:intro-abstract-response}\textnormal{(ii)}, the arguments
used to prove that theorem require no model-dependent inputs beyond
Assumptions~\ref{ass:admissible-forms}--\ref{ass:spectral-borel}.
The optional full Cameron--Martin results and the model-ready square criterion
state their additional hypotheses separately.  The analytic constructions
below use the form
conditions in Assumptions~\ref{ass:admissible-forms}--
\ref{ass:regular-form-multipliers}, the closure clause
\eqref{eq:abstract-admissible-orbit-closure}, and the orbit identity
\eqref{eq:pointwise-abstract-orbit}; the remaining clauses enter only where
cited.  We first extract the operator consequences and put the
tilted operators on fixed Hilbert spaces along finite-dimensional
Cameron--Martin slices.

\begin{proposition}
\label{prop:abstract-admissible-spectral-consequences}
For every $\mu\in\mathcal M_{\rm adm}$, there is a unique non-negative self-adjoint
operator $A_\mu$ associated with $(\E,\Vd_\mu)$, and
\begin{equation}
 (A_\mu+1)^{-1}=j_\mu j_\mu^*.
 \label{eq:abstract-form-resolvent}
\end{equation}
Here
\[
 j_\mu^*:L^2(\mu)\longrightarrow
 (\Vd_\mu,\inner{\cdot}{\cdot}_{\E,\mu})
\]
is the adjoint of $j_\mu$, where
$\inner{u}{v}_{\E,\mu}:=\E(u,v)+\inner{u}{v}_{L^2(\mu)}$.
Consequently $A_\mu$ has compact resolvent, $\ker A_\mu$ is
finite-dimensional, and its positive spectrum is infinite and can be listed,
with multiplicity, as
\begin{equation}
 0<\Lambda_1^\mu\le\Lambda_2^\mu\le\cdots\uparrow\infty.
 \label{eq:abstract-positive-spectrum}
\end{equation}
\end{proposition}

\begin{proof}
The first representation theorem for closed non-negative forms
\cite[Chapter~VI, Section~2]{KatoPerturbation} gives a unique non-negative
self-adjoint $A_\mu$.  For $f\in L^2(\mu)$ and $u\in\Vd_\mu$,
\begin{align*}
 u=j_\mu^*f
 &\Longleftrightarrow
 \E(u,v)+\inner{u}{v}_{L^2(\mu)}
 =\inner{f}{v}_{L^2(\mu)}
 \quad(v\in\Vd_\mu)\notag\\
 &\Longleftrightarrow
 j_\mu u=(A_\mu+1)^{-1}f.
\end{align*}
Thus \eqref{eq:abstract-form-resolvent} holds and its right-hand side is
compact.  Moreover,
\begin{equation*}
 u\in\ker A_\mu
 \Longrightarrow
 \norm{u}_{\E,\mu}=\norm{u}_{L^2(\mu)},
 \qquad
 j_\mu|_{\ker A_\mu}=I_{\ker A_\mu}.
\end{equation*}
Since the identity on a normed space is compact only in finite dimension,
\begin{equation*}
 \dim\ker A_\mu<\infty,
 \qquad
 \dim(\ker A_\mu)^\perp=\infty.
\end{equation*}
Let $(s_n^\mu)_{n\ge1}$ be the nonzero eigenvalues of the compact
positive operator
$(A_\mu+1)^{-1}|_{(\ker A_\mu)^\perp}$, in non-increasing order and with
multiplicity.  The compact spectral theorem gives
\begin{equation*}
 1>s_1^\mu\ge s_2^\mu\ge\cdots\downarrow0,
 \qquad
 \Lambda_n^\mu=(s_n^\mu)^{-1}-1\uparrow\infty,
\end{equation*}
which proves \eqref{eq:abstract-positive-spectrum}.
\end{proof}

For analytic perturbation, complexify the real form domain by setting
\[
 \Vd_\mu^{\mathbb C}:=\Vd_\mu\oplus i\Vd_\mu
\]
and extend $\E$ sesquilinearly, without changing its notation.  Thus
$\Vd_\mu^{\mathbb C}$ is the complex vector space generated by the real
form domain.  Every bounded real operator on $\Vd_\mu$ used below is
extended complex-linearly to $\Vd_\mu^{\mathbb C}$, with the same operator norm.

Fix $\mu\in\mathcal M_{\rm adm}$ and $f_1,\ldots,f_N\in\Hreg$, and set
\[
 v_z:=\sum_{j=1}^Nz_jf_j,
 \qquad
 m_z:=e^{-\kappa v_z/2},
 \qquad z=(z_1,\ldots,z_N)\in\mathbb C^N.
\]
Equation~\eqref{eq:abstract-regular-form-multiplier} makes each
$\Mult_{f_j}$ bounded on $\Vd_\mu^{\mathbb C}$.  The operator exponential therefore
gives
\begin{equation}
 \Mult_{m_z}
 =\exp\!\left(-\frac\kappa2\sum_{j=1}^Nz_j\Mult_{f_j}\right)
 \in\mathcal L(\Vd_\mu^{\mathbb C}),
 \qquad z\in\mathbb C^N.
 \label{eq:abstract-analytic-tilts}
\end{equation}
Thus $z\mapsto\Mult_{m_z}$ is entire in operator norm, with inverse
$\Mult_{m_{-z}}$.  In particular, for every compact $K\subset\mathbb C^N$,
\begin{equation}
 \sup_{z\in K}\left(
 \norm{\Mult_{m_z}}_{\mathcal L(\Vd_\mu^{\mathbb C})}
 +\norm{\Mult_{m_z}^{-1}}_{\mathcal L(\Vd_\mu^{\mathbb C})}\right)<\infty.
 \label{eq:abstract-multiplier-inverse-bounds}
\end{equation}

The structural weighted-form properties in
Assumption~\ref{ass:admissible-forms} are automatically stable under a bounded
real tilt.  Indeed,
\begin{equation}
 e^{-\kappa\norm{v_t}_\infty}\mu
 \le e^{\kappa v_t}\mu
 \le e^{\kappa\norm{v_t}_\infty}\mu,
 \qquad
 \Vd_{e^{\kappa v_t}\mu}=\Vd_\mu,
 \qquad
 \norm{\cdot}_{\E,e^{\kappa v_t}\mu}\asymp
 \norm{\cdot}_{\E,\mu}.
 \label{eq:abstract-real-tilt-equivalence}
\end{equation}
Thus the form properties in Assumption~\ref{ass:admissible-forms} are stable
under bounded tilts; membership in the designated model class follows from
the closure clause \eqref{eq:abstract-admissible-orbit-closure}.

\subsection{Analytic perturbation along Gaussian slices}

Fix an admissible measure $\mu$, let $f\in\Hreg$, put
\begin{equation}
 d\mu_\tau=e^{\kappa\tau f}\,d\mu,
 \qquad
 U_{\tau,f}:L^2(\mu_\tau)\longrightarrow L^2(\mu),
 \qquad
 U_{\tau,f}u=e^{\kappa\tau f/2}u,
 \label{eq:abstract-spectral-unitary}
\end{equation}
and set
$\widehat A_{\tau,f}=U_{\tau,f}A_{\mu_\tau}U_{\tau,f}^{-1}$.
The measures $\mu_\tau$ are admissible by
\eqref{eq:abstract-admissible-orbit-closure}.

\begin{proposition}
\label{prop:abstract-common-domain}
The operator $\widehat A_{\tau,f}$ is associated in $L^2(\mu)$ with
\begin{equation}
 a_{\tau,f}(u,v)
 =\E(e^{-\kappa\tau f/2}u,e^{-\kappa\tau f/2}v),
 \qquad
 \Dom a_{\tau,f}=\Vd_\mu.
 \label{eq:abstract-fixed-space-form}
\end{equation}
For every $T>0$, the form norms
\[
 \norm{u}_{\tau,f}^2
 :=a_{\tau,f}(u,u)+\norm{u}_{L^2(\mu)}^2
\]
are uniformly equivalent for $|\tau|\le T$: there are
$0<c_{T,f}\le C_{T,f}<\infty$ such that
\begin{equation}
 c_{T,f}\norm{u}_{\E,\mu}
 \le \norm{u}_{\tau,f}
 \le C_{T,f}\norm{u}_{\E,\mu},
 \qquad |\tau|\le T.
 \label{eq:abstract-uniform-form-norm-equivalence}
\end{equation}
For $u,v\in\Vd_\mu$,
$\tau\mapsto a_{\tau,f}(u,v)$ is real analytic and
\begin{equation}
 \dot{a}_{0,f}(u,v)
 =-\frac\kappa2\E(fu,v)-\frac\kappa2\E(u,fv).
 \label{eq:abstract-fixed-form-derivative}
\end{equation}
\end{proposition}

\begin{proof}
Put $m_\tau=e^{-\kappa\tau f/2}$.  The map $U_{\tau,f}$ is unitary because
\[
 \norm{U_{\tau,f}u}_{L^2(\mu)}^2
 =\int_{\State}e^{\kappa\tau f}|u|^2\,d\mu
 =\norm{u}_{L^2(\mu_\tau)}^2.
\]
By \eqref{eq:abstract-analytic-tilts} and
\eqref{eq:abstract-multiplier-inverse-bounds}, multiplication by $m_\tau$ is
an automorphism of the form space.  Hence
\begin{align*}
 u\in U_{\tau,f}\Vd_{\mu_\tau}
 &\Longleftrightarrow
 m_\tau u\in\Vd_{\rm e}\text{ and }
 \int|m_\tau u|^2e^{\kappa\tau f}\,d\mu<\infty\\
 &\Longleftrightarrow u\in\Vd_\mu.
\end{align*}
This proves \eqref{eq:abstract-fixed-space-form}.  Set
\begin{equation*}
 B_T^+:=\sup_{|\tau|\le T}
 \norm{\Mult_{m_\tau}}_{\mathcal L(\Vd_\mu^{\mathbb C})},
 \qquad
 B_T^-:=\sup_{|\tau|\le T}
 \norm{\Mult_{m_\tau}^{-1}}_{\mathcal L(\Vd_\mu^{\mathbb C})},
 \qquad
 L_T:=\sup_{|\tau|\le T}
 \norm{\Mult_{m_\tau}}_{\mathcal L(L^2(\mu))}.
\end{equation*}
The first two constants are finite by
\eqref{eq:abstract-multiplier-inverse-bounds}, and $L_T<\infty$ because $f$ is
bounded.  Therefore
\begin{align*}
 \norm{u}_{\tau,f}^2
 &\le \bigl(1+(B_T^+)^2\bigr)\norm{u}_{\E,\mu}^2,
\\
 \norm{u}_{\E,\mu}^2
 &\le (B_T^-)^2\max\{1,L_T^2\}\,
       \norm{u}_{\tau,f}^2.
\end{align*}
This proves \eqref{eq:abstract-uniform-form-norm-equivalence}.

Operator-norm holomorphy of the multiplier gives, in the form space,
\[
 \frac{m_\tau u-u}{\tau}\longrightarrow-\frac\kappa2fu.
\]
Therefore
\begin{align*}
 \frac{a_{\tau,f}(u,v)-a_{0,f}(u,v)}\tau
 &=\E\left(\frac{m_\tau u-u}\tau,m_\tau v\right)
   +\E\left(u,\frac{m_\tau v-v}\tau\right)\\
 &\longrightarrow
 -\frac\kappa2\E(fu,v)-\frac\kappa2\E(u,fv).
\end{align*}
The same operator-valued holomorphy proves analyticity of every form
coefficient.
\end{proof}

We now isolate finitely many Gaussian coordinates.  For every $g\in\Hc$, fix
a Borel representative
$\widehat W_g:\Xspace\to\R$ of the isonormal coordinate $W(g)$
viewed as a measurable linear functional
\cite[Section~2.10]{BogachevGaussianMeasures}.  Thus
\begin{equation}
 \mathbb E[W(f)W(g)]=\inner{f}{g}_{\Hc}.
 \label{eq:abstract-isonormal-field}
\end{equation}

Let $H_N:=\operatorname{span}\{g_1,\ldots,g_N\}\subset\Hc$, where
$g_1,\ldots,g_N$ are linearly independent, and set
\begin{equation}
 \Sigma:=(\inner{g_i}{g_j}_{\Hc})_{i,j=1}^N.
 \label{eq:abstract-slice-gram}
\end{equation}
Define the measurable maps
\begin{equation}
 t(x)
 :=\Sigma^{-1}
 (\widehat W_{g_1}(x),\ldots,\widehat W_{g_N}(x))^\top,
 \qquad
 v_a:=\sum_{j=1}^Na_jg_j,
 \qquad
 x^\perp:=x-v_{t(x)}.
 \label{eq:abstract-slice-coordinates}
\end{equation}

\begin{lemma}
\label{lem:abstract-gaussian-disintegration}
With $t$ and $x^\perp$ defined in
\eqref{eq:abstract-slice-coordinates}, $t\sim N(0,\Sigma^{-1})$ and is independent of
$x^\perp$.  Its density is
\begin{equation}
 \varphi_{H_N}(a)
 =(2\pi)^{-N/2}(\det\Sigma)^{1/2}
 \exp\!\left(-\tfrac12a^\top\Sigma a\right)>0.
 \label{eq:abstract-slice-density}
\end{equation}
Consequently, for every non-negative Borel $F:\Xspace\to[0,\infty]$,
\begin{equation}
 \mathbb E[F(x)]
 =\int_{\Xspace}\int_{\R^N}
 F(x^\perp+v_a)\varphi_{H_N}(a)\,da\,
 \mathbb P_{H_N}^\perp(dx^\perp),
 \label{eq:abstract-disintegration-formula}
\end{equation}
where $\mathbb P_{H_N}^\perp=\operatorname{Law}(x^\perp)$.
\end{lemma}

\begin{proof}
Set
\[
 c=(W(g_1),\ldots,W(g_N))^\top.
\]
Then $c\sim N(0,\Sigma)$,
$t=\Sigma^{-1}c$, and
$\Cov(t)=\Sigma^{-1}$, which gives
\eqref{eq:abstract-slice-density}.  For $L\in\Xspace^*$, let $h_L\in\Hc$
denote its Cameron--Martin representative,
\[
 L(g)=\inner{h_L}{g}_{\Hc},\qquad g\in\Hc,
\]
and set
\[
 b_L
 :=\bigl(\inner{h_L}{g_k}_{\Hc}\bigr)_{k=1}^N.
\]
Since $L(v_t)=b_L^\top t$,
\begin{equation*}
 \Cov(t,L(x^\perp))
 =\Sigma^{-1}b_L-\Cov(t,b_L^\top t)
 =\Sigma^{-1}b_L-\Sigma^{-1}b_L=0.
\end{equation*}
Joint Gaussianity and separability of $\Xspace$ give independence of
$t$ and $x^\perp$; their product law is
\eqref{eq:abstract-disintegration-formula}.
\end{proof}

The canonical orbit identity is available only when both endpoints belong to
$\Omega_{\rm can}$.  The next lemma replaces the canonical measures on almost
every fiber by a coherent family defined for all parameters.

\begin{lemma}
\label{lem:abstract-coherent-fiber}
Let
\[
 H_N:=\operatorname{span}\{e_1,\ldots,e_N\}\subset\Hreg,
\]
where $e_1,\ldots,e_N$ are linearly independent.  Put
$v_t:=\sum_{k=1}^Nt_ke_k$, and let
$t(x)$, $x^\perp$, and $\mathbb P_{H_N}^\perp$ be the corresponding slice
objects from \eqref{eq:abstract-slice-coordinates} and
Lemma~\ref{lem:abstract-gaussian-disintegration}.  Let
$\Omega_0\in\mathcal B(\Xspace)$ satisfy
$\Omega_0\subset\Omega_{\rm can}$ and $\mathbb P(\Omega_0)=1$; in the
applications below, $\Omega_0=\Omega_{\rm str}$.
For $\mathbb P_{H_N}^\perp$-almost every $x^\perp$, the set
\begin{equation}
 T_0(x^\perp)
 :=\{t\in\R^N:x^\perp+v_t\in\Omega_0\}
 \label{eq:abstract-good-fiber-parameters}
\end{equation}
satisfies
\[
 T_0(x^\perp)\in\mathcal B(\mathbb R^N),
 \qquad
 \mathcal L^N(\mathbb R^N\setminus T_0(x^\perp))=0.
\]
For such $x^\perp$ and any
$s\in T_0(x^\perp)$, define
\begin{equation}
 \mu_t^{\rm coh}
 :=e^{\kappa(v_t-v_s)}\mu_{x^\perp+v_s},
 \qquad t\in\R^N.
 \label{eq:abstract-coherent-fiber-measure}
\end{equation}
Then
\begin{align}
 d\mu_t^{\rm coh}
 &=e^{\kappa(v_t-v_r)}\,d\mu_r^{\rm coh},
 &&r,t\in\R^N,
 \label{eq:abstract-fiber-covariance}\\
 \mu_t^{\rm coh}
 &=\mu_{x^\perp+v_t},
 &&t\in T_0(x^\perp).
 \label{eq:abstract-canonical-coherent-agreement}
\end{align}
If $\Omega_0\subset\Omega_{\rm str}$, then
\begin{equation}
 \mu_t^{\rm coh}\in\mathcal M_{\rm adm},
 \qquad t\in\R^N.
 \label{eq:abstract-coherent-fiber-admissibility}
\end{equation}
\end{lemma}

\begin{proof}
Lemma~\ref{lem:abstract-gaussian-disintegration} gives
\[
 0=\mathbb P(\Omega_0^c)
 =\int_{\Xspace}\int_{\R^N}
 \one_{\Omega_0^c}(x^\perp+v_t)\varphi_{H_N}(t)\,dt\,
 d\mathbb P_{H_N}^\perp(x^\perp).
\]
The integrand is a non-negative measurable function on the product space.
Hence, for almost every fiber,
\[
 \int_{\R^N}\one_{T_0(x^\perp)^c}(t)
 \varphi_{H_N}(t)\,dt=0.
\]
Since $\varphi_{H_N}(t)>0$ for every $t$,
$\mathcal L^N(T_0(x^\perp)^c)=0$.  Continuity of
$t\mapsto x^\perp+v_t$ also gives
$T_0(x^\perp)\in\mathcal B(\mathbb R^N)$.

Equation~\eqref{eq:abstract-fiber-covariance} follows directly from
\eqref{eq:abstract-coherent-fiber-measure}.  If
$s,t\in T_0(x^\perp)$, both endpoints lie in $\Omega_{\rm can}$, and
the orbit identity \eqref{eq:pointwise-abstract-orbit}, applied to
$v_t-v_s\in\Hreg$, gives
\[
 \mu_{x^\perp+v_t}
 =e^{\kappa(v_t-v_s)}\mu_{x^\perp+v_s}
 =\mu_t^{\rm coh}.
\]
If $\Omega_0\subset\Omega_{\rm str}$, then
\eqref{eq:abstract-structural-data} makes
$\mu_{x^\perp+v_s}$ admissible and
\[
 \mu_t^{\rm coh}
 =e^{\kappa(v_t-v_s)}\mu_{x^\perp+v_s}
 \overset{\eqref{eq:abstract-admissible-orbit-closure}}{\in}
 \mathcal M_{\rm adm},
 \qquad t\in\R^N.
\]
The base point $s$ is selected only after fixing the fiber; no measurable choice
of $s$ as a function of $x^\perp$ is required.
\end{proof}

Once a coherent admissible fiber has been obtained, the following
deterministic perturbation statement applies.  Fix
$e_1,\ldots,e_N\in\Hreg$ and an admissible measure $\mu_0$, and set
\[
 v_z:=\sum_{k=1}^Nz_ke_k,
 \qquad
 m_z:=e^{-\kappa v_z/2}.
\]
For $t\in\R^N$, define
\begin{equation}
 d\mu_t=e^{\kappa v_t}\,d\mu_0,
 \label{eq:abstract-kato-measures}
\end{equation}
By \eqref{eq:abstract-admissible-orbit-closure}, every $\mu_t$ is admissible.
Set
\[
 U_tu:=e^{\kappa v_t/2}u,
 \qquad
 \widehat A_t:=U_tA_{\mu_t}U_t^{-1}.
\]

\begin{proposition}
\label{prop:abstract-kato-slice}
For every $t_0\in\R^N$, there exists a complex neighborhood
$\mathcal U\subset\mathbb C^N$ of $t_0$ such that the following assertions hold.
\begin{enumerate}[label=\textnormal{(\roman*)}]
\item The family of sesquilinear forms $(a_z)_{z\in\mathcal U}$, defined by
\begin{equation}
 a_z(u,v):=\E(m_zu,m_{\bar z}v),
 \qquad \Dom a_z=\Vd_{\mu_0}^{\mathbb C},
 \label{eq:abstract-kato-form}
\end{equation}
is a holomorphic family of closed sectorial forms with common domain; in Kato's
terminology, it is a family of type~\textnormal{(a)}.  Let $\widehat A_z$
denote the associated $m$-sectorial operator.  For real $t\in\mathcal U\cap\R^N$,
this agrees with the previously defined $\widehat A_t$.
\item For every $z\in\mathcal U$, the operator $\widehat A_z$ has compact
resolvent, and the map
$z\mapsto(\widehat A_z+1)^{-1}$ from $\mathcal U$ into
$\mathcal L(L^2(\mu_0))$ is holomorphic in the operator norm.
\item Suppose that a positively oriented contour $\Gamma\subset\mathbb C$ lies
in the resolvent set of $\widehat A_{t_0}$ and encloses a finite cluster of its
eigenvalues.  After shrinking $\mathcal U$ if necessary, $\Gamma$ remains in
the resolvent set of $\widehat A_t$ for every real
$t\in\mathcal U\cap\R^N$, and
\begin{equation}
 \RieszProj_{\Gamma}(t)
 =\frac1{2\pi i}\int_{\Gamma}
   (\zeta-\widehat A_t)^{-1}\,d\zeta
 \label{eq:abstract-riesz-projection}
\end{equation}
is a real-analytic family of finite-rank projections satisfying
\begin{equation*}
 \rank\RieszProj_{\Gamma}(t)=\rank\RieszProj_{\Gamma}(t_0).
\end{equation*}
The coefficients of the characteristic polynomial of
$\widehat A_t|_{\operatorname{Ran}\RieszProj_{\Gamma}(t)}$, and hence its
discriminant, depend real-analytically on $t$.
\end{enumerate}
If $t\mapsto(\lambda(t),\Phi(t))$ is a local real-analytic eigenpair branch of
$\widehat A_t$, normalized by
$\norm{\Phi(t)}_{L^2(\mu_0)}=1$, then $\Phi$ is real analytic as a map with
values in $\Vd_{\mu_0}^{\mathbb C}$ and
\begin{equation}
 \partial_{t_k}\lambda(t)
 =(\partial_{t_k}a_t)(\Phi(t),\Phi(t)),
 \qquad 1\le k\le N.
 \label{eq:abstract-form-hellmann-feynman}
\end{equation}
\end{proposition}

\begin{proof}
It suffices to prove the assertions near $t_0=0$.  Indeed, for a general real
$t_0$, replace $\mu_0$ by $\mu_{t_0}$, write $z=t_0+w$, and conjugate back to
$L^2(\mu_0)$ by the fixed unitary
$W_{t_0}u=e^{\kappa v_{t_0}/2}u$.  All the assertions are preserved under this
fixed unitary conjugation.

The common-domain identity follows from
Proposition~\ref{prop:abstract-common-domain}.  Complexify the fixed Hilbert and
form spaces, taking all inner products linear in the first variable.  Since
$\overline{m_{\bar z}}=m_z$, the map
$z\mapsto a_z(u,v)$ is holomorphic.  With
\[
 \norm{u}_+^2=a_0(u,u)+\norm{u}_{L^2(\mu_0)}^2,
\]
Equation~\eqref{eq:abstract-analytic-tilts} gives, locally uniformly,
\begin{equation*}
 |a_z(u,v)-a_0(u,v)|
 \le C|z|\norm{u}_+\norm{v}_+.
\end{equation*}
Choose $r_0>0$ so that $\delta:=Cr_0<1/2$.  For $|z|<r_0$,
\begin{align*}
 \operatorname{Re}\bigl(
   a_z(u,u)+\norm{u}_{L^2(\mu_0)}^2\bigr)
 &\ge(1-\delta)\norm{u}_+^2,
\\
 |\operatorname{Im}a_z(u,u)|
 &\le\frac\delta{1-\delta}
 \operatorname{Re}\bigl(
   a_z(u,u)+\norm{u}_{L^2(\mu_0)}^2\bigr),
\\
 |a_z(u,u)+\norm{u}_{L^2(\mu_0)}^2|
 &\le(1+\delta)\norm{u}_+^2.
\end{align*}
Thus $(a_z)$ is a holomorphic family of closed sectorial forms with common
domain, hence a family of type~(a) in the sense of
\cite[Chapter~VII, Section~4]{KatoPerturbation}.

Let $V_+=(\Vd_{\mu_0}^{\mathbb C},\inner{\cdot}{\cdot}_+)$ and define
$T(z)\in\mathcal L(V_+)$ by
\[
 \inner{T(z)u}{v}_+
 =a_z(u,v)+\inner{u}{v}_{L^2(\mu_0)}.
\]
Then $T(0)=I$, $T(z)$ is operator-norm holomorphic, and, after
shrinking $r_0$, $\norm{T(z)-I}<1$.  Hence
\begin{equation*}
 T(z)^{-1}
 =\sum_{n=0}^{\infty}(I-T(z))^n
\end{equation*}
locally uniformly.  If
$j:V_+\hookrightarrow L^2(\mu_0)$ denotes the embedding, the weak
form equation gives
\begin{equation*}
 (\widehat A_z+1)^{-1}
 =j\,T(z)^{-1}j^*.
\end{equation*}
The embedding $j$ is compact by
Assumption~\ref{ass:admissible-forms}.
Thus every nearby $\widehat A_z$ has compact resolvent and the displayed
resolvent is holomorphic in norm.

If $t\mapsto(\lambda(t),\Phi(t))$ is a real-analytic eigenpair branch in
$L^2(\mu_0)$, then
\begin{equation*}
 \Phi(t)=(1+\lambda(t))T(t)^{-1}j^*\Phi(t)
\end{equation*}
in $V_+$.  Hence $\Phi$ is real analytic with values in the form space,
which justifies the following calculation.  For a normalized branch and any $k$,
\[
 a_t(\Phi,\Phi)
 =\lambda(t)\norm{\Phi}_{L^2(\mu_0)}^2
 =\lambda(t).
\]
\begin{align*}
 \partial_{t_k}\lambda
 &=\partial_{t_k}\bigl[a_t(\Phi,\Phi)\bigr]\notag\\
 &=(\partial_{t_k}a_t)(\Phi,\Phi)
   +a_t(\partial_{t_k}\Phi,\Phi)
   +a_t(\Phi,\partial_{t_k}\Phi)\notag\\
 &=(\partial_{t_k}a_t)(\Phi,\Phi)
   +2\lambda\,\operatorname{Re}
       \inner{\partial_{t_k}\Phi}{\Phi}_{L^2(\mu_0)}
 =(\partial_{t_k}a_t)(\Phi,\Phi),
\end{align*}
because
\begin{equation*}
 2\operatorname{Re}
 \inner{\partial_{t_k}\Phi}{\Phi}_{L^2(\mu_0)}
 =\partial_{t_k}\norm{\Phi}_{L^2(\mu_0)}^2=0.
\end{equation*}
This proves \eqref{eq:abstract-form-hellmann-feynman}.

After shrinking the neighborhood so that $\Gamma$ remains in the resolvent set,
the factorization
\begin{equation*}
 (\widehat A_z-\zeta)^{-1}
 =(\widehat A_z+1)^{-1}
 [I-(\zeta+1)(\widehat A_z+1)^{-1}]^{-1}
\end{equation*}
shows that the corresponding Riesz projection is operator-norm holomorphic in
$z$, and therefore real analytic for real $t$; its rank is locally constant.
Set
\begin{equation*}
 d_{\Gamma}:=\rank\RieszProj_{\Gamma}(t_0),
\end{equation*}
choose a basis $w_1^0,\ldots,w_{d_{\Gamma}}^0$ of its range, and set
\[
 w_j(t)=\RieszProj_{\Gamma}(t)w_j^0,
 \qquad
 A_{\Gamma}(t)
 =\frac1{2\pi i}\int_{\Gamma}
 \zeta(\zeta-\widehat A_t)^{-1}\,d\zeta.
\]
With
\begin{equation*}
 (B_{\Gamma}(t))_{ij}
 =\inner{A_{\Gamma}(t)w_j(t)}{w_i(t)},
 \qquad
 (H_{\Gamma}(t))_{ij}=\inner{w_j(t)}{w_i(t)},
\end{equation*}
the cluster polynomial is
\begin{equation*}
 p_t(\lambda)
 =\frac{\det(\lambda H_{\Gamma}(t)-
                    B_{\Gamma}(t))}
        {\det H_{\Gamma}(t)}.
\end{equation*}
Its coefficients and discriminant are real analytic.  This proves~\textnormal{(iii)}
and completes the proof.
\end{proof}

The slicing arguments in Sections~\ref{sec:slicing-simplicity}
and~\ref{sec:density} integrate over spectral events, so we finally record
their Borel structure.  Let $e_1,\ldots,e_N\in\Hreg$ be linearly independent
and let $(\mu_t)_{t\in\R^N}$ be a coherent admissible family of the type
supplied by Lemma~\ref{lem:abstract-coherent-fiber},
\begin{equation}
 d\mu_t=e^{\kappa\sum_{k=1}^N(t_k-s_k)e_k}\,d\mu_s,
 \qquad s,t\in\R^N.
 \label{eq:abstract-coherent-borel-family}
\end{equation}
Fix $M\ge1$, $1\le n_1<\cdots<n_M$, and $g_1,\ldots,g_M\in\Hreg$, and write
$\nu=(n_1,\ldots,n_M)$ and $g=(g_1,\ldots,g_M)$.  On the locus
where the selected eigenvalues are simple, choose normalized real
eigenfunctions and, for $q\in\Hreg$, set
\begin{equation}
 \ell_{n_i}^{\mu_t}(q)
 :=-\kappa\Lambda_{n_i}^{\mu_t}
 \int_{\State}q(\phi_{n_i}^{\mu_t})^2\,d\mu_t,
 \label{eq:abstract-slice-response-functional}
\end{equation}
and
\begin{equation}
 \mathcal T_{\nu,g}
 :=\left\{t:\ \Lambda_{n_1}^{\mu_t},\ldots,\Lambda_{n_M}^{\mu_t}
 \text{ are simple and }
 \det(\ell_{n_i}^{\mu_t}(g_j))_{i,j=1}^M\ne0\right\}.
 \label{eq:abstract-slice-borel-determinant-event}
\end{equation}

\begin{lemma}
\label{lem:abstract-slice-borel}
Every positive ordered eigenvalue $t\mapsto\Lambda_n^{\mu_t}$ is continuous,
and
\[
 \mathcal T_{\nu,g}\in\mathcal B(\mathbb R^N).
\]
\end{lemma}

\begin{proof}
Fix $s\in\R^N$, put $v_t=\sum_kt_ke_k$, and define
\begin{equation*}
 m_{t,s}:=e^{-\kappa(v_t-v_s)/2},
 \qquad
 U_{t,s}u:=e^{\kappa(v_t-v_s)/2}u,
 \qquad
 \widehat A_{t,s}:=U_{t,s}A_{\mu_t}U_{t,s}^{-1}.
\end{equation*}
Thus all operators act on $L^2(\mu_s)$.
Proposition~\ref{prop:abstract-kato-slice} gives local norm-resolvent analyticity and
compact resolvent, while
\begin{equation*}
 \ker\widehat A_{t,s}
 =m_{t,s}^{-1}
   \{w\in\Vd_{\mu_s}:\E(w,w)=0\}.
\end{equation*}
Because $m_{t,s}$ is invertible, the nullity
\[
 d_0:=\dim\ker\widehat A_{t,s}
\]
is independent of $t$.  Set
\[
 R_{t,s}:=(\widehat A_{t,s}+1)^{-1}
\]
and let $r_k(t,s)$ be its eigenvalues in non-increasing order, with
multiplicity.  Then
\begin{equation*}
 \Lambda_n^{\mu_t}
 =r_{d_0+n}(t,s)^{-1}-1,
 \qquad
 |r_k(t,s)-r_k(t',s)|
 \le\norm{R_{t,s}-R_{t',s}},
 \qquad t,t'\in\R^N.
\end{equation*}
Local norm-resolvent continuity therefore makes every ordered positive
eigenvalue locally, hence globally, continuous.

Put $\Lambda_0^{\mu_t}=0$.  For rational vectors
$\alpha=(\alpha_i)$ and $\beta=(\beta_i)$ with $0<\alpha_i<\beta_i$, define
\begin{equation*}
 \mathcal W_{\alpha,\beta}
 :=\bigcap_{i=1}^M
 \{\Lambda_{n_i-1}^{\mu_t}<\alpha_i<\Lambda_{n_i}^{\mu_t}
   <\beta_i<\Lambda_{n_i+1}^{\mu_t}\}.
\end{equation*}
On this open set, let
$\WinProj_i^{\mu_t}=\one_{(\alpha_i,\beta_i)}(A_{\mu_t})$ and transport it to
the fixed Hilbert space as $\widehat\WinProj_i(t)$.  It is rank one and locally
real analytic.  Since multiplication commutes with the transport,
\begin{align*}
 \Tr(\widehat\WinProj_i(t)\Mult_{g_j}\widehat\WinProj_i(t))
 &=\int_{\State}g_j(\phi_{n_i}^{\mu_t})^2\,d\mu_t.
\end{align*}
Thus, on $\mathcal W_{\alpha,\beta}$,
\[
 D_{\alpha,\beta}(t)
 :=\det\left(
 -\kappa\Lambda_{n_i}^{\mu_t}
 \Tr(\widehat\WinProj_i(t)\Mult_{g_j}\widehat\WinProj_i(t))
 \right)_{i,j=1}^M
\]
is continuous.  The set in
\eqref{eq:abstract-slice-borel-determinant-event} equals the countable union
\begin{equation*}
 \bigcup_{\substack{\alpha,\beta\in\mathbb Q^M\\0<\alpha_i<\beta_i}}
 \{t\in\mathcal W_{\alpha,\beta}:D_{\alpha,\beta}(t)\ne0\},
\end{equation*}
and therefore belongs to $\mathcal B(\mathbb R^N)$.
\end{proof}

\section{Pathwise spectral response calculus}
\label{sec:pathwise-response}

Throughout this section, Assumptions~\ref{ass:admissible-forms} and
\ref{ass:regular-form-multipliers}, together with
\eqref{eq:abstract-admissible-orbit-closure} and the direction space
\eqref{eq:countable-direction-space}, are in force.  The annihilator identity
\eqref{eq:direction-annihilator} is invoked only where cited, and
Assumption~\ref{ass:spectral-borel} is not used.  The optional
potential-realization property is invoked only in
Subsection~\ref{sec:abstract-full-response}.  Fix
an admissible measure $\mu$ and a positive eigenvalue $\Lambda$ of multiplicity
$r$.  Regard
\begin{equation}
 E_\mu(\Lambda):=\ker(A_\mu-\Lambda)
 \label{eq:abstract-eigenspace}
\end{equation}
as a real Hilbert space, and let
$\phi_1,\ldots,\phi_r$ be a real $L^2(\mu)$-orthonormal basis.

For $u,v\in L^2(\mu)$, set
\begin{equation}
 \Theta_\mu(u,v)
 :=uv\,\mu-\inner{u}{v}_{L^2(\mu)}\rho.
 \label{eq:centered-product-measure}
\end{equation}
Whenever the annihilator identity \eqref{eq:direction-annihilator} is invoked,
$\rho\in\mathcal Q^\perp$, so $\int_{\State}q\,d\rho=0$ for every
$q\in\mathcal Q$.  Thus these directions do not see the centering term.

\subsection{Active response of an isolated cluster}

For the cluster at $\Lambda$, define the active response family
\begin{equation}
 \mathcal N_\mu(\Lambda)
 :=\left\{-\kappa\Lambda\Theta_\mu(\psi,\psi):
 \psi\in E_\mu(\Lambda),\ \norm{\psi}_{L^2(\mu)}=1\right\}.
 \label{eq:abstract-active-family}
\end{equation}

This family is compact in total variation.  Indeed, the unit sphere is compact
and the centering term cancels in differences, so
\begin{equation}
 \norm{\kappa\Lambda(\psi^2-\varphi^2)\mu}_{\rm TV}
 \le\kappa\Lambda
 \norm{\psi-\varphi}_{L^2(\mu)}
 \norm{\psi+\varphi}_{L^2(\mu)}.
 \label{eq:abstract-active-tv-continuity}
\end{equation}

\begin{lemma}
\label{lem:abstract-derivative-collapse}
For every $f\in\Hreg$ and $1\le i,j\le r$,
\begin{equation}
 \dot{a}_{0,f}(\phi_i,\phi_j)
 =-\kappa\Lambda\int_{\State}f\phi_i\phi_j\,d\mu.
 \label{eq:abstract-derivative-collapse}
\end{equation}
\end{lemma}

\begin{proof}
Assumption~\ref{ass:regular-form-multipliers} gives
$f\phi_i,f\phi_j\in\Vd_\mu$.  Hence symmetry of $\E$, the weak eigenvalue
equation, and
\eqref{eq:abstract-fixed-form-derivative} give
\begin{align*}
 \dot{a}_{0,f}(\phi_i,\phi_j)
 &=-\frac\kappa2\left[
    \E(f\phi_i,\phi_j)+\E(\phi_i,f\phi_j)\right]\\
 &=-\frac\kappa2\left[
    \E(\phi_j,f\phi_i)+\E(\phi_i,f\phi_j)\right]\\
 &=-\kappa\Lambda\int_{\State}f\phi_i\phi_j\,d\mu.
\end{align*}
\end{proof}

Let $\EigProj_\Lambda^\mu$ be the orthogonal projection onto
$E_\mu(\Lambda)$.  For $f\in\Hreg$ define
\begin{equation}
 \Compress_f^\Lambda(\mu)
 :=\EigProj_\Lambda^\mu\Mult_f|_{E_\mu(\Lambda)},
 \qquad
 \inner{\Compress_f^\Lambda(\mu)u}{v}_{L^2(\mu)}
 =\int fuv\,d\mu.
 \label{eq:abstract-cluster-compression}
\end{equation}
The first-order response operator in direction $f$ is
\begin{equation}
 \Response_f^\Lambda(\mu)
 :=-\kappa\Lambda\Compress_f^\Lambda(\mu).
 \label{eq:abstract-response-operator}
\end{equation}

\begin{proposition}
\label{prop:abstract-cluster-response}
For $f\in\Hreg$, there are $\eps>0$ and real-analytic branches
$\lambda_j:(-\eps,\eps)\to\R$, $1\le j\le r$, listing the
eigenvalues of $\widehat A_{\tau,f}$ issuing from $\Lambda$.  In the basis
$(\phi_i)$,
\begin{equation}
 \left(\inner{\Response_f^\Lambda(\mu)\phi_j}{\phi_i}_{L^2(\mu)}\right)_{i,j}
 =-\kappa\Lambda
 \left(\int f\phi_i\phi_j\,d\mu\right)_{i,j}.
 \label{eq:abstract-response-matrix}
\end{equation}
Moreover,
\begin{equation}
 \{\lambda_1'(0),\ldots,\lambda_r'(0)\}
 =\Spec(\Response_f^\Lambda(\mu))
 \label{eq:abstract-branch-derivatives}
\end{equation}
as multisets.  The statement is independent of the basis.
\end{proposition}

\begin{proof}
Proposition~\ref{prop:abstract-kato-slice} and the Rellich selection theorem give local
analytic eigenvalue branches and form-space analytic eigenvectors
$\Phi_j(\tau)$ that are orthonormal in the fixed $L^2(\mu)$ space.  For
$v\in E_\mu(\Lambda)$, differentiate
\[
 a_{\tau,f}(\Phi_j(\tau),v)
 =\lambda_j(\tau)
   \inner{\Phi_j(\tau)}{v}_{L^2(\mu)}
\]
at zero.  Form symmetry and the weak equation for
$v\in E_\mu(\Lambda)$ imply, for every $w\in\Vd_\mu$,
\begin{equation}
 a_{0,f}(w,v)
 =a_{0,f}(v,w)
 =\Lambda\inner{w}{v}_{L^2(\mu)}.
 \label{eq:abstract-eigenvector-second-slot}
\end{equation}
Consequently,
\begin{align*}
 &\dot{a}_{0,f}(\Phi_j(0),v)
  +a_{0,f}(\Phi_j'(0),v)\\
 &\qquad=\lambda_j'(0)
   \inner{\Phi_j(0)}{v}_{L^2(\mu)}
   +\Lambda\inner{\Phi_j'(0)}{v}_{L^2(\mu)},
\end{align*}
so \eqref{eq:abstract-eigenvector-second-slot} cancels the differentiated
eigenvector terms and yields
\begin{equation*}
 \dot{a}_{0,f}(\Phi_j(0),v)
 =\lambda_j'(0)
  \inner{\Phi_j(0)}{v}_{L^2(\mu)}.
\end{equation*}
Since $(\Phi_j(0))_{j=1}^r$ is an orthonormal basis of
$E_\mu(\Lambda)$, the branch derivatives are, as a multiset, the
eigenvalues of the compression of $\dot{a}_{0,f}$ to that
eigenspace.  Lemma~\ref{lem:abstract-derivative-collapse} identifies this
compression with \eqref{eq:abstract-response-operator}.
\end{proof}

For a scalar spectral quantity $F$ defined along a coherent line, write
\begin{equation}
 \partial_f^\pm F
 :=\lim_{\tau\to0^\pm}\frac{F(\mu_\tau)-F(\mu)}\tau
 \label{eq:abstract-one-sided-derivative}
\end{equation}
when the limit exists.

For $f\in\Hreg$, write
\begin{equation}
 \beta_1(f)\le\cdots\le\beta_r(f)
 \label{eq:abstract-ordered-response-eigenvalues}
\end{equation}
for the ordered eigenvalues of $\Response_f^\Lambda(\mu)$.

\begin{proposition}
\label{prop:abstract-ordered-cluster-response}
Let $\Lambda$ occupy the positive ordered labels
$n,\ldots,n+r-1$.  Then, for every $f\in\Hreg$,
\begin{equation}
 \partial_f^+\Lambda_{n+j-1}^{\mu}=\beta_j(f),
 \qquad
 \partial_f^-\Lambda_{n+j-1}^{\mu}=\beta_{r-j+1}(f),
 \qquad 1\le j\le r.
 \label{eq:abstract-ordered-one-sided-response}
\end{equation}
\end{proposition}

\begin{proof}
Relabel the analytic branches from
Proposition~\ref{prop:abstract-cluster-response} so that
$\lambda_a(\tau)=\Lambda+\tau\beta_a(f)+o(\tau)$.  Since increasing
rearrangement is $1$-Lipschitz in the maximum norm,
\[
 \max_j\left|\Lambda_{n+j-1}^{\mu_\tau}-
 \begin{cases}
  \Lambda+\tau\beta_j(f),&\tau>0,\\
  \Lambda+\tau\beta_{r-j+1}(f),&\tau<0
 \end{cases}\right|=o(|\tau|).
\]
Division by $\tau$ proves
\eqref{eq:abstract-ordered-one-sided-response}.
\end{proof}

\begin{proposition}
\label{prop:abstract-cluster-envelope}
Let $\Lambda$ occupy the positive ordered labels $n,\ldots,n+r-1$, and let
$f\in\Hreg$ satisfy
\begin{equation}
 \int_{\State}f\,d\rho=0.
 \label{eq:abstract-centering-invisible-direction}
\end{equation}
Then
\begin{equation}
 \partial_f^+\Lambda_n^\mu
 =\min_{\eta\in\mathcal N_\mu(\Lambda)}\int f\,d\eta,
 \qquad
 \partial_f^-\Lambda_n^\mu
 =\max_{\eta\in\mathcal N_\mu(\Lambda)}\int f\,d\eta.
 \label{eq:abstract-cluster-envelope}
\end{equation}
\end{proposition}

\begin{proof}
For every unit $\psi\in E_\mu(\Lambda)$,
\begin{align*}
 \int_{\State} f\,d[-\kappa\Lambda\Theta_\mu(\psi,\psi)]
 &=-\kappa\Lambda\left(
   \int_{\State}f\psi^2\,d\mu
   -\int_{\State}f\,d\rho\right)\\
 &=\inner{\Response_f^\Lambda(\mu)\psi}{\psi}_{L^2(\mu)}.
\end{align*}
Rayleigh--Ritz and
Proposition~\ref{prop:abstract-ordered-cluster-response} give
\eqref{eq:abstract-cluster-envelope}.
\end{proof}

This envelope also identifies when the active response can collapse to a single measure.

\begin{lemma}
\label{lem:abstract-active-response-multiplicity}
The active response family satisfies
\begin{equation}
 \mathcal N_\mu(\Lambda)\text{ is a singleton}
 \quad\Longleftrightarrow\quad r=1.
 \label{eq:abstract-active-singleton}
\end{equation}
\end{lemma}

\begin{proof}
If $r=1$, unit eigenvectors differ only by sign.  Conversely, suppose $r\ge2$
and the active family is a singleton.  For orthonormal $u,v$ and
$(u+v)/\sqrt2$, equality of the centered measures cancels the same reference
term and gives
\[
 u^2\mu=v^2\mu=\frac{(u+v)^2}{2}\mu.
\]
Thus $uv=0$ and $u^2=v^2$ $\mu$-almost everywhere, whence
$u^4=u^2v^2=0$, contradicting normalization.
\end{proof}

For a simple positive eigenvalue $\Lambda_n^\mu$ and a real
$L^2(\mu)$-normalized eigenfunction $\phi_n^\mu$, define the centered response
measure
\begin{equation}
 \eta_n^\mu
 :=-\kappa\Lambda_n^\mu
 \Theta_\mu(\phi_n^\mu,\phi_n^\mu)
 =-\kappa\Lambda_n^\mu
 ((\phi_n^\mu)^2\mu-\rho).
 \label{eq:abstract-centered-response-measure}
\end{equation}
It is independent of the sign chosen for $\phi_n^\mu$.

\begin{corollary}
\label{cor:abstract-simple-response}
If $\Lambda_n^\mu$ is simple, then for every $f\in\Hreg$ the two one-sided
derivatives agree and
\begin{equation}
 \partial_f^-\Lambda_n^\mu
 =\partial_f^+\Lambda_n^\mu
 =:\partial_f\Lambda_n^\mu
 =-\kappa\Lambda_n^\mu
 \int_{\State} f(\phi_n^\mu)^2\,d\mu.
 \label{eq:abstract-simple-eigenvalue-response}
\end{equation}
If $\int_{\State}f\,d\rho=0$, then also
\begin{equation}
 \partial_f\Lambda_n^\mu=\int_{\State}f\,d\eta_n^\mu.
 \label{eq:abstract-simple-centered-response}
\end{equation}
\end{corollary}

\begin{proof}
Apply Proposition~\ref{prop:abstract-ordered-cluster-response} with $r=1$.
Equation~\eqref{eq:abstract-simple-centered-response} follows from
\eqref{eq:abstract-centered-response-measure} and
\eqref{eq:abstract-centering-invisible-direction}.
\end{proof}

\subsection{Optional full Cameron--Martin response}
\label{sec:abstract-full-response}

The slicing arguments use only the countable family $\mathcal Q$.  The two
applications also admit a stronger representation of the response
functionals by Green potentials, verified for LBM in
Proposition~\ref{prop:lbm-response-realization} and for Liouville--Cauchy in
Appendix~\ref{app:lcp-response}.
This subsection is not used in the proof of
Theorem~\ref{thm:intro-abstract-response}.

\begin{definition}
\label{def:response-realization}
We say that the response is \emph{potential-realizable} if, for every
admissible $\mu$, there exist a real vector space
$\mathcal M_{\rm resp}(\mu)$ of finite signed measures and a linear
map
\begin{equation}
 \Potential_\mu:\mathcal M_{\rm resp}(\mu)\longrightarrow\Hc
 \label{eq:abstract-response-potential}
\end{equation}
such that $\Theta_\mu(u,v)\in\mathcal M_{\rm resp}(\mu)$ whenever $u$ and $v$
belong to a common positive eigenspace of $A_\mu$, and
\begin{equation}
 \inner{\Potential_\mu\sigma}{f}_{\Hc}
 =\int_{\State}f\,d\sigma,
 \qquad
 \sigma\in\mathcal M_{\rm resp}(\mu),\qquad f\in\Hreg.
 \label{eq:abstract-response-reproducing}
\end{equation}
It is \emph{faithfully potential-realizable} if every $\Potential_\mu$ is
injective.
\end{definition}

For LBM, $\Potential_\mu$ is the Dirichlet Green potential on finite-energy
measures.  For Liouville--Cauchy, it is the centered Cauchy potential on the
zero-mass response measures.

\begin{proposition}
\label{prop:full-cm-extension}
For a potential-realizable response satisfying
\begin{equation}
 \overline{\Hreg}^{\Hc}=\Hc,
 \qquad
 \int_{\State}f\,d\rho=0\quad(f\in\Hreg),
 \label{eq:abstract-full-cm-direction-conditions}
\end{equation}
the map $f\mapsto\Response_f^\Lambda(\mu)$ extends uniquely from $\Hreg$ to a
continuous linear map on $\Hc$, characterized in any orthonormal basis
$(\phi_i)_{i=1}^r$ of $E_\mu(\Lambda)$ by
\begin{equation}
 \left(\inner{\Response_f^\Lambda(\mu)\phi_j}{\phi_i}
              _{L^2(\mu)}\right)_{i,j=1}^r
 =-\kappa\Lambda
 \left(\inner{\Potential_\mu\Theta_\mu(\phi_i,\phi_j)}{f}_{\Hc}
  \right)_{i,j=1}^r,
 \qquad f\in\Hc.
 \label{eq:abstract-full-cm-matrix}
\end{equation}
For $f\notin\Hreg$, this notation refers only to the continuous extension;
no analytic perturbation path is asserted.
\end{proposition}

\begin{proof}
For $f\in\Hc$, first define $\Response_f^{\Lambda,\rm ext}(\mu)$ by
\begin{equation*}
 \inner{\Response_f^{\Lambda,\rm ext}(\mu)\phi_j}{\phi_i}_{L^2(\mu)}
 :=-\kappa\Lambda
 \inner{\Potential_\mu\Theta_\mu(\phi_i,\phi_j)}{f}_{\Hc}.
\end{equation*}
The right-hand side is a symmetric bilinear form on the finite-dimensional
eigenspace, so this defines a unique self-adjoint operator independently of the
chosen orthonormal basis.
For $f,g\in\Hc$,
\begin{align*}
 &\norm{\Response_f^{\Lambda,\rm ext}(\mu)-
       \Response_g^{\Lambda,\rm ext}(\mu)}_{\rm HS}^2\\
 &\quad=\kappa^2\Lambda^2\sum_{i,j=1}^r
 \left|\inner{\Potential_\mu\Theta_\mu(\phi_i,\phi_j)}{f-g}_{\Hc}\right|^2\\
 &\quad\le\kappa^2\Lambda^2\norm{f-g}_{\Hc}^2
 \sum_{i,j=1}^r
 \norm{\Potential_\mu\Theta_\mu(\phi_i,\phi_j)}_{\Hc}^2.
\end{align*}
For $f\in\Hreg$, orthonormality and
\eqref{eq:abstract-full-cm-direction-conditions} give
\begin{align*}
 \inner{\Response_f^{\Lambda,\rm ext}(\mu)\phi_j}{\phi_i}_{L^2(\mu)}
 &=-\kappa\Lambda\int_{\State}f\,d\Theta_\mu(\phi_i,\phi_j)\\
 &=-\kappa\Lambda\left(
   \int_{\State}f\phi_i\phi_j\,d\mu
   -\delta_{ij}\int_{\State}f\,d\rho\right)\\
 &=-\kappa\Lambda\int_{\State}f\phi_i\phi_j\,d\mu
 =\inner{\Response_f^\Lambda(\mu)\phi_j}{\phi_i}_{L^2(\mu)}.
\end{align*}
Thus the candidate agrees with the original response on $\Hreg$.  Since
$\Hreg$ is dense in $\Hc$ by
\eqref{eq:abstract-full-cm-direction-conditions}, the displayed
Hilbert--Schmidt estimate gives the unique continuous extension.
\end{proof}

\begin{corollary}
\label{cor:abstract-positive-response-gram}
For a faithfully potential-realizable response, let
$\Lambda_{n_1}^\mu,\ldots,\Lambda_{n_M}^\mu$ be distinct simple positive
eigenvalues whose centered response measures
$\eta_{n_i}^\mu$ are linearly independent in
$\mathcal M_{\rm resp}(\mu)$.  Then
\begin{equation}
 \left(
 \inner{\Potential_\mu\eta_{n_i}^\mu}
        {\Potential_\mu\eta_{n_j}^\mu}_{\Hc}
 \right)_{i,j=1}^M
 \label{eq:abstract-response-gram}
\end{equation}
defines a positive-definite matrix.
\end{corollary}

\begin{proof}
For $a\in\R^M$, linearity gives
\[
 \sum_{i,j=1}^Ma_ia_j
 \inner{\Potential_\mu\eta_{n_i}^\mu}
        {\Potential_\mu\eta_{n_j}^\mu}_{\Hc}
 =\norm{\Potential_\mu(\sum_i a_i\eta_{n_i}^\mu)}_{\Hc}^2.
\]
If this is zero, injectivity of $\Potential_\mu$ gives
$\sum_i a_i\eta_{n_i}^\mu=0$, and linear independence yields $a=0$.
\end{proof}

\subsection{Response transversality and first-order splitting}
\label{sec:abstract-transversality}

\subsubsection{First-order splitting of a multiple cluster}

\begin{proposition}
\label{prop:abstract-pairwise-nonscalar}
Let $\Lambda>0$ have multiplicity at least two and let $u,v$ be any orthonormal
pair in $E_\mu(\Lambda)$.  Under the direction space
\eqref{eq:countable-direction-space} and the annihilator identity
\eqref{eq:direction-annihilator},
there exists $q\in\mathcal Q$ such that
\begin{equation}
 \begin{pmatrix}
  \int qu^2\,d\mu&\int quv\,d\mu\\
  \int quv\,d\mu&\int qv^2\,d\mu
 \end{pmatrix}
 \label{eq:abstract-two-by-two-compression}
\end{equation}
is not a scalar matrix.
\end{proposition}

\begin{proof}
If the matrix were scalar for every $q\in\mathcal Q$, then
\[
 \int quv\,d\mu=0,
 \qquad
 \int q(u^2-v^2)\,d\mu=0
 \qquad(q\in\mathcal Q).
\]
Thus $uv\mu$ and $(u^2-v^2)\mu$ lie in
$\mathcal Q^\perp=\operatorname{span}\{\rho\}$.  Orthonormality gives
\[
 (uv\mu)(\State)=\int uv\,d\mu=0,
 \qquad
 ((u^2-v^2)\mu)(\State)=1-1=0.
\]
If $\rho$ is a probability measure, then for
$\sigma\in\{uv\mu,(u^2-v^2)\mu\}$,
\[
 \sigma=c\rho,\qquad
 0=\sigma(\State)=c\rho(\State)=c,
\]
so $\sigma=0$; if $\rho=0$, the same conclusion is immediate.  Consequently
\[
 uv=0,
 \qquad
 u^2=v^2
 \quad\mu\text{-almost everywhere}.
\]
It follows that $u^4=u^2v^2=(uv)^2=0$ almost everywhere, contrary to
$\norm{u}_{L^2(\mu)}=1$.
\end{proof}

\begin{lemma}
\label{lem:abstract-simple-matrix}
Let $E$ be a finite-dimensional real inner-product space and
$\mathcal A\subset\operatorname{Sym}(E)$ a linear subspace.  If
\begin{equation*}
 \forall L\subseteq E,\quad \dim L=2
 \quad\Longrightarrow\quad
 \exists B\in\mathcal A:\ \OrthProj_L B|_L\notin\R I_L,
\end{equation*}
where $I_L$ is the identity on $L$.  Then
\begin{equation*}
 \exists B_0\in\mathcal A:\qquad
 \#\Spec(B_0)=\dim E.
\end{equation*}
\end{lemma}

\begin{proof}
For $B\in\mathcal A$, let $N(B)$ be its number of distinct eigenvalues and
choose $B_0\in\mathcal A$ with $N(B_0)$ maximal.  Suppose $B_0$ has an
eigenvalue $\alpha$ with eigenspace $E_\alpha$ of dimension at least
two.  Fix $B\in\mathcal A$ and put
$C=\OrthProj_{E_\alpha}B|_{E_\alpha}$.

Let $\RieszProj_\alpha(\eps)$ be the Riesz projection of
$B_0+\eps B$ around $\alpha$.
For small $\eps$, the map
\begin{equation*}
 U_\alpha(\eps)
 :=\RieszProj_\alpha(\eps)\RieszProj_\alpha(0)
 \left(
  \RieszProj_\alpha(0)\RieszProj_\alpha(\eps)\RieszProj_\alpha(0)
  \big|_{E_\alpha}
 \right)^{-1/2}
\end{equation*}
is an analytic isometry from $E_\alpha$ onto the perturbed Riesz
range.  Define
\[
 B_\alpha^{\rm eff}(\eps)
 :=U_\alpha(\eps)^*(B_0+\eps B)U_\alpha(\eps).
\]
Since $U_\alpha(0)$ is the inclusion of $E_\alpha$ and
$B_0|_{E_\alpha}=\alpha I$,
\begin{align*}
 (B_\alpha^{\rm eff})'(0)
 &=C+\alpha\left.\frac d{d\eps}\right|_{\eps=0}
   \bigl[U_\alpha(\eps)^*U_\alpha(\eps)\bigr]\\
 &=C.
\end{align*}
Analyticity therefore gives
\begin{equation}
 B_\alpha^{\rm eff}(\eps)
 =\alpha I+\eps C+O(\eps^2)
 \quad\text{in operator norm}.
 \label{eq:abstract-effective-matrix-expansion}
\end{equation}

If $C$ is non-scalar, choose two of its ordered eigenvalues
$\xi_p<\xi_q$.  Weyl's inequality and
\eqref{eq:abstract-effective-matrix-expansion} give, for small $\eps>0$,
\[
 \zeta_q(\eps)-\zeta_p(\eps)
 \ge\eps(\xi_q-\xi_p)-O(\eps^2)>0,
\]
where $\zeta_j(\eps)$ are the corresponding cluster eigenvalues.  Other
clusters remain disjoint, so
$N(B_0+\eps B)>N(B_0)$, a contradiction.  Therefore every
$C=\OrthProj_{E_\alpha}B|_{E_\alpha}$ is scalar.  Every
two-dimensional
subspace of $E_\alpha$ then has scalar compression for every
$B\in\mathcal A$, contradicting the hypothesis.  Hence $B_0$ is simple.
\end{proof}

Proposition~\ref{prop:abstract-pairwise-nonscalar} and
Lemma~\ref{lem:abstract-simple-matrix} together produce a first-order splitting
direction in the fixed countable family.

\begin{theorem}
\label{thm:abstract-simple-splitting}
Under \eqref{eq:countable-direction-space} and
\eqref{eq:direction-annihilator}, for every multiple positive
eigenvalue $\Lambda$ of $A_\mu$ there exists $q\in\mathcal Q$ such that
\begin{equation}
 \Response_q^\Lambda(\mu)
 =-\kappa\Lambda\EigProj_\Lambda^\mu
 \Mult_q|_{E_\mu(\Lambda)}
 \label{eq:abstract-simple-splitting-matrix}
\end{equation}
has simple spectrum.  Writing $r=\dim E_\mu(\Lambda)$, let
$\lambda_1,\ldots,\lambda_r$ be the real-analytic eigenvalue branches issuing
from $\Lambda$ along $d\mu_\tau=e^{\kappa\tau q}\,d\mu$, so that
$\lambda_j(0)=\Lambda$ for $1\le j\le r$.  Then
\begin{equation*}
 \lambda_i'(0)\ne\lambda_j'(0),
 \qquad 1\le i<j\le r.
\end{equation*}
\end{theorem}

\begin{proof}
Let $\mathcal A$ be the real linear span of
$\{\Compress_q^\Lambda(\mu):q\in\mathcal Q\}$.  Proposition
\ref{prop:abstract-pairwise-nonscalar} and
Lemma~\ref{lem:abstract-simple-matrix} give a finite real linear combination
$f_* =\sum_{j=1}^Na_j^*q_j$ for which $\Compress_{f_*}^\Lambda(\mu)$ has simple
spectrum.  Define
\[
 \Delta(a):=\Disc\left(\sum_{j=1}^Na_j
 \Compress_{q_j}^\Lambda(\mu)\right),
 \qquad a\in\R^N.
\]
Then $\Delta(a^*)\ne0$.  Choose $b\in\mathbb Q^N$ sufficiently close to $a^*$ that
$\Delta(b)\ne0$.  Since $\mathcal Q$ is a rational vector space,
$q=\sum_jb_jq_j\in\mathcal Q$, and $\Compress_q^\Lambda(\mu)$ is simple.
Multiplication by $-\kappa\Lambda\ne0$ preserves simplicity.  By
Proposition~\ref{prop:abstract-cluster-response}, more precisely
\eqref{eq:abstract-branch-derivatives}, the eigenvalues of
$\Response_q^\Lambda(\mu)$ are exactly the derivatives
$\lambda_1'(0),\ldots,\lambda_r'(0)$, counted with multiplicity.  Hence these
derivatives are pairwise distinct.
\end{proof}

\subsubsection{A square criterion for response transversality}

The singular--absolutely-continuous decomposition below turns the square
identity into a Vandermonde system.  The following hypotheses are a
deterministic criterion for Definition~\ref{def:abstract-response-transversality},
not additional standing assumptions; Sections~\ref{sec:lbm-realization} and
\ref{sec:lcp-realization} verify them from their respective square identities.

\begin{theorem}
\label{thm:abstract-square-transversality}
Let $\nu_{\rm ref}$ be a sigma-finite positive Radon measure on $\State$, and let
$(\Lambda_i,\phi_i)_{i=1}^M$ be real eigenpairs of $A_\mu$ with continuous
representatives.  Assume that the following three conditions hold.
\begin{enumerate}[label=\textnormal{(\roman*)},leftmargin=*,itemsep=0pt,topsep=0.35em,parsep=0pt,partopsep=0pt]
\item The eigenpairs satisfy
\begin{equation}
 A_\mu\phi_i=\Lambda_i\phi_i,
 \qquad
 \Lambda_i>0,
 \qquad
 \Lambda_i\ne\Lambda_j\ (i\ne j),
 \qquad
 \norm{\phi_i}_{L^2(\mu)}=1.
 \label{eq:abstract-square-eigenpair-data}
\end{equation}
\item The measures satisfy
\begin{equation}
 \mu\perp\nu_{\rm ref},
 \qquad
 \operatorname{supp}\mu=\State,
 \qquad
 \rho=0\ \text{or}\ \rho\ll\nu_{\rm ref}.
 \label{eq:abstract-reference-singularity}
\end{equation}
\item There exist $c_0>0$, finite signed measures $b_i\ll\nu_{\rm ref}$, and a
linear map
\begin{equation}
 L:
 \operatorname{span}_{\R}\{\phi_1^2,\ldots,\phi_M^2\}
 \longrightarrow\Mc_{\rm fin}(\State),
 \label{eq:abstract-square-operator-domain}
\end{equation}
such that
\begin{equation}
 L(\phi_i^2)
 =c_0\Lambda_i\phi_i^2\mu-b_i,
 \qquad 1\le i\le M.
 \label{eq:abstract-square-identity}
\end{equation}
\end{enumerate}
Then, for every $a\in\R^M$,
\begin{align}
 \sum_{i=1}^Ma_i\Theta_\mu(\phi_i,\phi_i)=0
 &\Longrightarrow a=0,
 \label{eq:abstract-centered-square-independence}\\
 \sum_{i=1}^Ma_i\Lambda_i\Theta_\mu(\phi_i,\phi_i)=0
 &\Longrightarrow a=0.
 \label{eq:abstract-weighted-square-independence}
\end{align}
\end{theorem}

\begin{proof}
Suppose the left side of
\eqref{eq:abstract-centered-square-independence} vanishes.  Written without
centering, this says
\begin{equation*}
 \sum_{i=1}^Ma_i\phi_i^2\mu
 =\left(\sum_{i=1}^Ma_i\right)\rho.
\end{equation*}
The left side is singular with respect to $\nu_{\rm ref}$, whereas the right side is
absolutely continuous.  Uniqueness of the Lebesgue decomposition gives
\[
 \sum_{i=1}^Ma_i\phi_i^2\mu=0.
\]
Full support and continuity imply
$u_0:=\sum_i a_i\phi_i^2=0$ on $\State$.  Put
\begin{equation*}
 u_k:=\sum_{i=1}^Ma_i\Lambda_i^k\phi_i^2.
\end{equation*}
If $u_k=0$, linearity of $L$ and
\eqref{eq:abstract-square-identity} give
\begin{equation*}
 0=L u_k
 =c_0\sum_{i=1}^Ma_i\Lambda_i^{k+1}\phi_i^2\mu
   -\sum_{i=1}^Ma_i\Lambda_i^kb_i.
\end{equation*}
Again the singular and absolutely continuous parts vanish separately.  Full
support and continuity give $u_{k+1}=0$.  By induction,
$u_k=0$ for $0\le k<M$.

For every $y\in\State$,
\[
 \begin{pmatrix}
  1&\cdots&1\\
  \Lambda_1&\cdots&\Lambda_M\\
  \vdots&&\vdots\\
  \Lambda_1^{M-1}&\cdots&\Lambda_M^{M-1}
 \end{pmatrix}
 \begin{pmatrix}
  a_1\phi_1(y)^2\\ \vdots\\a_M\phi_M(y)^2
 \end{pmatrix}=0.
\]
The Vandermonde determinant is
\begin{equation*}
 \prod_{1\le i<j\le M}(\Lambda_j-\Lambda_i)\ne0.
\end{equation*}
Hence, for every $i$,
\[
 a_i\phi_i^2=0
 \quad\Longrightarrow\quad
 a_i=a_i\int_{\State}\phi_i^2\,d\mu=0.
\]
Since every $\Lambda_i>0$, replacing $a_i$ by $a_i\Lambda_i$ proves
\eqref{eq:abstract-weighted-square-independence}.
\end{proof}

Let $\Lambda_{n_1}^\mu,\ldots,\Lambda_{n_M}^\mu$ be distinct simple positive
eigenvalues with real normalized eigenfunctions
$\phi_{n_1}^\mu,\ldots,\phi_{n_M}^\mu$.  In terms of the centered response
measures \eqref{eq:abstract-centered-response-measure},
\eqref{eq:direction-annihilator} gives
\begin{equation}
 \ell_{n_i}^\mu(q)=\int_{\State}q\,d\eta_{n_i}^\mu,
 \qquad q\in\mathcal Q.
 \label{eq:abstract-response-functional-measure-form}
\end{equation}

The square criterion now gives the required linear independence on
$\mathcal Q$.

\begin{corollary}
\label{cor:square-implies-RT}
Let $M\ge1$, and let
$\Lambda_{n_1}^\mu,\ldots,\Lambda_{n_M}^\mu$ be distinct simple positive
eigenvalues with real normalized eigenfunctions
$\phi_{n_1}^\mu,\ldots,\phi_{n_M}^\mu$ whose eigenpairs satisfy the hypotheses
of Theorem~\ref{thm:abstract-square-transversality}.  Under
\eqref{eq:countable-direction-space} and
\eqref{eq:direction-annihilator}, for
$a\in\R^M$,
\begin{equation*}
 \sum_{i=1}^Ma_i\ell_{n_i}^\mu(q)=0\quad(q\in\mathcal Q)
 \qquad\Longrightarrow\qquad a=0.
\end{equation*}
\end{corollary}

\begin{proof}
Suppose $\sum_i a_i\ell_{n_i}^\mu(q)=0$ for every $q\in\mathcal Q$.  Then
\[
 \sigma:=\sum_{i=1}^Ma_i\eta_{n_i}^\mu
 \in\mathcal Q^\perp=\operatorname{span}\{\rho\}.
\]
If $\rho=0$, then $\sigma=0$.  If $\rho$ is a probability measure, then
\eqref{eq:abstract-centered-response-measure} and normalization give
\[
 \eta_{n_i}^\mu(\State)
 =-\kappa\Lambda_{n_i}^\mu
   \left(\int_{\State}(\phi_{n_i}^\mu)^2\,d\mu-\rho(\State)\right)
 =-\kappa\Lambda_{n_i}^\mu(1-1)=0.
\]
Writing $\sigma=c\rho$ therefore gives
\[
 0=\sigma(\State)=c\rho(\State)=c,
 \qquad\text{hence}\qquad \sigma=0.
\]
By
\eqref{eq:abstract-centered-response-measure}, this is
\[
 -\kappa\sum_{i=1}^Ma_i\Lambda_{n_i}^\mu
 \Theta_\mu(\phi_{n_i}^\mu,\phi_{n_i}^\mu)=0.
\]
Weighted square transversality
\eqref{eq:abstract-weighted-square-independence} gives $a=0$.
\end{proof}

Linear independence on the countable direction set can be witnessed by finitely many directions.

\begin{lemma}
\label{lem:abstract-countable-submersion}
Let the selected response functionals satisfy \eqref{eq:abstract-RT}.  Then there
are $g_1,\ldots,g_M\in\mathcal Q$ such that
\begin{equation}
 \det(\ell_{n_i}^\mu(g_j))_{i,j=1}^M\ne0.
 \label{eq:abstract-nonsingular-response-tuple}
\end{equation}
\end{lemma}

\begin{proof}
Extend every $\ell_{n_i}^\mu$ from $\mathcal Q$ to
$\operatorname{span}_{\R}\mathcal Q$ by the same integral formula, and let
$T:\operatorname{span}_{\R}\mathcal Q\to\R^M$ be
$T(f)=(\ell_{n_1}^\mu(f),\ldots,\ell_{n_M}^\mu(f))$.  Condition
\eqref{eq:abstract-RT} says that the coordinate functionals are linearly
independent, hence $T$ is surjective.  Choose
$f_j\in\operatorname{span}_{\R}\mathcal Q$ with $T(f_j)$ equal to the
$j$th standard basis vector.  After taking the union of the finitely many directions
appearing in these combinations, write
\[
 f_j=\sum_{k=1}^K a_{jk}q_k,
 \qquad q_k\in\mathcal Q.
\]
Choose $b_{jk}\in\mathbb Q$ sufficiently close to $a_{jk}$ and set
\begin{equation*}
 g_j:=\sum_{k=1}^K b_{jk}q_k\in\mathcal Q.
\end{equation*}
Continuity of the determinant and
$\det(\ell_{n_i}^\mu(f_j))_{i,j}=1$ ensure that the rational coefficients can be
chosen so that \eqref{eq:abstract-nonsingular-response-tuple} holds.
\end{proof}

\section{Gaussian slicing and almost-sure simplicity}
\label{sec:slicing-simplicity}

Throughout Sections~\ref{sec:slicing-simplicity} and~\ref{sec:density}, the
five standing assumptions are in force.
Section~\ref{sec:density} additionally restricts to the response-transverse
event in Theorem~\ref{thm:intro-abstract-response}\textnormal{(ii)}; this is a
pathwise spectral property, not another standing model input.

\subsection{Local analytic splitting}

We first show that a response compression with simple spectrum separates the
corresponding multiple eigenvalue along an analytic slice.  This is a local
cluster condition, supplied in the final argument by
Theorem~\ref{thm:abstract-simple-splitting} from
the annihilator identity \eqref{eq:direction-annihilator}.

\begin{lemma}
\label{lem:abstract-local-line-splitting}
Fix an admissible measure $\mu_0$, a direction $g\in\Hreg$, and an eigenvalue
$\Lambda>0$ of $A_{\mu_0}$ of multiplicity $r\ge2$.  Define
\begin{equation*}
 d\mu_s=e^{\kappa sg}\,d\mu_0,\qquad s\in\R.
\end{equation*}
By \eqref{eq:abstract-admissible-orbit-closure}, each $\mu_s$ is admissible.
If $\Compress_g^\Lambda(\mu_0)$ has simple spectrum, then there are $\eps>0$ and a
bounded open interval $I\ni\Lambda$ such that
\begin{equation*}
 \overline I\cap\Spec(A_{\mu_0})=\{\Lambda\}
\end{equation*}
and, for $0<|s|<\eps$,
\begin{equation*}
 \#\bigl(\Spec(A_{\mu_s})\cap I\bigr)=r,
 \qquad
 \dim\ker(A_{\mu_s}-\lambda)=1
 \quad\bigl(\lambda\in\Spec(A_{\mu_s})\cap I\bigr).
\end{equation*}
\end{lemma}

\begin{proof}
Choose $I$ whose closure meets $\Spec(A_{\mu_0})$ only at $\Lambda$.
Proposition~\ref{prop:abstract-cluster-response} supplies, for
$|s|<\eps_0$, analytic branches
$\lambda_1,\ldots,\lambda_r$ listing the spectrum in $I$.
Their first derivatives are the eigenvalues of
$-\kappa\Lambda\Compress_g^\Lambda(\mu_0)$ and hence are pairwise distinct.
Put
\[
 \delta_*:=\min_{i<j}
 |\lambda_i'(0)-\lambda_j'(0)|>0.
\]
Since
\[
 \frac{\lambda_i(s)-\lambda_j(s)}s
 \longrightarrow
 \lambda_i'(0)-\lambda_j'(0),
\]
for some $0<\eps<\eps_0$,
\begin{equation*}
 |\lambda_i(s)-\lambda_j(s)|
 \ge\frac12\delta_*|s|>0,
 \qquad i<j,\quad0<|s|<\eps.
\end{equation*}
The corresponding Riesz projection has constant rank $r$, so these branches
exhaust the spectrum in $I$.
\end{proof}

The exceptional parameters are controlled by the following standard zero-set
theorem.

\begin{lemma}
\label{lem:abstract-analytic-zero-set}
Let $U\subset\R^N$ be connected and open.  If a real-analytic
$F:U\to\R$ is not identically zero, then
\begin{equation*}
 \mathcal L^N\{t\in U:F(t)=0\}=0.
\end{equation*}
\end{lemma}

\begin{proof}
This is the zero-set theorem for non-trivial real-analytic functions; see
\cite{MityaginZeroSet}.
\end{proof}

For $K\ge1$, define the visible collision values at $\mu$ by
\begin{equation*}
 J_K(\mu)
 :=\{\alpha>0:\alpha=\Lambda_j^\mu=\Lambda_{j+1}^\mu
       \text{ for some }1\le j\le K\}.
\end{equation*}
For an eigenvalue $\alpha>0$, we say that a direction $g\in\Hreg$
\emph{splits $E_\mu(\alpha)$ to first order} when
\begin{equation*}
 \Compress_g^\alpha(\mu)\text{ has simple spectrum}.
\end{equation*}
Equivalently, $\Response_g^\alpha(\mu)$ has simple spectrum, because
$\Response_g^\alpha(\mu)=-\kappa\alpha\Compress_g^\alpha(\mu)$ and
$-\kappa\alpha\ne0$.

\begin{lemma}
\label{lem:abstract-local-window-nullity}
Let $\mu_0$ be admissible, and let $H_N\subset\Hreg$ be $N$-dimensional with
basis $e_1,\ldots,e_N$.  Define
\begin{equation*}
 d\mu_t=e^{\kappa\sum_{k=1}^Nt_ke_k}\,d\mu_0,
 \qquad t\in\R^N.
\end{equation*}
By \eqref{eq:abstract-admissible-orbit-closure}, every $\mu_t$ is admissible.
If every $\alpha\in J_K(\mu_0)$ is split to first order by some $g\in H_N$,
then some neighborhood $U$ of zero satisfies
\begin{equation*}
 \mathcal L^N\left\{t\in U:
 \Lambda_j^{\mu_t}=\Lambda_{j+1}^{\mu_t}
 \text{ for some }1\le j\le K\right\}=0.
\end{equation*}
\end{lemma}

\begin{proof}
Choose pairwise disjoint intervals isolating the full multiplicity blocks of
$A_{\mu_0}$ that contain the labels $1,\ldots,K+1$.  Eigenvalue continuity and
constancy of the Riesz ranks give a neighborhood $U_0$ of zero on which these
intervals contain all collisions among those labels.

For each $\alpha\in J_K(\mu_0)$,
Proposition~\ref{prop:abstract-kato-slice} gives a real-analytic cluster
discriminant $\Delta_\alpha^{\rm eig}$ on a connected neighborhood in $U_0$,
vanishing exactly when that cluster has a repeated eigenvalue.
Choose $g_\alpha=\sum_kb_{\alpha,k}e_k\in H_N$ as in the hypothesis.  By
Lemma~\ref{lem:abstract-local-line-splitting},
\begin{equation*}
 \Delta_\alpha^{\rm eig}(sb_\alpha)\ne0
 \qquad(0<|s|<\eps_\alpha).
\end{equation*}
Hence $\Delta_\alpha^{\rm eig}\not\equiv0$, and
Lemma~\ref{lem:abstract-analytic-zero-set} makes its zero set null.  Since
$J_K(\mu_0)$ is finite, all discriminants are defined on a common
neighborhood $U\subset U_0$.  With
\[
 E_K(U):=\{t\in U:\Lambda_j^{\mu_t}=\Lambda_{j+1}^{\mu_t}
               \text{ for some }1\le j\le K\}.
\]
The block isolation gives
\[
 E_K(U)
 \subseteq
 \bigcup_{\alpha\in J_K(\mu_0)}
 \{t\in U:\Delta_\alpha^{\rm eig}(t)=0\}.
\]
Therefore
\[
 \mathcal L^N(E_K(U))
 \le\sum_{\alpha\in J_K(\mu_0)}
 \mathcal L^N\{t\in U:\Delta_\alpha^{\rm eig}(t)=0\}=0.
\]
\end{proof}

\subsection{Borel finite-window splitting events}

Fix $K\ge1$ and a finite-dimensional
$H_N=\operatorname{span}\{e_1,\ldots,e_N\}$ with a basis in $\mathcal Q$.
\begin{equation*}
 \mathcal C_{K,H_N}
 :=\left\{x\in\Omega_{\rm str}:
 \begin{array}{l}
  J_K(\mu_x)\ne\varnothing,\\[2pt]
  \displaystyle
  \forall\alpha\in J_K(\mu_x)\ \exists g\in H_N:\
  \#\Spec\bigl(\Compress_g^\alpha(\mu_x)\bigr)
  =\dim E_{\mu_x}(\alpha)
 \end{array}\right\}.
\end{equation*}

\begin{lemma}
\label{lem:abstract-borel-cluster-event}
The finite-window splitting event $\mathcal C_{K,H_N}$ defined above is
Borel measurable; that is,
\begin{equation*}
 \mathcal C_{K,H_N}\in\mathcal B(\Xspace).
\end{equation*}
\end{lemma}

\begin{proof}
Put $\Lambda_0^{\mu_x}=0$.  For $1\le p<\ell$, let
\begin{equation*}
 \mathcal E_{p,\ell}^{\rm cl}
 :=\Omega_{\rm str}\cap
 \{\Lambda_{p-1}^{\mu_x}<\Lambda_p^{\mu_x}
 =\cdots=\Lambda_\ell^{\mu_x}<\Lambda_{\ell+1}^{\mu_x}\}.
\end{equation*}
For rational $0<a<b$, set
\begin{equation*}
 \mathcal W_{p,\ell}^{a,b}
 :=\mathcal E_{p,\ell}^{\rm cl}\cap
 \{\Lambda_{p-1}^{\mu_x}<a<\Lambda_p^{\mu_x}
 =\Lambda_\ell^{\mu_x}<b<\Lambda_{\ell+1}^{\mu_x}\}.
\end{equation*}

For $c=(c_1,\ldots,c_N)\in\R^N$, put
$g_c=\sum_kc_ke_k\in H_N$ and define
\begin{equation*}
 \Delta_{p,\ell;a,b,c}^{\rm resp}(x)
 :=\begin{cases}
 \Disc\!\left(
 \WinProj_{a,b}^x\Mult_{g_c}\WinProj_{a,b}^x
 |_{\operatorname{Ran}\WinProj_{a,b}^x}\right),
 &x\in\Omega_{\rm str},\quad
  \rank\WinProj_{a,b}^x=\ell-p+1,\\
 0,&\text{otherwise}.
 \end{cases}
\end{equation*}

If $c\in\mathbb Q^N$, then $g_c\in\mathcal Q$.  The rank formula following
Assumption~\ref{ass:spectral-borel} makes the rank condition measurable.  By
Newton identities, the discriminant is a polynomial in the power traces in
\eqref{eq:minimal-borel-interface}; together with the zero-extension
convention, this shows that the map
\[
 x\longmapsto\Delta_{p,\ell;a,b,c}^{\rm resp}(x):
 (\Xspace,\mathcal B(\Xspace))\longrightarrow
 (\mathbb R,\mathcal B(\mathbb R))
\]
is measurable.  Therefore
\begin{equation*}
 \mathcal S_{p,\ell}(H_N)
 :=\bigcup_{\substack{a,b\in\mathbb Q_{>0},\,a<b\\c\in\mathbb Q^N}}
 \left(\mathcal W_{p,\ell}^{a,b}
       \cap\{\Delta_{p,\ell;a,b,c}^{\rm resp}\ne0\}\right)
\end{equation*}
belongs to $\mathcal B(\Xspace)$.  If $x\in\mathcal E_{p,\ell}^{\rm cl}$, choose
rational $a,b$ with
$x\in\mathcal W_{p,\ell}^{a,b}$.  The compression discriminant is then a
polynomial in $c\in\R^N$, and
\begin{align*}
 &\exists g\in H_N:\Compress_g^{\Lambda_p^{\mu_x}}(\mu_x)
   \text{ has simple spectrum}\\
 &\quad\Longleftrightarrow
   \exists c\in\R^N:\Delta_{p,\ell;a,b,c}^{\rm resp}(x)\ne0\\
 &\quad\Longleftrightarrow
   \exists c\in\mathbb Q^N:\Delta_{p,\ell;a,b,c}^{\rm resp}(x)\ne0.
\end{align*}

Indeed, the non-vanishing set of a nonzero polynomial is nonempty and open,
so it contains a rational point.  Consequently,
\begin{align}
 \mathcal C_{K,H_N}
 &=\Omega_{\rm str}
 \cap\left(\bigcup_{j=1}^K
      \{\Lambda_j^{\mu_x}=\Lambda_{j+1}^{\mu_x}\}\right)\notag\\
 &\quad\cap\bigcap_{p=1}^K\bigcap_{\ell=p+1}^{\infty}
   \bigl((\mathcal E_{p,\ell}^{\rm cl})^c
         \cup\mathcal S_{p,\ell}(H_N)\bigr),
 \label{eq:abstract-CKV-borel-representation}
\end{align}
which belongs to $\mathcal B(\Xspace)$.
\end{proof}

\begin{lemma}
\label{lem:abstract-fixed-splitting-event-null}
For every $K\ge1$ and every nonzero
finite-dimensional $H_N$ with a basis in $\mathcal Q$,
\begin{equation*}
 \mathbb P(\mathcal C_{K,H_N})=0.
\end{equation*}
\end{lemma}

\begin{proof}
Choose a basis $e_1,\ldots,e_N\in\mathcal Q$ of $H_N$ and put
$v_t=\sum_kt_ke_k$.  Lemma~\ref{lem:abstract-gaussian-disintegration} gives
\begin{equation}
 x=x^\perp+v_t,
 \qquad
 d\mathbb P(x)=\varphi_{H_N}(t)\,dt\,
 d\mathbb P_{H_N}^\perp(x^\perp),
 \qquad \varphi_{H_N}(t)>0.
 \label{eq:abstract-simplicity-disintegration}
\end{equation}
The canonical fiber section
\begin{equation*}
 \mathcal Z_{x^\perp,H_N}^{\rm can}
 :=\{t:x^\perp+v_t\in\mathcal C_{K,H_N}\}
\end{equation*}
belongs to $\mathcal B(\R^N)$ by
Lemma~\ref{lem:abstract-borel-cluster-event}.

Apply Lemma~\ref{lem:abstract-coherent-fiber} with
$\Omega_0=\Omega_{\rm str}$ from \eqref{eq:abstract-structural-data}.  For
almost every $x^\perp$, it supplies a
full-measure Borel set $T_0(x^\perp)$ and an all-parameter coherent
admissible family $(\mu_t^{\rm coh})_{t\in\R^N}$ agreeing with the canonical
measures on $T_0(x^\perp)$.  Define
\begin{equation*}
 \mathcal Z_{x^\perp,H_N}^{\rm coh}
 :=\left\{t\in\R^N:
 \begin{array}{l}
  \Lambda_j^{\mu_t^{\rm coh}}=\Lambda_{j+1}^{\mu_t^{\rm coh}}
  \text{ for some }1\le j\le K,\\
  \text{for every }\alpha\in J_K(\mu_t^{\rm coh})
  \text{ some }g\in H_N\text{ splits }
  E_{\mu_t^{\rm coh}}(\alpha)\text{ to first order}
 \end{array}\right\}.
\end{equation*}
This set also belongs to $\mathcal B(\R^N)$.  Indeed,
Lemma~\ref{lem:abstract-slice-borel} makes the
ordered eigenvalues continuous.  On every rational isolating window, let
$\widehat\WinProj_{a,b}(t)$ denote the spectral projection transported to a fixed
Hilbert space.  Proposition~\ref{prop:abstract-kato-slice} makes this projection and
the power traces
\begin{equation*}
 t\longmapsto
 \Tr\left[
  (\widehat\WinProj_{a,b}(t)\Mult_q\widehat\WinProj_{a,b}(t))^k
 \right],
 \qquad q\in\mathcal Q,\quad k\ge1,
\end{equation*}
locally real analytic.  Newton identities make the compression discriminant
continuous, so the countable rational-window representation
\eqref{eq:abstract-CKV-borel-representation} proves the claim.

On $T_0(x^\perp)$ the canonical and coherent measures agree, hence
\begin{equation}
 \one_{\mathcal Z_{x^\perp,H_N}^{\rm can}}(t)
 =\one_{\mathcal Z_{x^\perp,H_N}^{\rm coh}}(t)
 \quad\text{for Lebesgue-almost every }t.
 \label{eq:abstract-canonical-coherent-collision-agreement}
\end{equation}
Fix $t_0\in\mathcal Z_{x^\perp,H_N}^{\rm coh}$.  Re-centering gives
\[
 d\mu_{t_0+z}^{\rm coh}
 =e^{\kappa\sum_kz_ke_k}\,d\mu_{t_0}^{\rm coh}.
\]
Lemma~\ref{lem:abstract-local-window-nullity} supplies a neighborhood
$U_{t_0}$ such that
\[
 \mathcal L^N(\mathcal Z_{x^\perp,H_N}^{\rm coh}\cap U_{t_0})=0.
\]
Since $\mathbb R^N$ is second countable (hence every subspace is Lindel\"of),
a countable subfamily $(U_{t_m})_{m\ge1}$ covers the coherent section, and
therefore
\begin{equation*}
 \mathcal L^N(\mathcal Z_{x^\perp,H_N}^{\rm coh})
 \le\sum_{m\ge1}
 \mathcal L^N(\mathcal Z_{x^\perp,H_N}^{\rm coh}\cap U_{t_m})=0.
\end{equation*}
Equation~\eqref{eq:abstract-canonical-coherent-collision-agreement} gives
$\mathcal L^N(\mathcal Z_{x^\perp,H_N}^{\rm can})=0$ for almost every fiber.
Finally, \eqref{eq:abstract-simplicity-disintegration} and Tonelli's theorem
give
\[
 \mathbb P(\mathcal C_{K,H_N})
 =\int_{\Xspace}\int_{\R^N}
 \one_{\mathcal Z_{x^\perp,H_N}^{\rm can}}(t)
 \varphi_{H_N}(t)\,dt\,d\mathbb P_{H_N}^\perp(x^\perp)=0.
\]
\end{proof}

\subsection{Final proof of Theorem~\ref{thm:intro-abstract-response}\textnormal{(i)}}

For $K\ge1$, let
\begin{equation*}
 \mathcal C_K
 :=\Omega_{\rm str}\cap
 \bigcup_{j=1}^K
 \{\Lambda_j^{\mu_x}=\Lambda_{j+1}^{\mu_x}\}.
\end{equation*}
The event $\Omega_{\rm str}$ from \eqref{eq:abstract-structural-data} and the
eigenvalue measurability in Assumption~\ref{ass:spectral-borel} give
$\mathcal C_K\in\mathcal B(\Xspace)$.
Let
\begin{equation*}
 \mathcal G
 :=\left\{\operatorname{span}\{q_1,\ldots,q_N\}:
 N\ge1,\ q_1,\ldots,q_N\in\mathcal Q\setminus\{0\}
 \text{ are linearly independent}\right\}
\end{equation*}
be the countable family of finite-dimensional subspaces generated by
directions in $\mathcal Q$.  For $x\in\mathcal C_K$, finiteness of
$J_K(\mu_x)$ and Theorem~\ref{thm:abstract-simple-splitting} give
some $H_N\in\mathcal G$ with $x\in\mathcal C_{K,H_N}$; the reverse inclusion is
immediate.  Hence
\begin{equation*}
 \mathcal C_K
 =\bigcup_{H_N\in\mathcal G}\mathcal C_{K,H_N}.
\end{equation*}
Hence Lemma~\ref{lem:abstract-fixed-splitting-event-null} gives
\begin{equation*}
 \mathbb P(\mathcal C_K)
 \le\sum_{H_N\in\mathcal G}\mathbb P(\mathcal C_{K,H_N})=0.
\end{equation*}
On $\Omega_{\rm str}$, nonsimplicity of the positive spectrum is exactly the
countable union of the finite-window collision events:
\[
 \{x\in\Omega_{\rm str}:
   \text{the positive spectrum of }A_{\mu_x}\text{ is not simple}\}
 =\bigcup_{K\ge1}\mathcal C_K.
\]
This union is $\mathbb P$-null, while
$\mathbb P(\Omega_{\rm str})=1$ by \eqref{eq:abstract-structural-data}.
Thus \eqref{eq:intro-abstract-simple-spectrum} follows.

\section{Joint eigenvalue densities by response submersion}
\label{sec:density}

\subsection{Borel events and fiberwise nullity}

We first encode response transversality by countably many measurable
submersion events and prove their fiberwise nullity.

Fix $M\ge1$, and let $\nu=(n_1,\ldots,n_M)$ with
$1\le n_1<\cdots<n_M$, and write
\begin{equation*}
 \Lambda_{\nu}(x)
 :=(\Lambda_{n_1}^{\mu_x},\ldots,\Lambda_{n_M}^{\mu_x}).
\end{equation*}
Whenever the selected eigenvalues are simple, use the response functionals
$\ell_{n_i}^{\mu_x}$ from
\eqref{eq:abstract-response-functionals-Q}.

\begin{lemma}
\label{lem:abstract-borel-submersion}
Let $g=(g_1,\ldots,g_M)\in\mathcal Q^M$ and define
\begin{equation*}
 \Omega_{\nu,g}^{\rm sub}
 :=
 \left\{
 x\in\Omega_{\rm str}:
 \begin{array}{l}
 \Lambda_{n_1}^{\mu_x},\ldots,\Lambda_{n_M}^{\mu_x}
 \text{ are simple},\\[1mm]
 \det(\ell_{n_i}^{\mu_x}(g_j))_{i,j=1}^M\ne0
 \end{array}
 \right\}
\end{equation*}
Then $\Omega_{\nu,g}^{\rm sub}\in\mathcal B(\Xspace)$.
\end{lemma}

\begin{proof}
Put $\Lambda_0^{\mu_x}=0$.  For
$a,b\in\mathbb Q^M$ with $0<a_i<b_i$, set
\[
 \mathcal W_{a,b}^{\nu}
 :=\Omega_{\rm str}\cap
 \bigcap_{i=1}^M
 \left\{
 \Lambda_{n_i-1}^{\mu_x}<a_i<\Lambda_{n_i}^{\mu_x}
 <b_i<\Lambda_{n_i+1}^{\mu_x}
 \right\}.
\]
On this event, $\WinProj_{a_i,b_i}^x$ has rank one and
\begin{equation*}
 \ell_{n_i}^{\mu_x}(g_j)
 =
 -\kappa\Lambda_{n_i}^{\mu_x}
 \Tr\!\left(
 \WinProj_{a_i,b_i}^x\Mult_{g_j}\WinProj_{a_i,b_i}^x
 \right).
\end{equation*}
Consequently,
\begin{align*}
 \Omega_{\nu,g}^{\rm sub}
 &=
 \bigcup_{\substack{a,b\in\mathbb Q^M\\0<a_i<b_i}}
 \left[
 \mathcal W_{a,b}^{\nu}
 \cap
 \left\{
 \det\!\left(
 -\kappa\Lambda_{n_i}^{\mu_x}
 \Tr(
 \WinProj_{a_i,b_i}^x\Mult_{g_j}\WinProj_{a_i,b_i}^x)
 \right)_{i,j=1}^M
 \ne0
 \right\}
 \right].
\end{align*}
The event $\Omega_{\rm str}$ from \eqref{eq:abstract-structural-data} and
Assumption~\ref{ass:spectral-borel} make every set in this countable union
Borel in $\Xspace$.
\end{proof}

Define
\begin{equation*}
 \Omega_{\nu}^{\rm RT}
 :=
 \bigcup_{g\in\mathcal Q^M}
 \Omega_{\nu,g}^{\rm sub}.
\end{equation*}
Since $\mathcal Q$ is countable,
Lemma~\ref{lem:abstract-borel-submersion} gives
$\Omega_{\nu}^{\rm RT}\in\mathcal B(\Xspace)$.
On the selected-simple locus, Lemma~\ref{lem:abstract-countable-submersion}
and finite-dimensional duality give
\begin{equation}
 x\in\Omega_{\nu}^{\rm RT}
 \quad\Longleftrightarrow\quad
 (\ell_{n_1}^{\mu_x},\ldots,\ell_{n_M}^{\mu_x})
 \text{ satisfies \eqref{eq:abstract-RT}}.
 \label{eq:abstract-RT-event-equivalence}
\end{equation}

\begin{lemma}
\label{lem:abstract-fixed-submersion-nullity}
Let $g\in\mathcal Q^M$.  If $E\in\mathcal B(\mathbb R^M)$ and
$\mathcal L^M(E)=0$, then
\begin{equation*}
 \Omega_{\nu,g}^{\rm sub}
 \cap\Lambda_{\nu}^{-1}(E)
 \in\mathcal B(\Xspace),
 \qquad
 \mathbb P\left(
 \Omega_{\nu,g}^{\rm sub}
 \cap\Lambda_{\nu}^{-1}(E)
 \right)=0.
\end{equation*}
\end{lemma}

\begin{proof}
If $\sum_{j=1}^Mc_jg_j=0$ for some $c\ne0$, then
\[
 \sum_{j=1}^Mc_j\ell_{n_i}^{\mu_x}(g_j)=0,
 \qquad 1\le i\le M,
\]
so $\Omega_{\nu,g}^{\rm sub}=\varnothing$.  Otherwise
$g_1,\ldots,g_M$ are linearly independent; set
\[
 H_M=\operatorname{span}\{g_1,\ldots,g_M\},
 \qquad
 v_t=\sum_{j=1}^Mt_jg_j.
\]
Gaussian disintegration gives
\begin{equation}
 x=x^\perp+v_t,
 \qquad
 d\mathbb P(x)=\varphi_{H_M}(t)\,dt\,
 d\mathbb P_{H_M}^\perp(x^\perp),
 \qquad \varphi_{H_M}(t)>0.
 \label{eq:abstract-density-disintegration}
\end{equation}
For the canonical realization define
\begin{equation*}
 \mathcal Z_{x^\perp,g,E}^{\rm can}
 :=
 \left\{
 t\in\mathbb R^M:
 x^\perp+v_t\in\Omega_{\nu,g}^{\rm sub},\quad
 \Lambda_{\nu}(x^\perp+v_t)\in E
 \right\}.
\end{equation*}
The affine map $t\mapsto x^\perp+v_t$ is continuous, so this set belongs to
$\mathcal B(\mathbb R^M)$ by Lemma~\ref{lem:abstract-borel-submersion},
Assumption~\ref{ass:spectral-borel}, and $E\in\mathcal B(\mathbb R^M)$.

Apply Lemma~\ref{lem:abstract-coherent-fiber} with
$\Omega_0=\Omega_{\rm str}$ from \eqref{eq:abstract-structural-data}.  For
almost every fixed fiber it gives a
full-measure Borel set $T_0(x^\perp)$ and an admissible coherent
family $(\mu_t^{\rm coh})_{t\in\mathbb R^M}$ such that
\begin{equation*}
 d\mu_t^{\rm coh}=e^{\kappa(v_t-v_s)}d\mu_s^{\rm coh},
 \qquad s,t\in\mathbb R^M,
\end{equation*}
and $\mu_t^{\rm coh}=\mu_{x^\perp+v_t}$ for
$t\in T_0(x^\perp)$.  Write
\[
 \Lambda_n^{\rm coh}(t):=\Lambda_n^{\mu_t^{\rm coh}},
 \qquad
 \Lambda_{\nu}^{\rm coh}(t)
 :=
 (\Lambda_{n_1}^{\rm coh}(t),\ldots,
  \Lambda_{n_M}^{\rm coh}(t))
\]
for the coherent ordered eigenvalue vector, and set
\begin{equation*}
 \mathcal Z_{x^\perp,g,E}^{\rm coh}
 :=
 \left\{
 t\in T_0(x^\perp):
 \begin{array}{l}
 \Lambda_{n_1}^{\rm coh}(t),\ldots,
 \Lambda_{n_M}^{\rm coh}(t)\text{ are simple},\\[1mm]
 \det(\ell_{n_i}^{\mu_t^{\rm coh}}(g_j))_{i,j=1}^M\ne0,\\[1mm]
 \Lambda_{\nu}^{\rm coh}(t)\in E
 \end{array}
 \right\}.
\end{equation*}
Since $T_0(x^\perp)\in\mathcal B(\mathbb R^M)$,
Lemma~\ref{lem:abstract-slice-borel} gives both the response-determinant
measurability and continuity of $\Lambda_\nu^{\rm coh}$.  Hence this set
belongs to $\mathcal B(\mathbb R^M)$.  On every such
fiber, canonical--coherent agreement and the definition of
$T_0(x^\perp)$ give
\begin{equation}
 \one_{\mathcal Z_{x^\perp,g,E}^{\rm can}}(t)
 =
 \one_{\mathcal Z_{x^\perp,g,E}^{\rm coh}}(t),
 \qquad \text{for Lebesgue-a.e. }t\in\mathbb R^M.
 \label{eq:abstract-density-canonical-coherent-agreement}
\end{equation}

Fix $t_0\in\mathcal Z_{x^\perp,g,E}^{\rm coh}$.  On a neighborhood
$U_{t_0}$, the selected eigenvalues remain simple and
$\Lambda_{\nu}^{\rm coh}$ is real analytic.  After
re-centering the coherent family at $t_0$,
Corollary~\ref{cor:abstract-simple-response} gives
\begin{equation*}
 D\Lambda_{\nu}^{\rm coh}(t_0)
 =
 \bigl(\ell_{n_i}^{\mu_{t_0}^{\rm coh}}(g_j)\bigr)_{i,j=1}^M,
 \qquad
 \det D\Lambda_{\nu}^{\rm coh}(t_0)\ne0.
\end{equation*}

The inverse function theorem therefore gives open sets
$\mathcal O_{t_0}\subset U_{t_0}$ and $\mathcal O'_{t_0}$ on which
\[
 \Lambda_{\nu}^{\rm coh}:
 \mathcal O_{t_0}\longrightarrow\mathcal O'_{t_0}
\]
is a $C^1$ diffeomorphism.  Choose a ball $B_{t_0}$ with
$t_0\in B_{t_0}$ and
$\overline B_{t_0}\subset\mathcal O_{t_0}$.  The change-of-variables
formula yields
\begin{align*}
 \mathcal L^M\!\left(
 \mathcal Z_{x^\perp,g,E}^{\rm coh}
 \cap B_{t_0}\right)
 &\le
 \mathcal L^M\!\left(
 B_{t_0}\cap
 (\Lambda_{\nu}^{\rm coh})^{-1}(E)\right)
 \notag\\
 &=
 \int_{E\cap
 \Lambda_{\nu}^{\rm coh}(B_{t_0})}
 \left|
 \det D(\Lambda_{\nu}^{\rm coh})^{-1}(y)
 \right|\,dy
 =0.
\end{align*}

Since $\mathbb R^M$ is second countable (hence every subspace is Lindel\"of),
the balls $\{B_{t_0}:t_0\in
\mathcal Z_{x^\perp,g,E}^{\rm coh}\}$ admit a countable subcover, so
\[
 \mathcal Z_{x^\perp,g,E}^{\rm coh}
 \subseteq\bigcup_{m\ge1}B_{t_m},
 \qquad
 \mathcal L^M(\mathcal Z_{x^\perp,g,E}^{\rm coh})
 \le
 \sum_{m\ge1}
 \mathcal L^M(
 \mathcal Z_{x^\perp,g,E}^{\rm coh}\cap B_{t_m})
 =0.
\]

Equations~\eqref{eq:abstract-density-canonical-coherent-agreement} and
\eqref{eq:abstract-density-disintegration} now give
\begin{align*}
 &\mathbb P\left(
 \Omega_{\nu,g}^{\rm sub}
 \cap\{\Lambda_{\nu}\in E\}\right)\\
 &\quad=
 \int_{\Xspace}\int_{\mathbb R^M}
 \one_{\mathcal Z_{x^\perp,g,E}^{\rm can}}(t)
 \varphi_{H_M}(t)\,dt\,
 d\mathbb P_{H_M}^\perp(x^\perp)
 =0.
\end{align*}
\end{proof}

\subsection{Final proof of Theorem~\ref{thm:intro-abstract-response}\textnormal{(ii)}}

The hypotheses in Theorem~\ref{thm:intro-abstract-response}\textnormal{(ii)}
and equivalence~\eqref{eq:abstract-RT-event-equivalence} give
\begin{equation*}
 \mathbb P(\Omega_{\nu}^{\rm RT})=1.
\end{equation*}
If $E\in\mathcal B(\mathbb R^M)$ and $\mathcal L^M(E)=0$, then
Lemma~\ref{lem:abstract-fixed-submersion-nullity} gives
\begin{align*}
 \mathbb P\bigl(\Lambda_{\nu}^{-1}(E)\bigr)
 &\le
 \mathbb P\bigl((\Omega_{\nu}^{\rm RT})^c\bigr)\\
 &\quad+
 \sum_{g\in\mathcal Q^M}
 \mathbb P\left(
 \Omega_{\nu,g}^{\rm sub}
 \cap\Lambda_{\nu}^{-1}(E)
 \right)
 =0.
\end{align*}
Hence $(\Lambda_{\nu})_\#\mathbb P\ll\mathcal L^M$, which is
\eqref{eq:intro-abstract-density}.

The next corollary is a model-ready sufficient criterion for
Definition~\ref{def:abstract-response-transversality}.  Its square-identity
event is a conclusion to be proved in each model, not an additional standing
assumption of the abstract theory.

\begin{corollary}
\label{cor:abstract-density-square-criterion}
Let $\Omega_{\rm sq}\in\mathcal B(\Xspace)$ be an event such that
\[
 \Omega_{\rm sq}\subset\Omega_{\rm str},
 \qquad
 \mathbb P(\Omega_{\rm sq})=1,
\]
and such that the hypotheses of
Theorem~\ref{thm:abstract-square-transversality} hold for every finite family
of normalized eigenpairs of $A_{\mu_x}$ with pairwise distinct positive
eigenvalues.  Then, for every $M\ge1$ and $1\le n_1<\cdots<n_M$,
\[
 (\Lambda_{n_1}^{\mu_x},\ldots,\Lambda_{n_M}^{\mu_x})_\#\mathbb P
 \ll\mathcal L^M.
\]
\end{corollary}

\begin{proof}
Fix $M\ge1$ and $1\le n_1<\cdots<n_M$, put
$\nu=(n_1,\ldots,n_M)$, and set
\[
 \Omega_{\rm simp}
 :=\{x\in\Omega_{\rm str}:
       \Lambda_1^{\mu_x}<\Lambda_2^{\mu_x}<\cdots\}.
\]
Assumption~\ref{ass:spectral-borel} and
Theorem~\ref{thm:intro-abstract-response}\textnormal{(i)} give
$\Omega_{\rm simp}\in\mathcal B(\Xspace)$ and
$\mathbb P(\Omega_{\rm simp})=1$.  Corollary~\ref{cor:square-implies-RT} and
\eqref{eq:abstract-RT-event-equivalence} yield
\[
 \Omega_{\rm simp}\cap\Omega_{\rm sq}
 \subseteq\Omega_{\nu}^{\rm RT},
 \qquad
 \mathbb P(\Omega_{\nu}^{\rm RT})=1.
\]
Theorem~\ref{thm:intro-abstract-response}\textnormal{(ii)} now applies.
\end{proof}

\section{Dirichlet Liouville Brownian motion}
\label{sec:lbm-realization}

This section verifies the abstract framework for Dirichlet LBM and records its
pathwise spectral response.

Throughout this section,
\begin{equation}\label{eq:standing-domain-range}
 D\subset\R^2\ \text{bounded, open, and connected},
 \qquad \gamma\in(0,2).
\end{equation}
We write $\mathcal L^2$ for planar Lebesgue measure and use $z,w$ for points
of $D$; differential notation such as $dz$ always means integration against
$\mathcal L^2$.

We use the normalization
\begin{equation}\label{eq:standing-dirichlet-laplacian}
 \Lzero:=-\frac1{2\pi}\Delta_D,
 \qquad
 \Lzero\varphi_k^{(0)}=\lambda_k^{(0)}\varphi_k^{(0)},
 \qquad
 0<\lambda_1^{(0)}\le\lambda_2^{(0)}\le\cdots,
\end{equation}
where $(\varphi_k^{(0)})_{k\ge1}$ is a real orthonormal basis of
$L^2(D,\mathcal L^2)$.  We take
$\Hc=H_0^1(D)$ and $\Hreg=C_c^\infty(D)$.
The Cameron--Martin inner product is the Dirichlet energy,
\begin{equation}\label{eq:standing-energy-normalization}
 \langle u,v\rangle_\Hc=\E(u,v)
 :=\frac1{2\pi}\int_D\nabla u\cdot\nabla v\,dz.
\end{equation}
We choose a symmetric, jointly measurable version of the Dirichlet Green
kernel, normalized by
\begin{equation}\label{eq:standing-green-normalization}
 \int_DG_D(z,w)\Lzero\zeta(z)\,dz=\zeta(w),
 \qquad
 -\Delta_zG_D(z,w)=2\pi\delta_w
 \quad\text{in }\mathcal D'(D).
\end{equation}
Here $w\in D$ and $\zeta\in C_c^\infty(D)$.  The boundary condition is
variational, or equivalently quasi-everywhere; it is not imposed pointwise on
$\partial D$.  Notice also that the logarithmic singularity prevents
$G_D(\cdot,w)$ from belonging to $H_0^1(D)$.

No boundary regularity of $D$ is assumed.  If a ball $B$
contains $\overline D$, zero extension and Rellich compactness give
\[
 H_0^1(D)\hookrightarrow H_0^1(B)\hookrightarrow L^2(B),
 \qquad H_0^1(D)\hookrightarrow L^2(D,\mathcal L^2)\ \text{compactly}.
\]
Thus $\Lzero$ has compact resolvent.  Dirichlet domain monotonicity also gives
\[
 \lambda_k^{(0)}(D)\ge \lambda_k^{(0)}(B),\qquad k\ge1,
\]
and Weyl's law on $B$ makes the Gaussian series below convergent for every
$s_0>0$.
Fix $s_0>0$.  Here $H^{-s_0}(D)$ denotes the spectral Hilbert space defined by
the following norm:
\begin{equation}\label{eq:standing-abstract-wiener-space}
 \Xspace=H^{-s_0}(D),
 \qquad
 \|x\|_{-s_0}^2
 :=\sum_{k\ge1}(1+\lambda_k^{(0)})^{-s_0}
       |\langle x,\varphi_k^{(0)}\rangle|^2.
\end{equation}
On this space the zero-boundary GFF is the convergent series
\begin{equation}\label{eq:standing-gff-series}
 h=\sum_{k\ge1}\frac{\xi_k}{\sqrt{\lambda_k^{(0)}}}\varphi_k^{(0)},
 \qquad \xi_k\stackrel{\rm iid}{\sim}N(0,1),
 \qquad
 \mathbb E\|h\|_{-s_0}^2
 =\sum_{k\ge1}\frac{(1+\lambda_k^{(0)})^{-s_0}}{\lambda_k^{(0)}}<\infty.
\end{equation}
Let $\mathbb P$ denote its law and use
$(\Xspace,\mathcal B(\Xspace)^{\mathbb P},\mathbb P)$ as the canonical
probability space, with coordinate GFF $h(x)=x$.
Thus $(\Xspace,\Hc,\mathbb P)$ is the abstract Wiener triple used above; throughout
this section $x$ is a deterministic environment and $h(x)=x$ the coordinate
GFF.

\pagebreak[3]
In the notation of Section~\ref{sec:abstract-setting},
$\State=D$, $\kappa=\gamma$, and $\rho=0$; the objects
$\mu_x=M_x$ and $A_{\mu_x}=A_x$ are constructed below.  Perturbations are first taken in
$\Hreg=C_c^\infty(D)$; the extension to $\Hc$ is treated in
Appendix~\ref{app:lbm-finite-energy-response}.

Two conventions will be used repeatedly.  For
$u\in H_0^1(D)$, $\widetilde u$ denotes a fixed quasi-continuous version.
Also, whenever a completed full-probability event $\Omega$ occurs, we choose a
Borel subset $\Omega^{\rm Bor}$ such that
\begin{equation}\label{eq:borel-full-core-convention}
 \Omega^{\rm Bor}\in\mathcal B(\Xspace),\qquad
 \Omega^{\rm Bor}\subseteq\Omega,\qquad
 \mathbb P(\Omega^{\rm Bor})=1.
\end{equation}
Such a subset exists by the definition of the completed probability space.

\subsection{Coherent Gaussian multiplicative chaos under Cameron--Martin shifts}
\label{sec:lbm-shifts}

The measure entering the Liouville time change cannot be chosen only as an
almost-sure equivalence class.  We obtain the spectral measurability used in
Gaussian slicing from a measurable map $x\mapsto M_x$, while coherent analytic
perturbation requires the pointwise identity
\[
 M_{x+g}=e^{\gamma g}M_x.
\]
We construct such a version and then place its trace-form properties on one
measurable full-probability event.

Write $\Mc_{\rm fin}^+(D)$ for the finite positive Radon measures on $D$,
equipped with the narrow topology and its Borel $\sigma$-field, and fix a complete
compatible metric $d_{\rm w}$.  We first regularize distributions by the killed heat semigroup.

\begin{lemma}
\label{lem:lbm-borel-heat-smoothing}
For every $\eps>0$,
\[
 Q_\eps:=e^{\eps\Delta_D}
 =e^{-2\pi\eps\Lzero}
 \in\mathcal L(\Xspace,C_b(D)).
\]
The map
\[
 (x,z)\longmapsto Q_\eps x(z):
 (\Xspace\times D,\mathcal B(\Xspace)\otimes\mathcal B(D))
 \longrightarrow(\mathbb R,\mathcal B(\mathbb R))
\]
is measurable.
\end{lemma}

\begin{proof}
Write $q_t^D$ for the kernel of $e^{-t\Lzero}$.  The spectral representative is
\begin{equation}
 Q_\eps x(z)
 =\sum_{k\ge1}e^{-2\pi\eps\lambda_k^{(0)}}
 \langle x,\varphi_k^{(0)}\rangle\varphi_k^{(0)}(z).
 \label{eq:lbm-heat-smoothing-series}
\end{equation}
For $x\in\Xspace$, Cauchy--Schwarz gives
\begin{align*}
 |Q_\eps x(z)|^2
 &\le \|x\|_{-s_0}^2
 \sum_{k\ge1}e^{-4\pi\eps\lambda_k^{(0)}}
 (1+\lambda_k^{(0)})^{s_0}|\varphi_k^{(0)}(z)|^2\\
 &\le C_{\eps,s_0}\|x\|_{-s_0}^2
 \sum_{k\ge1}e^{-2\pi\eps\lambda_k^{(0)}}
 |\varphi_k^{(0)}(z)|^2\\
 &=C_{\eps,s_0}\|x\|_{-s_0}^2q_{2\pi\eps}^D(z,z)
 \le C'_{\eps,s_0}\|x\|_{-s_0}^2.
\end{align*}
The last inequality uses domination of the killed heat kernel by the planar
heat kernel.
If $R_N^\eps x$ denotes the tail after the $N$th term, the same estimate gives
\[
 \sup_{z\in D}|R_N^\eps x(z)|^2
 \le C_{\eps,s_0}\|x\|_{-s_0}^2
 e^{-\pi\eps\lambda_{N+1}^{(0)}}
 \sup_{z\in D}q_{2\pi\eps}^D(z,z)
 \xrightarrow[N\to\infty]{}0.
\]
Thus the series converges uniformly in $z$.  Each summand is continuous in
$D$, hence
$Q_\eps x\in C_b(D)$ and
\[
 \|Q_\eps x\|_\infty
 \le C_{\eps,s_0}\|x\|_{-s_0}.
\]
Thus $Q_\eps\in\mathcal L(\Xspace,C_b(D))$.  Since evaluation is
continuous on $C_b(D)\times D$, the displayed map is measurable.
\end{proof}

Put $h_\eps=Q_\eps h$.  For $x\in\Xspace$, define
\begin{equation}\label{eq:regularized-gmc}
 M_{x,\eps}(dz)
 =\exp\!\left(\gamma Q_\eps x(z)
 -\frac{\gamma^2}{2}\mathbb E[h_\eps(z)^2]\right)dz.
\end{equation}
For a countable convergence-determining family
$(\varphi_\ell)_{\ell\ge1}\subset C_b(D)$, every map
\[
 x\longmapsto \int_D\varphi_\ell\,dM_{x,\eps}:
 (\Xspace,\mathcal B(\Xspace))\longrightarrow
 (\mathbb R,\mathcal B(\mathbb R))
\]
is measurable.  Hence
\[
 x\longmapsto M_{x,\eps}:
 (\Xspace,\mathcal B(\Xspace))\longrightarrow
 (\Mc_{\rm fin}^+(D),\mathcal B(\Mc_{\rm fin}^+(D)))
\]
is Borel measurable with respect to the narrow Borel structure by
Lemma~\ref{lem:lbm-borel-heat-smoothing}.

For later shifts, let $g\in C_c^\infty(D)$.  The semigroup identity and the
$L^\infty$ contraction property give
\begin{equation}\label{eq:test-function-killed-semigroup-approximation}
 Q_tg-g=\int_0^tQ_s\Delta_Dg\,ds,
 \qquad
 \|Q_tg-g\|_\infty
 \le t\|\Delta_Dg\|_\infty.
\end{equation}
The first equality holds initially in $L^2(D,\mathcal L^2)$ and its right-hand side
provides the displayed bounded representative.  In particular, no continuity
of the killed semigroup at irregular boundary points is being assumed.

\subsubsection{A measurable GMC version and coherent Cameron--Martin lines}

The existence and approximation independence of subcritical GMC are standard
\cite[Corollary~18 and Theorem~25]{ShamovGMC}.  For the spectral construction
below, we choose a representative that depends measurably on every
$x\in\Xspace$ and obeys the Cameron--Martin cocycle whenever both sides are
defined by the chosen deterministic approximation.  The next lemma
constructs this map; Theorem~\ref{thm:gmc-cm-shift} proves the cocycle.  A
related measurable construction on compact manifolds appears in
\cite[Theorem~4.1]{DelloSchiavoHerryKopferSturm}.

\begin{lemma}
\label{lem:canonical-borel-gmc}
There exist a deterministic sequence $\eps_j\downarrow0$ and a set
\[
 \Omega_{\rm GMC}\in\mathcal B(\Xspace),
 \qquad \mathbb P(\Omega_{\rm GMC})=1,
\]
together with a measurable map
\begin{equation}\label{eq:canonical-borel-gmc-map}
 M^{\rm can}:(\Xspace,\mathcal B(\Xspace))
 \longrightarrow
 (\Mc_{\mathrm{fin}}^+(D),\mathcal B(\Mc_{\mathrm{fin}}^+(D)))
\end{equation}
such that
\begin{equation}\label{eq:canonical-borel-gmc-limit}
 M_{x,\eps_j}\Longrightarrow M^{\rm can}(x),
 \qquad x\in\Omega_{\rm GMC}.
\end{equation}
Here $\Longrightarrow$ denotes narrow convergence of finite measures.
Moreover, $M^{\rm can}(h)$ agrees almost surely with the subcritical GMC
measure.
\end{lemma}

\begin{proof}
Let $q_t^U$ denote the heat kernel of $e^{t\Delta_U/(2\pi)}$.  If a disk $B$ contains
$\overline D$, domain monotonicity and the semigroup property give
\begin{equation}\label{eq:canonical-heat-covariance-comparison}
 K_{\eps,D}(z,w)
 :=\Cov(Q_\eps h(z),Q_\eps h(w))
 =\int_{4\pi\eps}^{\infty}q_t^D(z,w)\,dt
 \le\int_{4\pi\eps}^{\infty}q_t^B(z,w)\,dt
 =:K_{\eps,B}(z,w).
\end{equation}
By monotone convergence,
$K_{\eps,D}(z,w)\uparrow G_D(z,w)$ for almost every $(z,w)\in D^2$.  The disk
heat-kernel bound also yields
\begin{equation}\label{eq:canonical-heat-covariance-log-majorant}
 K_{\eps,D}(z,w)
 \le C_B+\log^+\frac1{|z-w|\vee\sqrt\eps},
 \qquad z,w\in D,\quad0<\eps\le1.
\end{equation}
For every $g\in\Hc$, strong continuity of the killed heat semigroup gives
\[
 Q_\eps g\longrightarrow g
 \quad\text{in }L^2(D,\mathcal L^2).
\]
Let $K_\delta^{\rm ex}$ be the covariance of the standard exact-scale
martingale cutoff on an auxiliary square $S_0\supset\overline D$.  The explicit exact-scale
construction \cite[Proposition~2.15]{RhodesVargasGMCReview} and
\eqref{eq:canonical-heat-covariance-log-majorant} give
$\delta(\eps)\downarrow0$ and $C_0<\infty$ such that
\begin{equation}
 K_{\eps,D}(z,w)
 \le K_{\delta(\eps)}^{\rm ex}(z,w)+C_0,
 \qquad z,w\in D,\quad 0<\eps\le1.
 \label{eq:canonical-exact-scale-covariance-comparison}
\end{equation}
If $M_{\gamma,\delta}^{\rm ex}$ and $M_\gamma^{\rm ex}$ denote the cutoff and
limiting exact-scale chaoses on $S_0$, then for every $p\in(1,4/\gamma^2)$ the
positive-moment theorem makes the martingale $L^p$-closed.  For its natural
cutoff filtration, the cutoff mass is therefore the conditional
expectation of the limiting mass given the field down to scale $\delta$, and
conditional Jensen gives
\[
 \mathbb E[(M_{\gamma,\delta}^{\rm ex}(S_0))^p]
 \le \mathbb E[(M_\gamma^{\rm ex}(S_0))^p].
\]
Kahane's inequality \cite[Theorem~2.1]{RhodesVargasGMCReview}, the positive
moment bound \cite[Theorem~2.11]{RhodesVargasGMCReview}, and the preceding
conditional-Jensen estimate give
\begin{equation}\label{eq:canonical-gmc-uniform-integrability}
 \begin{aligned}
 \sup_{0<\eps\le1}\mathbb E[M_{h,\eps}(D)^p]
 &\le e^{\frac{\gamma^2}{2}C_0p(p-1)}
       \sup_{\delta>0}\mathbb E[(M_{\gamma,\delta}^{\rm ex}(S_0))^p]\\
 &\le e^{\frac{\gamma^2}{2}C_0p(p-1)}
       \mathbb E[(M_\gamma^{\rm ex}(S_0))^p]<\infty.
 \end{aligned}
\end{equation}
The approximation theorem \cite[Theorem~25]{ShamovGMC} now applies: the
covariances converge to $G_D$ in measure, they satisfy
\eqref{eq:canonical-heat-covariance-log-majorant}, the Cameron--Martin shifts
converge in $L^2$, and the total masses are uniformly integrable by
\eqref{eq:canonical-gmc-uniform-integrability}.  Hence $M_{h,\eps}$ converges
in probability, for $d_{\rm w}$, to the subcritical GMC.

Choose a deterministic sequence $\eps_j\downarrow0$ so that
\[
 \mathbb P\!\left(
 d_{\mathrm w}(M_{h,\eps_{j+1}},M_{h,\eps_j})>2^{-j}
 \right)\le 2^{-j}.
\]
By Borel--Cantelli, almost surely and for all sufficiently large $j<k$,
\[
 d_{\mathrm w}(M_{h,\eps_k},M_{h,\eps_j})
 \le\sum_{\ell=j}^{k-1}2^{-\ell}
 \le2^{1-j}.
\]
Thus the convergence event satisfies
\[
 \Omega_{\rm GMC}
 :=\bigcap_{m\ge1}\bigcup_{N\ge1}\bigcap_{j,k\ge N}
 \left\{x\in\Xspace:
 d_{\mathrm w}(M_{x,\eps_j},M_{x,\eps_k})<m^{-1}\right\}
 \in\mathcal B(\Xspace),
 \qquad \mathbb P(\Omega_{\rm GMC})=1.
\]
Completeness of $d_{\mathrm w}$ permits the definition
\[
 M^{\rm can}(x)
 :=\begin{cases}
 \displaystyle\lim_{j\to\infty}M_{x,\eps_j},&x\in\Omega_{\rm GMC},\\[1ex]
 0,&x\notin\Omega_{\rm GMC}.
 \end{cases}
\]
The resulting map
\[
 M^{\rm can}:(\Xspace,\mathcal B(\Xspace))\longrightarrow
 (\Mc_{\rm fin}^+(D),\mathcal B(\Mc_{\rm fin}^+(D)))
\]
is Borel measurable with respect to the narrow Borel structure, being a pointwise
limit on $\Omega_{\rm GMC}$ and constant on its complement.  Since the same subsequence
converges in probability to the subcritical GMC, uniqueness identifies the
two limits almost surely.
\end{proof}

From now on set
\[
 M_x:=M^{\rm can}(x),
 \qquad M_h:=M^{\rm can}(h).
\]
The superscript ``${\rm can}$'' is retained only when comparison with another
GMC construction is needed.

\begin{theorem}
\label{thm:gmc-cm-shift}
For every $g\in\Hreg$,
\begin{equation}\label{eq:pointwise-gmc-orbit-compatibility}
 M_{x+g}=e^{\gamma g}M_x,
 \qquad x,x+g\in\Omega_{\rm GMC}.
\end{equation}
In particular, for each fixed $g\in\Hreg$,
\begin{equation}\label{eq:gmc-cm-shift-g}
        M_{h+g}=e^{\gamma g}M_h
        \qquad\text{almost surely}.
\end{equation}
\end{theorem}

\begin{proof}
For the Borel heat-semigroup version from
Lemma~\ref{lem:canonical-borel-gmc}, the fixed-shift identity follows directly
from
\[
 \|Q_\eps g-g\|_\infty
 \le\eps\|\Delta_Dg\|_\infty\longrightarrow0,
 \qquad
 M_{x+g,\eps_j}=e^{\gamma Q_{\eps_j}g}M_{x,\eps_j}.
\]
Indeed, for every $\psi\in C_b(D)$,
\[
 \int_D\psi\,dM_{x+g}
 =\lim_{j\to\infty}\int_D
   \psi e^{\gamma Q_{\eps_j}g}\,dM_{x,\eps_j}
 =\int_D\psi e^{\gamma g}\,dM_x.
\]
This proves \eqref{eq:pointwise-gmc-orbit-compatibility}.
Cameron--Martin equivalence gives \eqref{eq:gmc-cm-shift-g}.
\end{proof}

To obtain the full-quasi-support input for the trace form, we compare the
canonical GMC on $D$ with its counterpart on a smooth interior subdomain.
Fix a smooth connected $U\Subset D$.  By the domain Markov property, the
zero-boundary GFFs on $D$ and $U$ can be coupled so that
\begin{equation}\label{eq:gff-domain-markov-local-decomposition}
 h_D=h_U+\chi_{D,U}
 \qquad\text{in }U,
\end{equation}
where $h_U$ and $\chi_{D,U}$ are independent and $\chi_{D,U}$ is almost surely
harmonic.  On compact subsets of $U$, this harmonic correction will appear as
a bounded positive density between the two chaos measures.

\begin{lemma}
\label{lem:interior-dirichlet-gmc-comparison}
Let $M_{h_D}^{\rm can}$ and $M_{h_U}^{\rm can}$ be the canonical GMC measures
obtained from the construction of Lemma~\ref{lem:canonical-borel-gmc} on $D$
and $U$, respectively.  Under the coupling above, for every fixed $V\Subset U$,
almost surely,
\begin{equation}\label{eq:interior-dirichlet-gmc-comparison}
 M_{h_D}^{\rm can}|_V
 =w_{D,U}\,M_{h_U}^{\rm can}|_V,
 \qquad
 w_{D,U}(z)
 :=\exp\!\left(
   \gamma\chi_{D,U}(z)
   -\frac{\gamma^2}{2}\mathbb E[\chi_{D,U}(z)^2]
  \right),
\end{equation}
and
\begin{equation}\label{eq:interior-dirichlet-gmc-bounded-density}
 w_{D,U},\ w_{D,U}^{-1}\in C_b(V).
\end{equation}
In particular, the two measures in
\eqref{eq:interior-dirichlet-gmc-comparison} have the same null sets on $V$.
\end{lemma}

\begin{proof}
The orthogonal decomposition
\[
 H_0^1(D)
 =H_0^1(U)\mathbin\oplus
   \{v\in H_0^1(D):\E(v,\phi)=0\text{ for every }\phi\in H_0^1(U)\}
\]
gives the domain Markov decomposition
\eqref{eq:gff-domain-markov-local-decomposition}
\cite[Section~2.6]{SheffieldGFF}.  Its harmonic part satisfies
\[
 \mathbb E[\chi_{D,U}(z)\chi_{D,U}(w)]
 =G_D(z,w)-G_U(z,w),
\]
so interior elliptic regularity gives a smooth version of
$\chi_{D,U}$ and a continuous diagonal variance on $U$.  Choose a
standard mollifier $\omega$.  For
$0<\delta<\tfrac12\operatorname{dist}(V,\partial U)$, convolve the three
fields in \eqref{eq:gff-domain-markov-local-decomposition} with the same
$\omega_\delta$.  Independence then gives the exact identity on $V$
\begin{align*}
 M_{h_D,\delta}^{(\omega)}(dz)
 &=\exp\!\left(
   \gamma\chi_{D,U,\delta}^{(\omega)}(z)
   -\frac{\gamma^2}{2}
    \mathbb E[(\chi_{D,U,\delta}^{(\omega)}(z))^2]
  \right)M_{h_U,\delta}^{(\omega)}(dz),\\
 \chi_{D,U,\delta}^{(\omega)}
 &\longrightarrow\chi_{D,U}
 \quad\text{locally uniformly on }U.
\end{align*}
Approximation independence
\cite[Corollary~18 and Theorem~25]{ShamovGMC} identifies the two
GMC limits with $M_{h_D}^{\rm can}$ and $M_{h_U}^{\rm can}$.  Letting
$\delta\downarrow0$ yields \eqref{eq:interior-dirichlet-gmc-comparison}, and
continuity on $\overline V$ gives
$w_{D,U},w_{D,U}^{-1}\in C_b(V)$.
\end{proof}

The localization argument also needs continuity of the massive--massless correction.

\begin{lemma}
\label{lem:massive-massless-continuous-correction}
Fix a disk $U\Subset\mathbb R^2$ and $m_0>0$.  Write $\Delta_U$ for the
Dirichlet Laplacian on $U$ and set $\Lzero^U:=-(2\pi)^{-1}\Delta_U$.  Consider
a centered Gaussian field $Z_U$ with covariance operator
\begin{equation}
 \Cov(Z_U)=C_U^{(m_0)}
 :=(\Lzero^U)^{-1}-(\Lzero^U+m_0^2)^{-1}
 =m_0^2(\Lzero^U)^{-1}
   (\Lzero^U+m_0^2)^{-1}.
 \label{eq:massive-massless-correction-covariance}
\end{equation}
Then $Z_U$ has a continuous version.  More precisely, for every $K\Subset U$
there is $C_{K,U,m_0}<\infty$ such that
\begin{equation}
 \mathbb E|Z_U(z)-Z_U(z')|^2
 \le C_{K,U,m_0}r^2
 \left(1+\log\frac1r\right),
 \qquad
 z,z'\in K,\quad 0<r:=|z-z'|\le\frac12.
 \label{eq:massive-massless-sharp-increment-bound}
\end{equation}
\end{lemma}

\begin{proof}
Let $G_U^{(m_0)}$ be the kernel of
$(\Lzero^U+m_0^2)^{-1}$.  Then
\[
 C_U^{(m_0)}(z,w)
 =m_0^2\int_U
 G_U(z,\zeta)G_U^{(m_0)}(\zeta,w)\,d\zeta
\]
and, by symmetry,
\begin{align*}
 \mathbb E|Z_U(z)-Z_U(z')|^2
 &=
 m_0^2\int_U
 [G_U(z,w)-G_U(z',w)]
 \notag\\
 &\hspace{29mm}\times
 [G_U^{(m_0)}(w,z)
  -G_U^{(m_0)}(w,z')]\,dw.
\end{align*}

Set
\[
 r_K:=\min\left\{\frac18,
        \frac1{16}\operatorname{dist}(K,\partial U)\right\}.
\]
The integral representation and the local logarithmic bound give
$\sup_{\zeta\in K}C_U^{(m_0)}(\zeta,\zeta)<\infty$.  Hence, for
$r_K\le r\le1/2$,
\[
 \mathbb E|Z_U(z)-Z_U(z')|^2
 \le4\sup_{\zeta\in K}C_U^{(m_0)}(\zeta,\zeta)
 \le C_{K,U,m_0}r^2\left(1+\log\frac1r\right).
\]
It remains to consider $0<r<r_K$.  Put
$A_r=B(z,4r)\cup B(z',4r)$.  The common logarithmic singularity and bounded
local regular parts give
\[
 \int_{A_r}
 |G_U(z,w)-G_U(z',w)|
 |G_U^{(m_0)}(w,z)-G_U^{(m_0)}(w,z')|\,dw
 \le
 C\int_0^{5r}\left(1+\log^+\frac r s\right)^2s\,ds
 \le Cr^2.
\]

On $U\setminus A_r$, interior gradient estimates on the dyadic annuli
\[
 A_\ell
 :=\{w:2^\ell r\le |w-z|<2^{\ell+1}r\},
 \qquad
 4r\le2^\ell r\le2\operatorname{diam}(U),
\]
give
\[
 |G_U(z,w)-G_U(z',w)|
 +|G_U^{(m_0)}(w,z)-G_U^{(m_0)}(w,z')|
 \le C\frac r{2^\ell r},
 \qquad w\in A_\ell.
\]
Here we used, for poles in a fixed neighborhood of $K$,
\[
 |\nabla_zG_U(z,w)|
 +|\nabla_zG_U^{(m_0)}(z,w)|
 \le C_{K,U,m_0}\left(1+\frac1{|z-w|}\right).
\]
Since $|A_\ell|\le C(2^\ell r)^2$,
\[
\begin{aligned}
 &\sum_\ell\int_{A_\ell}
 |G_U(z,w)-G_U(z',w)|
 |G_U^{(m_0)}(w,z)-G_U^{(m_0)}(w,z')|\,dw\\
 &\qquad\le Cr^2\sum_{\ell\le C\log(1/r)}1
 \le Cr^2\left(1+\log\frac1r\right).
\end{aligned}
\]

This proves \eqref{eq:massive-massless-sharp-increment-bound}.  Hence, for
every $\beta<1$ and every $q\ge2$,
\[
 \mathbb E|Z_U(z)-Z_U(z')|^q
 \le C_{K,\beta,q}|z-z'|^{q\beta}.
\]
Choose $q\beta>2$ and apply Kolmogorov's continuity theorem on each compact
$K\Subset U$.  A compact exhaustion gives a continuous version on $U$.
\end{proof}

For the trace-form construction, we need two concrete potential-theoretic
properties of $M_h$.  We first recall the notation.  For an open set
$G\subset O\subset\mathbb R^2$, set
\begin{equation*}
 \operatorname{Cap}_{1,O}(G)
 :=\inf\left\{
  \frac1{2\pi}\int_O|\nabla v|^2\,dz+\int_O|v|^2\,dz:
  v\in H_0^1(O),\ v\ge1\ \mathcal L^2\text{-a.e. on }G
 \right\}.
\end{equation*}
For a Borel set $A\subset O$, define
\begin{equation*}
 \operatorname{Cap}_{1,O}(A)
 :=\inf\left\{\operatorname{Cap}_{1,O}(G):
 A\subset G\subset O,\ G\text{ open}\right\},
\end{equation*}
with $H_0^1(\mathbb R^2)=H^1(\mathbb R^2)$.  A set of zero $1$-capacity is
called \emph{polar}; a property holds $\E$-quasi-everywhere (q.e.) if it fails
only on a polar set.  Thus the next lemma says that $M_h$ ignores polar sets
and that vanishing $M_h$-almost everywhere forces an $H_0^1(D)$ function to
vanish quasi-everywhere.  We prove these properties first on disks and then
transfer them to $D$ using the local comparison above.

\begin{lemma}
\label{lem:rough-domain-full-quasi-support}
Almost surely, $M_h(D)<\infty$ and, for every Borel set $A\subset D$,
\[
 \operatorname{Cap}_{1,D}(A)=0
 \quad\Longrightarrow\quad
 M_h(A)=0.
\]
Moreover, for every $u\in H_0^1(D)$,
\begin{equation}\label{eq:rough-domain-full-quasi-support-test}
 \widetilde u=0\quad M_h\text{-a.e.}
 \quad\Longrightarrow\quad
 u=0\quad\E\text{-q.e. on }D.
\end{equation}
Consequently, $M_h$ is a finite smooth measure for $(\E,H_0^1(D))$ and has
full $\E$-quasi-support.
\end{lemma}

\begin{proof}
We use repeatedly the following local comparison of capacities.  If
$O_1\subset O_2$, then zero extension and multiplication by a cutoff give
\begin{equation}\label{eq:localized-capacity-comparison}
 \begin{aligned}
 \operatorname{Cap}_{1,O_2}(A)
 &\le \operatorname{Cap}_{1,O_1}(A),\\
 \operatorname{Cap}_{1,O_1}(A)
 &\le C_{K,O_1,O_2}\operatorname{Cap}_{1,O_2}(A),
 \qquad A\subset K\Subset O_1.
 \end{aligned}
\end{equation}

\emph{The disk case.}
Fix a disk $U\Subset\mathbb R^2$ and $m_0>0$, and set
$\Lzero^{\mathbb R^2}:=-(2\pi)^{-1}\Delta_{\mathbb R^2}$.  Let
$\Phi^{(m_0)}$ be the whole-plane massive field with covariance
$(\Lzero^{\mathbb R^2}+m_0^2)^{-1}$.  Its Liouville measure is smooth in the
strict sense and has full quasi-support
\cite[Theorem~3.5 and Remark~3.7]{ShinLBMDirichlet}.

Let $\Phi_U^{(m_0)}$ be the Dirichlet massive field on $U$, with covariance
$(\Lzero^U+m_0^2)^{-1}$.  On a common probability space we may write
\[
 \Phi^{(m_0)}|_U=\Phi_U^{(m_0)}+\chi_U^{(m_0)},
 \qquad
 h_U=\Phi_U^{(m_0)}+Z_U,
\]
with independent corrections.  The field $\chi_U^{(m_0)}$ is
massive-harmonic, while $Z_U$ has a continuous version by
Lemma~\ref{lem:massive-massless-continuous-correction}.  Hence exact
factorization of smooth approximants and approximation independence
\cite[Corollary~18 and Theorem~25]{ShamovGMC} imply that, for every
$V\Subset U$, there are random constants $0<c_V\le C_V<\infty$ such that
\begin{equation}\label{eq:local-gmc-two-sided-comparison}
 c_V\,M_{\Phi^{(m_0)}}|_V
 \le M_{h_U}^{\rm can}|_V
 \le C_V\,M_{\Phi^{(m_0)}}|_V.
\end{equation}
This comparison transfers both required properties from the massive
Liouville measure to $M_{h_U}^{\rm can}$.

For full quasi-support, choose $K_j\Subset V_j\Subset U$ exhausting $U$ and
$\theta_j\in C_c^\infty(V_j)$ with $\theta_j=1$ on $K_j$.  If
$u\in H_0^1(U)$ and $\widetilde u=0$ $M_{h_U}^{\rm can}$-a.e., then
\begin{equation*}
 \begin{aligned}
 \widetilde{\theta_ju}=0\quad M_{\Phi^{(m_0)}}\text{-a.e.}
 &\Longrightarrow
 \theta_ju=0\quad\mathbb R^2\text{-q.e.}\\
 &\Longrightarrow
 u=0\quad U\text{-q.e. on }K_j.
 \end{aligned}
\end{equation*}
Here the first implication uses \eqref{eq:local-gmc-two-sided-comparison}
and full quasi-support of $M_{\Phi^{(m_0)}}$, while the second uses
\eqref{eq:localized-capacity-comparison}.  Since the $K_j$ exhaust $U$,
$M_{h_U}^{\rm can}$ has full quasi-support.

\medskip
For the polar-null property, let $A\subset U$ be polar.  Then
$A\cap K_j$ is $\mathbb R^2$-polar by
\eqref{eq:localized-capacity-comparison}.  Smoothness of
$M_{\Phi^{(m_0)}}$ and \eqref{eq:local-gmc-two-sided-comparison} give
\[
 M_{h_U}^{\rm can}(A\cap K_j)=0,
 \qquad j\ge1.
\]
Letting $j\to\infty$ gives $M_{h_U}^{\rm can}(A)=0$.

\medskip
\emph{Localization to $D$.}
Choose disks $U_j\Subset D$ and open sets $V_j\Subset U_j$ with
$D=\bigcup_{j\ge1}V_j$, together with
$\theta_j\in C_c^\infty(U_j)$ satisfying $\theta_j=1$ on $V_j$.
Lemma~\ref{lem:interior-dirichlet-gmc-comparison} and the disk case give,
for $u\in H_0^1(D)$ with $\widetilde u=0$ $M_h$-a.e.,
\begin{equation*}
 \begin{aligned}
 \widetilde{\theta_ju}=0\quad M_{h_{U_j}}^{\rm can}\text{-a.e.}
 &\Longrightarrow
 \theta_ju=0\quad U_j\text{-q.e.}\\
 &\Longrightarrow
 u=0\quad D\text{-q.e. on }V_j.
 \end{aligned}
\end{equation*}
The last implication follows from
$\operatorname{Cap}_{1,D}\le\operatorname{Cap}_{1,U_j}$.  Since the $V_j$
cover $D$, this proves \eqref{eq:rough-domain-full-quasi-support-test}.

Similarly, if $A\subset D$ is Borel and
$\operatorname{Cap}_{1,D}(A)=0$, then for every $j$,
\[
 \operatorname{Cap}_{1,U_j}(A\cap V_j)=0
 \Longrightarrow
 M_{h_{U_j}}^{\rm can}(A\cap V_j)=0
 \Longrightarrow
 M_h(A\cap V_j)=0.
\]
Here we use \eqref{eq:localized-capacity-comparison}, the disk case, and
Lemma~\ref{lem:interior-dirichlet-gmc-comparison}.  Since $D=\bigcup_jV_j$,
we obtain $M_h(A)=0$.

Finally, $M_h(D)<\infty$ follows from
Lemma~\ref{lem:canonical-borel-gmc}.  For a finite Radon measure, the
polar-null property gives smoothness, while
\eqref{eq:rough-domain-full-quasi-support-test} is precisely full
$\E$-quasi-support.
\end{proof}

Lemma~\ref{lem:rough-domain-full-quasi-support} gives the required trace
properties almost surely.  For the measurable spectral construction below, we
need them on one fixed Borel set of environments.  We record such a set and the
corresponding time-changed form explicitly.

\begin{proposition}
\label{prop:borel-trace-form-event}
There exists a Borel set $\Omega_{\rm tr}\subseteq\Omega_{\rm GMC}$ with
$\mathbb P(\Omega_{\rm tr})=1$ such that the conclusions of
Lemma~\ref{lem:rough-domain-full-quasi-support} hold for every
$x\in\Omega_{\rm tr}$, with $M_h$ replaced by $M_x$.  For each such $x$, the
killed Brownian motion time-changed by $M_x$ is associated on $L^2(M_x)$ with
the closed, densely defined Dirichlet form $(\E,\Vd_x)$, where
\begin{equation}\label{eq:trace-form-domain}
 \begin{aligned}
  \E(u,v)&=\frac1{2\pi}\int_D\nabla u\cdot\nabla v\,dz,
  \qquad u,v\in\Vd_x,\\
  \Vd_x&=\{u\in H_0^1(D):\widetilde u\in L^2(M_x)\}.
 \end{aligned}
\end{equation}
\end{proposition}

\begin{proof}
Only countably many auxiliary fields, localizations, and comparison events are
used in the proof of Lemma~\ref{lem:rough-domain-full-quasi-support}.  Put them
on a standard Borel space $Z_{\rm aux}$ and let $\overline{\mathbb P}$ be their
joint law with the canonical field.  Their common full-probability event may be
chosen so that
\begin{equation*}
 E\in\mathcal B(\Xspace\times Z_{\rm aux}),\qquad
 \overline{\mathbb P}(E)=1,\qquad
 (\operatorname{pr}_{\Xspace})_\#\overline{\mathbb P}=\mathbb P.
\end{equation*}
Set $A:=\operatorname{proj}_{\Xspace}E$.  The set $A$ is analytic and hence
measurable for the completed law $\mathbb P$.  Since $E\subseteq A\times
Z_{\rm aux}$,
\begin{equation*}
 1=\overline{\mathbb P}(E)
 \le \overline{\mathbb P}(A\times Z_{\rm aux})
 =\mathbb P(A),
\end{equation*}
so $\mathbb P(A)=1$.  Therefore
\begin{equation*}
 \Omega_{\rm tr}^{\rm an}:=\Omega_{\rm GMC}\cap A
 \qquad\text{satisfies}\qquad
 \mathbb P(\Omega_{\rm tr}^{\rm an})=1.
\end{equation*}
Choose a Borel full core
$\Omega_{\rm tr}\subseteq\Omega_{\rm tr}^{\rm an}$ as in
\eqref{eq:borel-full-core-convention}.  For $x\in\Omega_{\rm tr}$ the section
$E_x$ is nonempty, and any point of that section gives the conclusions of
Lemma~\ref{lem:rough-domain-full-quasi-support} for $M_x$.  Finally, the
extended Dirichlet space of killed Brownian motion is
$\Vd_{\rm e}=H_0^1(D)$; the time-change
theorem \cite[Theorem~6.2.1 and equation~(6.2.22)]{FukushimaOshimaTakeda}
then gives \eqref{eq:trace-form-domain}.
\end{proof}

Fix $x\in\Omega_{\rm tr}$ and $f\in\Hreg$.  Along this Cameron--Martin
direction we use the exponential tilt
\begin{equation}\label{eq:coherent-line-gmc}
 M_{x,\tau}^{f}:=e^{\gamma\tau f}M_x,
 \qquad \tau\in\mathbb R.
\end{equation}
Since $f$ is bounded, the density $e^{\gamma\tau f}$ is bounded above and below
by positive constants.  Thus all measures in \eqref{eq:coherent-line-gmc} are
finite and mutually equivalent, and they have the same polar null sets and full
$\E$-quasi-support as $M_x$.  The remaining point is their coherence and their
relation to the canonical GMC.

\begin{corollary}
\label{cor:coherent-gmc-lines}
The family \eqref{eq:coherent-line-gmc} satisfies
\begin{equation*}
 M_{x,\tau}^{f}
 =e^{\gamma(\tau-r)f}M_{x,r}^{f},
 \qquad r,\tau\in\mathbb R.
\end{equation*}
For the coordinate GFF $h$ and every fixed $\tau\in\mathbb R$,
\begin{equation}\label{eq:canonical-coherent-fixed-parameter}
 M_{h,\tau}^{f}=M_{h+\tau f}
 \qquad\text{almost surely}.
\end{equation}
\end{corollary}

\begin{proof}
For $r,\tau\in\mathbb R$, the definition \eqref{eq:coherent-line-gmc} gives
\begin{equation*}
 M_{x,\tau}^{f}
 =e^{\gamma\tau f}M_x
 =e^{\gamma(\tau-r)f}\,e^{\gamma r f}M_x
 =e^{\gamma(\tau-r)f}M_{x,r}^{f},
\end{equation*}
which proves the coherence identity.  For the coordinate field $h$ and a fixed
$\tau$, Theorem~\ref{thm:gmc-cm-shift} applied with $g=\tau f$ gives
\begin{equation*}
 M_{h+\tau f}=e^{\gamma\tau f}M_h=M_{h,\tau}^{f}
 \qquad\text{almost surely}.
\end{equation*}
Thus \eqref{eq:canonical-coherent-fixed-parameter} holds for each fixed
$\tau$.  In other words, for every prescribed value of $\tau$ there is a
probability-one event on which the identity holds.  We do not require the same
event to work for all $\tau\in\mathbb R$ at once.
\end{proof}

Along the coherent path \eqref{eq:coherent-line-gmc}, let
$A_{x,\tau}^{f}$ denote the trace-form generator associated with
$M_{x,\tau}^{f}$.

\subsubsection{Frostman control and a Borel Green-smoothing bound}
\label{subsec:early-frostman-green-package}

We establish the Frostman and Green-smoothing estimates needed below.

\begin{lemma}
\label{lem:green-kernel-regularity-package}
There exists a constant
$C_D\ge1$ such that
\begin{equation}\label{eq:green-kernel-global-log-bound}
 0\le G_D(z,w)
 \le C_D+\log^+\frac1{|z-w|},
 \qquad z,w\in D.
\end{equation}
Away from the diagonal, the Green kernel is locally uniformly continuous in
its first variable.  More precisely, for every compact $K\Subset D$ and every
$r_\ast>0$,
\begin{equation}\label{eq:green-kernel-uniform-offdiagonal-continuity}
 \begin{aligned}
 \omega_{K,r_\ast}(\delta)
 &:=\sup_{\substack{z,z'\in K,\ |z-z'|\le\delta\\
                    w\in D,\ |z-w|\ge r_\ast,\ |z'-w|\ge r_\ast}}
    |G_D(z,w)-G_D(z',w)|,
    \qquad \delta>0,\\
 \omega_{K,r_\ast}(\delta)
 &\xrightarrow[\delta\downarrow0]{}0.
 \end{aligned}
\end{equation}
\end{lemma}

\begin{proof}
Choose $R>0$ with $\overline D\subset B(0,R/2)$.  Domain monotonicity and the
disk Green function give
\[
 0\le G_D(z,w)\le G_{B(0,R)}(z,w)
 =\log\frac{|R^2-z\overline w|}{R|z-w|}
 \le \log^+\frac{5R}{4|z-w|}
 \le C_D+\log^+\frac1{|z-w|}.
\]
This is \eqref{eq:green-kernel-global-log-bound}.

Fix $K\Subset D$ and $r_\ast>0$, and set
\[
 r_0:=\frac18\min\{r_\ast,\operatorname{dist}(K,\partial D)\}.
\]
For $z\in K$, $|z-w|\ge r_\ast$, and $\zeta\in B(z,4r_0)$,
\[
 |\zeta-w|\ge r_\ast/2,
 \qquad
 \Delta_\zeta G_D(\zeta,w)=0.
\]
The interior gradient estimate and \eqref{eq:green-kernel-global-log-bound} yield
\[
 \sup_{\zeta\in B(z,2r_0)}|\nabla_\zeta G_D(\zeta,w)|
 \le\frac{C}{r_0}\sup_{\zeta\in B(z,4r_0)}G_D(\zeta,w)
 \le C_{D,K,r_\ast}.
\]
Thus, whenever $|z-z'|\le r_0$ and
$|z-w|\wedge|z'-w|\ge r_\ast$,
\[
 |G_D(z,w)-G_D(z',w)|
 \le C_{D,K,r_\ast}|z-z'|,
\]
which implies \eqref{eq:green-kernel-uniform-offdiagonal-continuity}.
\end{proof}

The random input is a uniform small-ball estimate for $M_h$.

\begin{lemma}
\label{lem:gmc-frostman-log-potential}
There exist a deterministic $\alpha=\alpha(\gamma)>0$ and a Borel set
$\Omega_{\rm Fr}\subseteq\Omega_{\rm GMC}$ such that
\[
 \mathbb P(\Omega_{\rm Fr})=1.
\]
For every $x\in\Omega_{\rm Fr}$, there is $C_x<\infty$ with
\begin{equation}\label{eq:gmc-uniform-frostman}
 M_x(B_D(z,r))\le C_x r^\alpha,
 \qquad z\in D,
 \quad 0<r\le1,
 \qquad B_D(z,r):=D\cap B(z,r).
\end{equation}
\end{lemma}

\begin{proof}
The massive whole-plane analogue, including uniformity in the center and in the
approximating sequence, is \cite[Theorem~2.2]{GarbanRhodesVargasLBM}.  For the
killed-heat zero-boundary setting on a rough domain, we need the same estimate
on a Borel full-probability event.

\emph{Uniform moment bound.}
Choose $p\in(1,4/\gamma^2)$ sufficiently close to $1$ that
\[
 \xi_\gamma(p)
 :=\left(2+\frac{\gamma^2}{2}\right)p
   -\frac{\gamma^2}{2}p^2>2,
\]
and fix $0<\alpha<(\xi_\gamma(p)-2)/p$.  Choose a square $S_0$ such that
$\operatorname{dist}(\overline D,\partial S_0)>4$.  Let $\Psi^{\rm ex}$ be
the exact-scale field with covariance
$K^{\rm ex}(z,w)=\log^+(T/|z-w|)$, where $T$ is large enough that $S_0$ lies in
its scaling region \cite[Proposition~2.15]{RhodesVargasGMCReview}, and let
$M_\gamma^{\rm ex}$ be its limiting chaos.  Lemma~\ref{lem:green-kernel-regularity-package}
gives $G_D\le K^{\rm ex}+C_0$ on $D\times D$.

Fix a square $Q$ with $Q\cap D\ne\varnothing$ and $\ell(Q)\le1$.  Choose
$\chi_Q\in C_c(S_0)$ with $0\le\chi_Q\le1$, equal to one on $Q$, and supported
in the concentric square $Q^+$ of side length $3\ell(Q)$.  The choice of $S_0$
ensures that $Q^+\Subset S_0$.  Kahane's convexity
inequality \cite[Theorem~2.1]{RhodesVargasGMCReview}, applied to
\eqref{eq:canonical-exact-scale-covariance-comparison}, gives
\[
 \mathbb E\!\left[\left(\int_D\chi_Q\,dM_{h,\eps}\right)^p\right]
 \le e^{\frac{\gamma^2}{2}C_0p(p-1)}
      \mathbb E\!\left[
       \left(\int_{S_0}\chi_Q\,dM_{\gamma,\delta(\eps)}^{\rm ex}\right)^p
      \right].
\]
Along the deterministic subsequence defining $M_h$, narrow
convergence applies to the continuous test function $\chi_Q$.  Fatou's lemma,
$L^p$ convergence of the exact-scale chaos, the positive-moment theorem, and
exact stochastic scaling \cite[Theorems~2.11 and~2.16]{RhodesVargasGMCReview}
therefore give
\begin{equation}\label{eq:gmc-dyadic-square-moment}
 \begin{aligned}
 \mathbb E\!\left[M_h(Q\cap D)^p\right]
 &\le
 e^{\frac{\gamma^2}{2}C_0p(p-1)}
 \mathbb E\!\left[M_\gamma^{\rm ex}(Q^+)^p\right]\\
 &\le C_p\,\ell(Q)^{\xi_\gamma(p)}.
 \end{aligned}
\end{equation}
The constant is independent of the position of $Q$.

\emph{A single Borel event.}
Put $Q_{n,k}:=2^{-n}(k+[0,1)^2)$ and let
\[
 \mathcal D_n
 :=\bigl\{Q_{n,k}:k\in\mathbb Z^2,\ Q_{n,k}\cap D\ne\varnothing\bigr\}.
\]
Since $D$ is bounded, $\#\mathcal D_n\le C_{\rm grid}(D)2^{2n}$.  Markov's inequality and
\eqref{eq:gmc-dyadic-square-moment} give
\[
 \mathbb P\!\left(
  \max_{Q\in\mathcal D_n}M_h(Q\cap D)>2^{-n\alpha}
 \right)
 \le C2^{-n(\xi_\gamma(p)-2-\alpha p)}.
\]
The exponent is positive by the choice of $\alpha$, so Borel--Cantelli gives
the desired bound for all sufficiently large $n$.  The remaining finitely many
scales are bounded by $M_h(D)<\infty$.  Consequently,
\begin{equation}\label{eq:borel-frostman-constant}
 C_{\rm Fr}(x)
 :=\sup_{n\ge0}2^{n\alpha}
  \max_{Q\in\mathcal D_n}M_x(Q\cap D)
 \in[0,\infty]
\end{equation}
and
\[
 \mathbb P\bigl(C_{\rm Fr}(h)<\infty\bigr)=1.
\]
Each measure evaluation in
\eqref{eq:borel-frostman-constant} is measurable, and the supremum is
countable.  Hence
\[
 C_{\rm Fr}:(\Xspace,\mathcal B(\Xspace))
 \longrightarrow([0,\infty],\mathcal B([0,\infty]))
\]
is measurable.  Thus
\begin{equation}\label{eq:borel-frostman-event}
 \Omega_{\rm Fr}
 :=\Omega_{\rm GMC}\cap\{x\in\Xspace:C_{\rm Fr}(x)<\infty\}
\end{equation}
satisfies
\[
 \Omega_{\rm Fr}\in\mathcal B(\Xspace),
 \qquad \mathbb P(\Omega_{\rm Fr})=1.
\]
If
$2^{-(n+1)}<r\le2^{-n}$, then $B_D(z,r)$ meets at most a deterministic number
$N_0$ of squares in $\mathcal D_n$.  Hence, for $x\in\Omega_{\rm Fr}$,
\[
 M_x(B_D(z,r))
 \le N_0C_{\rm Fr}(x)2^{-n\alpha}
 \le N_0 2^\alpha C_{\rm Fr}(x)r^\alpha.
\]
This is \eqref{eq:gmc-uniform-frostman} with
$C_x=N_0 2^\alpha C_{\rm Fr}(x)$.
\end{proof}

The Frostman bound controls the logarithmic Green singularity.

\begin{lemma}
\label{lem:frostman-green-smoothing}
Let $\mu$ be a finite positive Radon measure on $D$ satisfying, for some
$C_\mu<\infty$ and $\alpha>0$,
\begin{equation}\label{eq:deterministic-frostman}
 \mu(B_D(z,r))\le C_\mu r^\alpha,
 \qquad z\in D,
 \quad 0<r\le1.
\end{equation}
Then $\mu(\{z\})=0$ for every $z\in D$.  For every $p>0$,
\begin{equation}\label{eq:gmc-uniform-log-moment}
 L_p(\mu)
 :=\sup_{z\in D}\int_D
 \left(1+\log^+\frac1{|z-w|}\right)^p\mu(dw)<\infty.
\end{equation}
Moreover, for every compact $K\Subset D$,
\begin{equation}\label{eq:green-kernel-L2-continuity}
 \sup_{\substack{z,z'\in K\\|z-z'|\le\delta}}
 \int_D|G_D(z,w)-G_D(z',w)|^2\,\mu(dw)
 \xrightarrow[\delta\downarrow0]{}0.
\end{equation}
In particular,
\begin{equation}\label{eq:deterministic-green-square-bound}
 \sup_{z\in D}\int_DG_D(z,w)^2\,\mu(dw)
 \le C_D^2L_2(\mu)<\infty.
\end{equation}
\end{lemma}

\begin{proof}
The analogous potential and $L^2$-kernel bounds appear in
\cite[Proposition~2.3]{GarbanRhodesVargasLBM} and in the proof of
\cite[Proposition~5.2]{AndresKajinoHeatKernel}.  We include the deterministic
argument because the local $L^2(\mu)$-continuity below is needed to make the
subsequent supremum measurable.
Letting $r\downarrow0$ in \eqref{eq:deterministic-frostman} gives
$\mu(\{z\})=0$.  Put
\[
 \ell_z(w):=1+\log^+\frac1{|z-w|}.
\]
For
$A_k(z):=B_D(z,2^{-k})\setminus B_D(z,2^{-k-1})$ and $k\ge0$, one has
$\ell_z\le C(k+1)$ on $A_k(z)$.  Hence
\begin{align*}
 \int_D\ell_z(w)^p\,\mu(dw)
 &\le C\mu(D)+C\sum_{k\ge0}(k+1)^p\mu(B_D(z,2^{-k}))\\
 &\le C\mu(D)+CC_\mu\sum_{k\ge0}(k+1)^p2^{-\alpha k},
\end{align*}
uniformly in $z$.  This proves \eqref{eq:gmc-uniform-log-moment}.  The same
annular decomposition yields
\begin{equation*}
 \varepsilon_\mu(s)
 :=\sup_{z\in D}\int_{B_D(z,s)}\ell_z(w)^2\,\mu(dw)
 \xrightarrow[s\downarrow0]{}0.
\end{equation*}
Indeed, if $2^{-n-1}<s\le2^{-n}$, then
\[
 \varepsilon_\mu(s)
 \le CC_\mu\sum_{k\ge n-1}(k+1)^2 2^{-\alpha k}.
\]

Fix $K\Subset D$, $0<s<1/6$, and $z,z'\in K$ with $|z-z'|\le s$.  Set
$A=B_D(z,2s)\cup B_D(z',2s)$.  Since $|z-z'|\le s$, one has
$A\subset B_D(z,3s)$ and $A\subset B_D(z',3s)$.  Hence
\eqref{eq:green-kernel-global-log-bound} gives
\begin{align*}
 \int_A|G_D(z,w)-G_D(z',w)|^2\,\mu(dw)
 &\le2\int_A\bigl(G_D(z,w)^2+G_D(z',w)^2\bigr)\,\mu(dw)\\
 &\le C\varepsilon_\mu(3s).
\end{align*}
On $A^c$, both distances from $w$ to $z$ and $z'$ are at least $2s$, whence
\[
 \int_{A^c}|G_D(z,w)-G_D(z',w)|^2\,\mu(dw)
 \le \mu(D)\,\omega_{K,2s}(|z-z'|)^2.
\]
Consequently, whenever $0<\delta\le s$,
\begin{equation*}
 \sup_{\substack{z,z'\in K\\|z-z'|\le\delta}}
 \|G_D(z,\cdot)-G_D(z',\cdot)\|_{L^2(\mu)}^2
 \le C\varepsilon_\mu(3s)+\mu(D)\omega_{K,2s}(\delta)^2.
\end{equation*}
First let $\delta\downarrow0$ and then $s\downarrow0$.  This proves
\eqref{eq:green-kernel-L2-continuity};
\eqref{eq:deterministic-green-square-bound} follows from
\eqref{eq:green-kernel-global-log-bound} and
\eqref{eq:gmc-uniform-log-moment}.
\end{proof}

Fix an increasing compact exhaustion $D=\bigcup_{m\ge1}K_m$ and a countable
dense set $Z_m\subset K_m$ for each $m$.

\begin{lemma}
\label{lem:green-smoothing-constant}
Define
\begin{equation}\label{eq:green-smoothing-measurable-version}
 K_G(x)
 :=\sup_{m\ge1}\sup_{z\in Z_m}
 \int_DG_D(z,w)^2\,M_x(dw)\in[0,\infty].
\end{equation}
Then $K_G:\Xspace\to[0,\infty]$ is measurable and, for every
$x\in\Omega_{\rm Fr}$,
\begin{equation}\label{eq:green-smoothing-measurable-sup}
 K_G(x)
 =\sup_{z\in D}\int_DG_D(z,w)^2\,M_x(dw)<\infty.
\end{equation}
For such $x$, the function
\[
 F_x(z):=\int_DG_D(z,w)^2\,M_x(dw),\qquad z\in D,
\]
is continuous.
\end{lemma}

\begin{proof}
For fixed $z\in D$, the map
\[
 \nu\longmapsto\int_DG_D(z,w)^2\,\nu(dw):
 (\Mc_{\rm fin}^+(D),\mathcal B(\Mc_{\rm fin}^+(D)))
 \longrightarrow([0,\infty],\mathcal B([0,\infty]))
\]
is Borel measurable with respect to the narrow Borel structure.  After composition
with $x\mapsto M_x$, the countable supremum in
\eqref{eq:green-smoothing-measurable-version} gives the measurability of
$K_G$.

Fix $x\in\Omega_{\rm Fr}$.  Lemma~\ref{lem:frostman-green-smoothing} gives
$C_x^{G}:=\sup_{z\in D}F_x(z)<\infty$ and local
$L^2(M_x)$-continuity of $z\mapsto G_D(z,\cdot)$.  Therefore
\[
 |F_x(z)-F_x(z')|
 \le 2(C_x^{G})^{1/2}
 \|G_D(z,\cdot)-G_D(z',\cdot)\|_{L^2(M_x)}
 \longrightarrow0
\]
whenever $z'\to z$ within a fixed compact subset of $D$.  Thus $F_x$ is
continuous and, by density of $Z_m$ in $K_m$,
\[
 \sup_{z\in Z_m}F_x(z)
 =\sup_{z\in K_m}F_x(z).
\]
Taking the supremum over $m$ and using
$D=\bigcup_{m\ge1}K_m$ proves
\eqref{eq:green-smoothing-measurable-sup}.
\end{proof}

\subsection{Green potentials and measurable spectral objects}
\subsubsection{Finite Green-energy measures}

This subsection has three steps.  We first construct Green potentials for
finite-energy signed measures.  We then identify the resulting integral Green
operator with $A_x^{-1}$.  Finally, we place these operators and their spectra
in a measurable-field framework.

Recall the non-negative Dirichlet Green kernel $G_D$ associated with $\E$.  For
finite positive Radon measures $\mu_1$ and $\mu_2$ on $D$, set
\[
 I_D(\mu_1,\mu_2)
 :=\iint_{D^2}G_D(z,w)\,\mu_1(dz)\mu_2(dw)\in[0,\infty],
 \qquad
 I_D(\mu_1):=I_D(\mu_1,\mu_1).
\]
A finite signed Radon measure $\nu=\nu^+-\nu^-$ has finite Green energy if
\[
 I_D(\nu^+)+I_D(\nu^-)<\infty.
\]
We write $\ME$ for the class of all such measures.

The finite-energy extension is standard
\cite[Lemma~2.2.1(ii), Lemma~2.2.3, and
Theorem~2.2.5]{FukushimaOshimaTakeda}.  What remains specific here is to
identify that potential with the normalized Dirichlet Green kernel on a rough
domain and to fix the representatives used later.

\begin{lemma}
\label{lem:positive-finite-energy-potential}
For every finite positive Radon measure $\mu$ with $I_D(\mu)<\infty$, there is
a unique potential $\Pot\mu\in\Hc$ characterized by
\[
 \E(\Pot\mu,v)=\int_D\widetilde v\,d\mu,
 \qquad v\in\Hc.
\]
Moreover,
\[
 \E(\Pot\mu,\Pot\mu)=I_D(\mu),
\]
and its quasi-continuous representative is given by
\begin{equation}
 \widetilde{\Pot\mu}(z)
 =\int_DG_D(z,w)\,\mu(dw)
 \qquad\text{for $\E$-q.e. }z\in D.
 \label{eq:finite-energy-green-representation}
\end{equation}
\end{lemma}

\begin{proof}
\emph{Semigroup approximation.}
Let $P_t^D=e^{-t\Lzero}$ and write $q_t^D$ for its kernel.  For $t>0$ set
\begin{align*}
 G_t(z,w)&:=\int_t^\infty q_s^D(z,w)\,ds,
 &u_t(z)&:=\int_DG_t(z,w)\,\mu(dw),\\
 I_t(\mu)&:=\iint_{D^2}G_t(z,w)\,\mu(dz)\mu(dw).
\end{align*}
Heat-kernel domination for small times and the spectral gap for large times
make $G_t$ bounded for every $t>0$.  Equivalently,
if
\[
 g_t(z):=\int_Dq_t^D(z,w)\,\mu(dw),
\]
then $g_t\in L^\infty(D)\subset L^2(D,\mathcal L^2)$ and, by Tonelli's
theorem and the semigroup identity,
\[
 \Lzero^{-1}g_t(z)
 =\int_0^\infty P_s^Dg_t(z)\,ds
 =\int_t^\infty\!\int_Dq_r^D(z,w)\,\mu(dw)\,dr
 =u_t(z).
\]
Thus $u_t\in\Dom(\Lzero)$.  Interior regularity identifies the kernel
integral above with the continuous, hence quasi-continuous, representative of
$u_t$.  For $v\in C_c^\infty(D)$, symmetry and Tonelli's theorem give
\begin{equation}\label{eq:finite-energy-semigroup-core-pairing}
 \E(u_t,v)=\int_Dg_tv\,d\mathcal L^2
 =\int_DP_t^Dv\,d\mu.
\end{equation}
Moreover, for $s,t>0$,
\begin{equation}\label{eq:finite-energy-semigroup-energy}
 \E(u_t,u_s)
 =\int_Dg_tu_s\,d\mathcal L^2
 =I_{t+s}(\mu).
\end{equation}
Since $G_t\uparrow G_D$ as $t\downarrow0$,
$I_t(\mu)\uparrow I_D(\mu)$.  Hence
\[
 \|u_t-u_s\|_{\Hc}^2
 =I_{2t}(\mu)+I_{2s}(\mu)-2I_{t+s}(\mu)
 \longrightarrow0,
 \qquad s,t\downarrow0.
\]
\emph{Passage to the limit and Green representation.}
Let $u$ be the resulting limit in $H_0^1(D)$.  If
$v\in C_c^\infty(D)$, then
\[
 \|P_t^Dv-v\|_\infty\le t\|\Lzero v\|_\infty,
\]
and \eqref{eq:finite-energy-semigroup-core-pairing} yields
\begin{equation}\label{eq:finite-energy-core-bound}
 \E(u,v)=\int_Dv\,d\mu,
 \qquad
 \left|\int_Dv\,d\mu\right|
 \le I_D(\mu)^{1/2}\E(v,v)^{1/2}.
\end{equation}
Since $D$ is bounded, Poincar\'e's inequality makes the form transient and
identifies its extended Dirichlet space with $H_0^1(D)$.  The special-core
criterion \cite[Lemma~2.2.1(ii)]{FukushimaOshimaTakeda} extends
\eqref{eq:finite-energy-core-bound} from $C_c^\infty(D)$ to finite $0$-order
energy.  The $0$-order potential theorem and the polar-set criterion
\cite[Lemma~2.2.3 and Theorem~2.2.5]{FukushimaOshimaTakeda} then yield
\[
 \E(u,v)=\int_D\widetilde v\,d\mu,
 \qquad v\in H_0^1(D).
\]
Boundedness of $D$ also makes the $0$- and $1$-order capacities equivalent.

Choose $t_n\downarrow0$ and pass to a subsequence, not relabelled.  Strong
convergence in $H_0^1(D)$, hence in the $\E_1$-norm by Poincar\'e's inequality,
gives $u_{t_n}\to\widetilde u$ quasi-everywhere.  On the other hand, monotone
convergence of the continuous kernel representatives gives, for every $z\in D$,
\[
 u_{t_n}(z)\uparrow\int_DG_D(z,w)\,\mu(dw).
\]
Thus \eqref{eq:finite-energy-green-representation} holds with
$\Pot\mu=u$.  Finally,
\[
 \E(\Pot\mu,\Pot\mu)
 =\lim_{t\downarrow0}\E(u_t,u_t)
 =\lim_{t\downarrow0}I_{2t}(\mu)=I_D(\mu),
\]
and uniqueness follows by testing the difference of two solutions against
itself.
\end{proof}

For signed measures for which the Jordan-part integrals are finite, extend the
Green energy bilinearly by
\[
 \begin{aligned}
 I_D(\nu_1,\nu_2)
 &:=I_D(\nu_1^+,\nu_2^+)
   -I_D(\nu_1^+,\nu_2^-)
   -I_D(\nu_1^-,\nu_2^+)
   +I_D(\nu_1^-,\nu_2^-),\\
 I_D(\nu_j)&:=I_D(\nu_j,\nu_j),
 \qquad j=1,2.
 \end{aligned}
\]

\begin{lemma}
\label{lem:finite-energy-potential}
Let $\nu,\eta\in\ME$.  Then $I_D(|\nu|,|\eta|)<\infty$, so the mutual Green
energy $I_D(\nu,\eta)$ is well defined.  For each
$\sigma\in\{\nu,\eta\}$ there is a unique $\Pot\sigma\in\Hc$ characterized by
\begin{equation}
 \E(\Pot\sigma,v)=\int_D\widetilde v\,d\sigma,
 \qquad v\in\Hc.
 \label{eq:finite-energy-potential-pairing}
\end{equation}
Its quasi-continuous representative is
\begin{equation}
 \widetilde{\Pot\sigma}(z)
 =\int_DG_D(z,w)\,\sigma(dw)
 \qquad\text{for $\E$-q.e. }z\in D,
 \label{eq:signed-finite-energy-green-representation}
\end{equation}
where the signed integral is taken through the Jordan decomposition.  In
addition,
\begin{align}
 I_D(\nu,\eta)
 &=\E(\Pot\nu,\Pot\eta),
 \label{eq:green-mutual-energy}\\
 |I_D(\nu,\eta)|
 &\le I_D(\nu)^{1/2}I_D(\eta)^{1/2},
 \label{eq:green-energy-cs}\\
 \|\Pot\nu\|_{\Hc}^2
 &=I_D(\nu),
 \qquad
 \langle\Pot\nu,f\rangle_{\Hc}=\int_D\widetilde f\,d\nu,
 \quad f\in\Hc.
 \label{eq:response-potential-energy-identity}
\end{align}
\end{lemma}

\begin{proof}
For positive $\mu_1,\mu_2$ of finite self-energy,
Lemma~\ref{lem:positive-finite-energy-potential} shows that $\mu_2$ charges no
polar set.  Its Green representation and potential pairing therefore give
\begin{equation}\label{eq:positive-mutual-energy-pairing}
 \begin{aligned}
 I_D(\mu_1,\mu_2)
 &=\int_D\widetilde{\Pot\mu_1}\,d\mu_2
 =\E(\Pot\mu_1,\Pot\mu_2),\\
 I_D(\mu_1,\mu_2)
 &\le I_D(\mu_1)^{1/2}I_D(\mu_2)^{1/2}.
 \end{aligned}
\end{equation}
Applying this estimate to the Jordan parts gives
\[
 I_D(|\nu|,|\eta|)
 \le
 \bigl(I_D(\nu^+)^{1/2}+I_D(\nu^-)^{1/2}\bigr)
 \bigl(I_D(\eta^+)^{1/2}+I_D(\eta^-)^{1/2}\bigr)
 <\infty.
\]
For $\sigma\in\{\nu,\eta\}$, define
$\Pot\sigma:=\Pot(\sigma^+)-\Pot(\sigma^-)$.  Subtracting the positive
pairings and Green representations gives
\eqref{eq:finite-energy-potential-pairing} and
\eqref{eq:signed-finite-energy-green-representation}.  Bilinearity in
\eqref{eq:positive-mutual-energy-pairing} gives
\[
 I_D(\nu,\eta)
 =\E(\Pot\nu,\Pot\eta),
 \qquad
 I_D(\nu)=\|\Pot\nu\|_{\Hc}^2\ge0.
\]
Hilbert-space Cauchy--Schwarz now gives
\[
 |I_D(\nu,\eta)|
 \le I_D(\nu)^{1/2}I_D(\eta)^{1/2},
\]
which proves \eqref{eq:green-mutual-energy}--
\eqref{eq:response-potential-energy-identity}.
If $V\in\Hc$ also satisfies
\eqref{eq:finite-energy-potential-pairing}, then
\[
 \E(\Pot\sigma-V,v)=0\quad(v\in\Hc),
 \qquad v=\Pot\sigma-V
 \quad\Longrightarrow\quad V=\Pot\sigma.
\]
\end{proof}

The potential representation makes the signed Green energy positive definite.

\begin{lemma}
\label{lem:green-energy-strict-positive}
Every nonzero $\nu\in\ME$ satisfies
\begin{equation}\label{eq:strict-positive-green-energy}
        I_D(\nu)>0.
\end{equation}
\end{lemma}

\begin{proof}
If $I_D(\nu)=0$, then
\eqref{eq:response-potential-energy-identity} gives $\Pot\nu=0$ in $\Hc$.
Hence, for every $\phi\in C_c^\infty(D)$,
\[
 \int_D\phi\,d\nu
 =\langle\Pot\nu,\phi\rangle_{\Hc}=0.
\]
Thus $\nu=0$ as a distribution and therefore as a finite signed Radon measure.
\end{proof}

\subsubsection{The Liouville Green operator}

For $x\in\Omega_{\rm tr}$, let $A_x$ be the non-negative self-adjoint generator
of the time-changed form from Proposition~\ref{prop:borel-trace-form-event}; thus
\begin{equation}\label{eq:liouville-form-domain}
 \Vd_x=\{u\in H_0^1(D):\widetilde u\in L^2(M_x)\},
 \qquad
 \E(u,v)=\frac1{2\pi}\int_D\nabla u\cdot\nabla v\,dz.
\end{equation}

The Green-resolvent representation and its Hilbert--Schmidt consequences are
standard in LBM; see
\cite[(3.1), Lemma~3.1, and (3.5)--(3.6)]{MaillardRhodesVargasZeitouniHeatKernel},
\cite[Lemma~2.4, Definition~2.5, Corollary~2.7, and
Lemma~2.8]{BerestyckiSpectralGeometry}, and the proof of
\cite[Proposition~5.2]{AndresKajinoHeatKernel}.  Those references treat the
standard whole-plane, torus, compact-surface, or regular-domain models.  Here
we also need the inverse on a rough Dirichlet domain and a realization
measurable in the canonical GMC field.

Full support makes $L^2(M_x)$ infinite-dimensional, a fact used below to
obtain infinitely many positive Green eigenvalues.

\begin{lemma}
\label{lem:gmc-full-support}
For every $x\in\Omega_{\rm tr}$,
\begin{equation}\label{eq:gmc-full-support}
        M_x(O)>0
        \qquad\text{for every non-empty open }O\subset D.
\end{equation}
\end{lemma}

\begin{proof}
Suppose that $M_x(O)=0$ for some non-empty open $O\subset D$, and choose
$0\ne\phi\in C_c^\infty(O)$.  Then $\widetilde\phi=0$ $M_x$-almost everywhere.
Full $\E$-quasi-support gives $\phi=0$ quasi-everywhere, hence
$\phi=0$ in $H_0^1(D)$, a contradiction.
\end{proof}

The smoothing estimate makes the Green kernel Hilbert--Schmidt against the
random speed measure.

Set
\begin{align}
 \Omega_{\rm HS}
 &:=\left\{x\in\Xspace:
   \iint_{D^2}G_D(z,w)^2\,M_x(dz)M_x(dw)<\infty\right\}.
 \label{eq:measurable-HS-event}\\
 \Omega_{\rm op}&:=\Omega_{\rm tr}\cap\Omega_{\rm HS}.
 \label{eq:measurable-operator-event}
\end{align}
For $x\in\Omega_{\rm HS}$, define the Green integral operator on $L^2(M_x)$ by
\begin{equation}\label{eq:green-operator-kh}
 \GreenOp_x\theta(z):=\int_DG_D(z,w)\theta(w)\,M_x(dw),
 \qquad \theta\in L^2(M_x),
\end{equation}
where the integral is defined for $M_x$-almost every $z$.  The defining
condition of $\Omega_{\rm HS}$ is exactly
$G_D\in L^2(M_x\otimes M_x)$, so \eqref{eq:green-operator-kh} defines a
Hilbert--Schmidt operator.

\begin{lemma}
\label{lem:green-hilbert-schmidt}
The inclusion $\Omega_{\rm Fr}\subseteq\Omega_{\rm HS}$ holds.  For every
$x\in\Omega_{\rm HS}$, the operator $\GreenOp_x$ is positive, self-adjoint,
and Hilbert--Schmidt, with
\begin{equation}\label{eq:green-kernel-HS}
 \|\GreenOp_x\|_{\rm HS}^2
 =\iint_{D^2}G_D(z,w)^2\,M_x(dz)M_x(dw)<\infty.
\end{equation}
For all $\theta,\eta\in L^2(M_x)$,
\begin{align}
 \iint_{D^2}|G_D(z,w)\theta(z)\eta(w)|\,M_x(dz)M_x(dw)
 &\le \|\GreenOp_x\|_{\mathrm{HS}}\|\theta\|_2\|\eta\|_2,              \label{eq:green-kernel-absolute}\\
 I_D(\theta M_x,\eta M_x)
 &=\langle\theta,\GreenOp_x\eta\rangle_{L^2(M_x)},                       \label{eq:green-operator-bilinear}\\
 |I_D(\theta M_x,\eta M_x)|
 &\le \|\GreenOp_x\|_{\mathrm{HS}}\|\theta\|_2\|\eta\|_2.            \label{eq:energy-L2-control}
\end{align}
Consequently, $\theta M_x\in\ME$ for every $\theta\in L^2(M_x)$, and
\begin{equation}\label{eq:green-operator-strict-positive}
 \theta\ne0
 \quad\Longrightarrow\quad
 \langle\theta,\GreenOp_x\theta\rangle_{L^2(M_x)}
 =I_D(\theta M_x)>0.
\end{equation}
In particular, $\ker\GreenOp_x=\{0\}$.
\end{lemma}

\begin{proof}
For $x\in\Omega_{\rm Fr}$, Lemma~\ref{lem:green-smoothing-constant} gives
\begin{equation}\label{eq:green-HS-from-smoothing}
 \iint_{D^2}G_D(z,w)^2\,M_x(dz)M_x(dw)
 \le M_x(D)K_G(x)<\infty,
\end{equation}
which proves $\Omega_{\rm Fr}\subseteq\Omega_{\rm HS}$.

Fix $x\in\Omega_{\rm HS}$.  The Hilbert--Schmidt kernel formula gives
\eqref{eq:green-kernel-HS}, and symmetry of $G_D$ gives self-adjointness.  For
$\theta,\eta\in L^2(M_x)$, Cauchy--Schwarz on
$L^2(M_x\otimes M_x)$ yields
\[
 \iint_{D^2}|G_D(z,w)\theta(z)\eta(w)|\,M_x(dz)M_x(dw)
 \le \|\GreenOp_x\|_{\rm HS}\|\theta\|_2\|\eta\|_2.
\]
This proves \eqref{eq:green-kernel-absolute}; Fubini then gives
\eqref{eq:green-operator-bilinear}, and
\eqref{eq:energy-L2-control} follows immediately.

Taking both arguments equal to $|\theta|$ gives
$I_D(|\theta|M_x)<\infty$.  The positive and negative parts of $\theta M_x$
are dominated by $|\theta|M_x$, hence $\theta M_x\in\ME$.  By
Lemma~\ref{lem:green-energy-strict-positive},
\[
 \langle\theta,\GreenOp_x\theta\rangle_{L^2(M_x)}
 =I_D(\theta M_x)\ge0,
\]
with strict inequality when $\theta\ne0$.  Thus $\GreenOp_x$ is positive and
\eqref{eq:green-operator-strict-positive} holds; injectivity follows.
\end{proof}

The resulting integral operator is the inverse of the trace-form generator.

\begin{lemma}
\label{lem:green-form-inverse}
Fix $x\in\Omega_{\rm op}$ and $\theta\in L^2(M_x)$.  The potential
$u:=\Pot(\theta M_x)$ belongs to $\Vd_x$ and is the unique solution of
\begin{equation}
 \E(u,v)=\int_D\theta\widetilde v\,dM_x,
 \qquad v\in\Vd_x.
 \label{eq:green-form-inverse-weak}
\end{equation}
Its $L^2(M_x)$ representative and its energy are given by
\begin{align}
 \widetilde u&=\GreenOp_x\theta\quad\text{in }L^2(M_x),
 \label{eq:green-form-inverse-L2-identity}\\
 I_D(\theta M_x)&=\langle\theta,\GreenOp_x\theta\rangle_{L^2(M_x)}.
 \label{eq:green-form-inverse-energy}
\end{align}
Equivalently,
\[
 \GreenOp_x\theta\in\Dom(A_x),
 \qquad A_x\GreenOp_x\theta=\theta.
\]
\end{lemma}

\begin{proof}
Set $u:=\Pot(\theta M_x)$.  Lemma~\ref{lem:green-hilbert-schmidt} gives
$\theta M_x\in\ME$, so $u\in\Hc$ is well defined by
Lemma~\ref{lem:finite-energy-potential}.  In view of
\eqref{eq:liouville-form-domain}, it remains first to prove
$\widetilde u\in L^2(M_x)$; duality against bounded test functions will then
identify its $L^2(M_x)$ class with $\GreenOp_x\theta$.

\emph{$L^2$ integrability.}
Since $M_x$ is smooth, $\widetilde u$ is finite $M_x$-almost everywhere.  Put
\[
 b_n:=\operatorname{sgn}(\widetilde u)(|\widetilde u|\wedge n).
\]
Because $M_x(D)<\infty$, $b_n\in L^2(M_x)$ and
$b_nM_x\in\ME$.  Equations
\eqref{eq:finite-energy-potential-pairing} and
\eqref{eq:green-operator-bilinear} yield
\begin{align*}
 \|b_n\|_2^2
 \le
 \int_D\widetilde u\,b_n\,dM_x
 &=\E(\Pot(b_nM_x),u)\\
 &=\langle b_n,\GreenOp_x\theta\rangle_{L^2(M_x)}
 \le\|b_n\|_2\|\GreenOp_x\theta\|_2.
\end{align*}
Thus $\|b_n\|_2\le\|\GreenOp_x\theta\|_2$, and monotone convergence gives
$\widetilde u\in L^2(M_x)$.

\emph{Identification and uniqueness.}
For every $b\in L^\infty(M_x)\cap L^2(M_x)$, the same identities give
\[
 \int_Db\widetilde u\,dM_x
 =\E(\Pot(bM_x),u)
 =I_D(bM_x,\theta M_x)
 =\langle b,\GreenOp_x\theta\rangle_{L^2(M_x)}.
\]
Density of $L^\infty\cap L^2$ in $L^2$ proves
$\widetilde u=\GreenOp_x\theta$ in $L^2(M_x)$, and therefore
$u\in\Vd_x$.  Equation \eqref{eq:green-form-inverse-weak} is
\eqref{eq:finite-energy-potential-pairing}.  If $u_1,u_2\in\Vd_x$ both satisfy
\eqref{eq:green-form-inverse-weak}, then, with $w=u_1-u_2$,
\[
 \E(w,v)=0\quad(v\in\Vd_x),
 \qquad v=w\quad\Longrightarrow\quad \E(w,w)=0
 \quad\Longrightarrow\quad w=0.
\]
Thus the solution is unique.  Finally,
\eqref{eq:green-form-inverse-energy} follows from
\eqref{eq:green-operator-bilinear}.
\end{proof}

\subsubsection{Measurable-field tools}
\label{subsubsec:measurable-field-tools}

The Hilbert space $L^2(D,\mu)$ changes with the measure $\mu$, so ordinary
operator-valued measurability is not directly available.  We therefore test
vectors and operators against one fixed countable family of functions.  The
next three lemmas record the resulting measurable-field facts.  They are used
here to verify Assumption~\ref{ass:spectral-borel} for LBM and will be reused in
Section~\ref{sec:lcp-realization}.

Choose $(\chi_j)_{j\ge1}\subset\Hreg$ whose rational span is uniformly dense in
$C_0(D)$.  For a finite positive Radon measure $\mu$, set $H_\mu=L^2(D,\mu)$ and
$s_j(\mu)=[\chi_j]_\mu$.  A map $\mu\mapsto\xi(\mu)\in H_\mu$ is called measurable when
\[
 \mu\longmapsto\langle s_j(\mu),\xi(\mu)\rangle_{H_\mu}
 \quad\text{is measurable for every }j.
\]
An operator field $T_\mu$ is called measurable when
\begin{equation}\label{eq:measurable-operator-field-definition}
 \mu\longmapsto
 \langle s_i(\mu),T_\mu s_j(\mu)\rangle_{H_\mu}
 \quad\text{is measurable for every }i,j.
\end{equation}

For each $\mu$, apply Gram--Schmidt to $(s_j(\mu))_{j\ge1}$ by
\begin{equation}\label{eq:measurable-frame-gram-schmidt}
 v_j=s_j-\sum_{\ell<j}\langle s_j,e_\ell\rangle e_\ell,
 \qquad
 e_j=
 \begin{cases}
 v_j/\|v_j\|,&\|v_j\|>0,\\
 0,&\|v_j\|=0.
 \end{cases}
\end{equation}
We shall use the notation
\begin{equation}\label{eq:measurable-frame-definitions}
 \mathcal E_\mu
 :=\{e_j(\mu):j\ge1,\ e_j(\mu)\ne0\}.
\end{equation}
For $N\ge1$, define the finite-rank operator $J_N(\mu)$ on $H_\mu$ by
\begin{equation}\label{eq:measurable-frame-projection}
 J_N(\mu)u
 :=\sum_{j=1}^N
   \langle u,e_j(\mu)\rangle_{H_\mu}\,e_j(\mu),
 \qquad u\in H_\mu.
\end{equation}

\begin{lemma}
\label{lem:measurable-Hilbert-field-calculus}
From the preceding definitions we obtain the following three properties.
\begin{enumerate}[label=\textnormal{(\roman*)},leftmargin=*,itemsep=0pt,topsep=2pt,parsep=0pt]
 \item For every $j\ge1$, the map $\mu\mapsto e_j(\mu)\in H_\mu$ is
 measurable.
 \item For every $\mu$, $\mathcal E_\mu$ is an orthonormal basis of $H_\mu$.
 In particular, $J_N(\mu)$ is the orthogonal projection onto
 $\operatorname{span}\{e_1(\mu),\ldots,e_N(\mu)\}$ and
 \begin{equation}\label{eq:measurable-frame-density}
  \|J_N(\mu)u-u\|_{H_\mu}\longrightarrow0
  \qquad\text{for every }u\in H_\mu.
 \end{equation}
 \item An operator field $T$ is measurable in the sense of
 \eqref{eq:measurable-operator-field-definition} if and only if
 \[
  \mu\longmapsto
  \langle e_i(\mu),T_\mu e_j(\mu)\rangle_{H_\mu}
  \quad\text{is measurable for every }i,j.
 \]
\end{enumerate}
\end{lemma}

\begin{proof}
The Gram matrix is measurable because
\[
 \langle s_i(\mu),s_j(\mu)\rangle_{H_\mu}
 =\int_D\chi_i\chi_j\,d\mu.
\]
Induction in \eqref{eq:measurable-frame-gram-schmidt} therefore gives measurable
coefficients $a_{j\ell},b_{j\ell}$ such that
\[
 e_j=\sum_{\ell\le j}a_{j\ell}s_\ell,
 \qquad
 s_j=\sum_{\ell\le j}b_{j\ell}e_\ell.
\]
Thus every $e_j$ is measurable, and
\[
 \overline{\operatorname{span}}\{e_j(\mu):j\ge1\}
 =\overline{\operatorname{span}}\{s_j(\mu):j\ge1\}
 =\overline{\{[\chi]_\mu:
   \chi\in\operatorname{span}_{\mathbb Q}\{\chi_j:j\ge1\}\}}
 =L^2(D,\mu).
\]
The nonzero $e_j(\mu)$ are orthonormal by construction, so they form an
orthonormal basis of $H_\mu$; consequently
$\|J_N(\mu)u-u\|_{H_\mu}\to0$ for every $u\in H_\mu$.
The two matrix criteria are equivalent because
\begin{align*}
 \langle e_i,Te_j\rangle
 &=\sum_{r\le i}\sum_{s\le j}
   a_{ir}a_{js}\langle s_r,Ts_s\rangle,\\
 \langle s_i,Ts_j\rangle
 &=\sum_{r\le i}\sum_{s\le j}
   b_{ir}b_{js}\langle e_r,Te_s\rangle.
\end{align*}
\end{proof}

Fix bounded operator fields
$S=(S_\mu)_\mu$ and $T=(T_\mu)_\mu$.  Relative to the frame constructed
above, their adjoint and product have matrix coefficients
\begin{align}
 \langle e_i(\mu),S_\mu^*e_j(\mu)\rangle_{H_\mu}
 &=\langle S_\mu e_i(\mu),e_j(\mu)\rangle_{H_\mu},
 \label{eq:measurable-adjoint-matrix}\\
 \langle e_i(\mu),S_\mu T_\mu e_j(\mu)\rangle_{H_\mu}
 &=\lim_{N\to\infty}\sum_{r=1}^N
   \langle e_i(\mu),S_\mu e_r(\mu)\rangle_{H_\mu}
   \langle e_r(\mu),T_\mu e_j(\mu)\rangle_{H_\mu}.
 \label{eq:measurable-product-matrix}
\end{align}
When $T_\mu$ is positive, self-adjoint, and compact for every $\mu$, define
its finite-dimensional compression by
\begin{equation}\label{eq:measurable-finite-dimensional-compression}
 T_{\mu,N}:=J_N(\mu)T_\mu J_N(\mu),
 \qquad N\ge1.
\end{equation}
We write
\begin{equation}\label{eq:measurable-ordered-compact-eigenvalues}
 \lambda_1(T_\mu)\ge\lambda_2(T_\mu)\ge\cdots\ge0
\end{equation}
for the eigenvalues of $T_\mu$ in non-increasing order, counted with
multiplicity and continued by zeros when the nonzero spectrum is finite.

\begin{lemma}
\label{lem:measurable-operator-field-calculus}
From the preceding definitions we obtain the following three properties.
\begin{enumerate}[label=\textnormal{(\roman*)},leftmargin=2.1em,itemsep=2pt,topsep=3pt]
 \item If $S$ and $T$ are measurable, then
 \[
  S^*,\ ST\quad\text{are measurable operator fields}.
 \]
 Consequently, measurable bounded operator fields are closed under adjoints,
 products, and polynomials.

 \item Let $(T^{(\ell)})_{\ell\ge1}$ be measurable bounded operator fields.  If
 \begin{equation*}
  T_\mu^{(\ell)}\xrightarrow[\ell\to\infty]{\rm s}T_\mu
  \qquad\text{for every }\mu,
 \end{equation*}
 then $T=(T_\mu)_\mu$ is measurable.

 \item Suppose $T$ is measurable and every $T_\mu$ is positive,
 self-adjoint, and compact.  Then, for every $F\in C_b([0,\infty))$, every
 open interval $I\subset(0,\infty)$, and every $n\ge1$,
 \begin{equation}\label{eq:measurable-functional-calculus}
  \begin{gathered}
   F(T),\ \one_I(T)\quad\text{are measurable operator fields},\\
   \mu\longmapsto\lambda_n(T_\mu)
   \quad\text{is measurable}.
  \end{gathered}
 \end{equation}
\end{enumerate}
\end{lemma}

\begin{proof}
For (i), \eqref{eq:measurable-adjoint-matrix} is immediate.  Since
$J_N(\mu)\to I_{H_\mu}$ strongly by
Lemma~\ref{lem:measurable-Hilbert-field-calculus},
\[
 \langle e_i(\mu),S_\mu T_\mu e_j(\mu)\rangle_{H_\mu}
 =\lim_{N\to\infty}
   \langle S_\mu^*e_i(\mu),J_N(\mu)T_\mu e_j(\mu)\rangle_{H_\mu},
\]
which is \eqref{eq:measurable-product-matrix}.  Lemma~\ref{lem:measurable-Hilbert-field-calculus}
then gives the measurability of $S^*$ and $ST$; the polynomial statement
follows by iteration.

For (ii), fiberwise strong convergence gives, for every $i,j$ and $\mu$,
\[
 \langle e_i(\mu),T_\mu e_j(\mu)\rangle_{H_\mu}
 =\lim_{\ell\to\infty}
   \langle e_i(\mu),T_\mu^{(\ell)}e_j(\mu)\rangle_{H_\mu}.
\]
Thus all matrix coefficients of $T$ are measurable, and
Lemma~\ref{lem:measurable-Hilbert-field-calculus} applies again.

For (iii), compactness together with $J_N(\mu)\to I_{H_\mu}$ strongly gives
\begin{equation*}
 \|T_{\mu,N}-T_\mu\|\longrightarrow0.
\end{equation*}
Each $T_{\mu,N}$ is represented by a finite matrix of measurable
coefficients.  Hence its ordered eigenvalues are measurable, and the min--max
principle yields
\begin{equation}\label{eq:measurable-eigenvalue-compression-limit}
 \lambda_n(T_\mu)
 =\lim_{N\to\infty}\lambda_n(T_{\mu,N}),
 \qquad
 \|T_\mu\|=\lim_{N\to\infty}\|T_{\mu,N}\|.
\end{equation}
In particular, $\mu\mapsto\lambda_n(T_\mu)$ and $\mu\mapsto\|T_\mu\|$ are
measurable.  Therefore the sets
\[
 E_m:=\{\mu:m-1\le\|T_\mu\|<m\},
 \qquad m\ge1,
\]
are measurable.  If $p_{m,\ell}\to F$ uniformly on $[0,m]$, then
\[
 F(T)|_{E_m}
 =\operatorname*{s-lim}_{\ell\to\infty}
   p_{m,\ell}(T)|_{E_m},
\]
so (ii) gives the measurability of $F(T)$.  Finally, set
\[
 F_{I,\ell}(\lambda)
 :=\min\{1,\ell\,\operatorname{dist}(\lambda,I^c)\}.
\]
Then $F_{I,\ell}\uparrow\one_I$ and the spectral theorem gives
$F_{I,\ell}(T_\mu)\to\one_I(T_\mu)$ strongly in every fiber.  A final
application of (ii) proves the measurability of $\one_I(T)$.
\end{proof}

Fix a measurable finite-rank projection field $\Pi=(\Pi_\mu)_\mu$ and a
measurable finite-rank operator field $B=(B_\mu)_\mu$.  For $N,k\ge1$, define
\begin{align}
 d_N^\Pi(\mu)
 &:=\sum_{j=1}^N
   \langle e_j(\mu),\Pi_\mu e_j(\mu)\rangle_{H_\mu},
 \label{eq:measurable-rank-partial-sum}\\
 t_{k,N}^B(\mu)
 &:=\sum_{j=1}^N
   \langle e_j(\mu),B_\mu^k e_j(\mu)\rangle_{H_\mu}.
 \label{eq:measurable-trace-partial-sum}
\end{align}

\begin{lemma}
\label{lem:measurable-finite-rank-invariants}
From the preceding definitions we obtain the following three properties.
\begin{enumerate}[label=\textnormal{(\roman*)},leftmargin=2.1em,itemsep=2pt,topsep=3pt]
 \item For every $\mu$,
 \begin{equation}\label{eq:abstract-measurable-rank}
  d_N^\Pi(\mu)\uparrow\rank\Pi_\mu,
  \qquad
  \mu\longmapsto\rank\Pi_\mu\quad\text{is measurable}.
 \end{equation}

 \item For every $k\ge1$ and every $\mu$,
 \begin{equation}\label{eq:abstract-measurable-trace}
  t_{k,N}^B(\mu)\longrightarrow\Tr(B_\mu^k),
  \qquad
  \mu\longmapsto\Tr(B_\mu^k)\quad\text{is measurable}.
 \end{equation}

 \item Let $Y$ be a standard Borel space.  If the matrix coefficients of
 $\Pi_{\mu,y}$ and $B_{\mu,y}$ are measurable in $(\mu,y)$, then, for every
 $k\ge1$,
 \begin{equation}\label{eq:measurable-finite-rank-joint-map}
  (\mu,y)\longmapsto
  \bigl(\rank\Pi_{\mu,y},\Tr(B_{\mu,y}^k)\bigr)
  \quad\text{is measurable into }\mathbb N_0\times\mathbb R.
 \end{equation}
\end{enumerate}
\end{lemma}

\begin{proof}
For (i), orthonormality of the frame gives
\[
 d_N^\Pi(\mu)
 =\sum_{j=1}^N\|\Pi_\mu e_j(\mu)\|_{H_\mu}^2
 \uparrow\Tr\Pi_\mu=\rank\Pi_\mu.
\]
Each $d_N^\Pi$ is measurable, so the limit is measurable.

For (ii), $B_\mu^k$ is finite-rank and therefore trace class.  Hence
\[
 \sum_{j\ge1}
 \bigl|\langle e_j(\mu),B_\mu^k e_j(\mu)\rangle_{H_\mu}\bigr|
 \le\|B_\mu^k\|_1<\infty,
\]
and the trace formula gives
$t_{k,N}^B(\mu)\to\Tr(B_\mu^k)$.  By
Lemma~\ref{lem:measurable-operator-field-calculus}, each finite sum
$t_{k,N}^B$ is measurable, hence so is its limit.

For (iii), the same finite-sum formulas, now with the parameter $y$, show
that $d_N^\Pi(\mu,y)$ and $t_{k,N}^B(\mu,y)$ are jointly measurable.  Their
pointwise limits are therefore jointly measurable as well.
\end{proof}

\subsubsection{Canonical measurable spectral objects}

The Green operator constructed above is pathwise.  We now apply the preceding
measurable-field tools to the canonical Green operator using the canonical
Borel GMC map of Lemma~\ref{lem:canonical-borel-gmc}.
Define $\GreenOp_x$ by \eqref{eq:green-operator-kh} on
$\Omega_{\rm op}$ and as the zero operator on its complement.  Using
\eqref{eq:measurable-ordered-compact-eigenvalues}, set
\begin{equation}\label{eq:canonical-green-ordered-eigenvalues}
 r_n(x):=\lambda_n(\GreenOp_x),
 \qquad n\ge1.
\end{equation}
Thus $r_1(x)\ge r_2(x)\ge\cdots\ge0$ on all of $\Xspace$.

Recall the sets $\Omega_{\rm HS}$ and $\Omega_{\rm op}$ defined in
\eqref{eq:measurable-HS-event}--\eqref{eq:measurable-operator-event}.

\begin{proposition}
\label{prop:canonical-measurable-liouville-data}
The two sets $\Omega_{\rm HS}$ and $\Omega_{\rm op}$ defined in
\eqref{eq:measurable-HS-event}--\eqref{eq:measurable-operator-event} satisfy
\begin{equation}
 \Omega_{\rm HS},\Omega_{\rm op}\in\mathcal B(\Xspace),
 \qquad
 \mathbb P(\Omega_{\rm op})=1.
 \label{eq:canonical-operator-event-properties}
\end{equation}
Moreover, $x\mapsto\GreenOp_x$ is a measurable field of positive self-adjoint
compact operators, and each
\begin{equation}
 r_n:\Xspace\longrightarrow[0,\infty),
 \qquad n\ge1,
 \label{eq:canonical-green-eigenvalue-borel-map}
\end{equation}
is measurable.  On $\Omega_{\rm op}$,
\[
 \GreenOp_x=A_x^{-1},
 \qquad r_n(x)>0\quad(n\ge1).
\]
\end{proposition}

\begin{proof}
Put
\[
 J_{\rm HS}(x):=\sup_{N\ge1}\iint_{D^2}
       (G_D(z,w)^2\wedge N)\,M_x(dz)M_x(dw).
\]
The functional monotone-class theorem gives
\[
 J_{\rm HS}^{-1}(B)\in\mathcal B(\Xspace)
 \quad\bigl(B\in\mathcal B([0,\infty])\bigr),
 \qquad
 \Omega_{\rm HS}=\{J_{\rm HS}<\infty\},
 \qquad
 \Omega_{\rm op}=\Omega_{\rm tr}\cap\{J_{\rm HS}<\infty\}
 \in\mathcal B(\Xspace).
\]
Since $\Omega_{\rm tr}\in\mathcal B(\Xspace)$ by
Proposition~\ref{prop:borel-trace-form-event}, both sets in
\eqref{eq:canonical-operator-event-properties} are Borel.  In addition,
Lemma~\ref{lem:green-hilbert-schmidt} gives
$\Omega_{\rm Fr}\subseteq\Omega_{\rm HS}$, and hence
\[
 \Omega_{\rm tr}\cap\Omega_{\rm Fr}\subseteq\Omega_{\rm op}.
\]
Together with
$\mathbb P(\Omega_{\rm tr})=\mathbb P(\Omega_{\rm Fr})=1$, this gives
$\mathbb P(\Omega_{\rm op})=1$ and completes
\eqref{eq:canonical-operator-event-properties}.

For every bounded Borel $c$, the map $x\mapsto s_c(M_x):=[c]_{M_x}$ is measurable
because
\[
 \langle s_j(M_x),s_c(M_x)\rangle
 =\int_D\chi_jc\,dM_x
\]
depends measurably on $x$.  Moreover,
\begin{equation}
 \langle s_i(M_x),\GreenOp_xs_j(M_x)\rangle
 =\lim_{m\to\infty}\one_{\Omega_{\rm op}}(x)
   \iint_{D^2}\chi_i(z)(G_D(z,w)\wedge m)\chi_j(w)\,
   M_x(dz)M_x(dw).
 \label{eq:measurable-field-green-matrix}
\end{equation}
The truncated integrals are measurable and, for $x\in\Omega_{\rm op}$,
\eqref{eq:green-kernel-absolute} gives
\[
 \iint_{D^2}|\chi_i(z)\chi_j(w)|G_D(z,w)\,
 M_x(dz)M_x(dw)<\infty.
\]
Hence dominated convergence identifies
\eqref{eq:measurable-field-green-matrix}.  The operator field is measurable by
\eqref{eq:measurable-operator-field-definition}, positive Hilbert--Schmidt by
Lemma~\ref{lem:green-hilbert-schmidt}, and therefore compact; its ordered
eigenvalues are measurable by
Lemma~\ref{lem:measurable-operator-field-calculus}.

For $x\in\Omega_{\rm op}$ and $\theta\in L^2(M_x)$,
Lemma~\ref{lem:green-form-inverse} gives
\[
 \GreenOp_x\theta\in\Dom(A_x),
 \qquad
 A_x\GreenOp_x\theta=\theta.
\]
Conversely, if $u\in\Dom(A_x)$, the same lemma and uniqueness in
\eqref{eq:green-form-inverse-weak} give
$\GreenOp_xA_xu=u$.  Hence $\GreenOp_x=A_x^{-1}$.

Finally, if $x\in\Omega_{\rm op}$, choose pairwise
disjoint non-empty balls $B_j\Subset D$.  Lemma~\ref{lem:gmc-full-support}
and \eqref{eq:green-operator-strict-positive} give
\[
 \left\{\frac{\one_{B_j}}{M_x(B_j)^{1/2}}:j\ge1\right\}
 \text{ orthonormal in }L^2(M_x),
 \qquad
 \ker\GreenOp_x=\{0\}.
\]
A compact injective operator on this infinite-dimensional fiber has
$r_n(x)>0$ for every $n$.
\end{proof}

For Borel completion on all of $\Xspace$, set
\[
 \Lambda_n^x:=
 \begin{cases}
  r_n(x)^{-1},&x\in\Omega_{\rm op},\\
  0,&x\notin\Omega_{\rm op}.
 \end{cases}
\]
The zero value on $\Omega_{\rm op}^c$ is only a measurable completion convention
and does not affect spectral laws or almost-sure assertions.  For
$0<a<b$, put
\begin{equation}\label{eq:measurable-data-window}
 \WinProj_{a,b}^x:=\one_{(1/b,1/a)}(\GreenOp_x).
\end{equation}

\begin{corollary}
\label{cor:canonical-borel-liouville-data}
For every $n\ge1$ and $0<a<b$, the maps
\begin{equation}
 x\longmapsto\Lambda_n^x,
 \qquad
 x\longmapsto\WinProj_{a,b}^x
 \label{eq:canonical-borel-spectral-maps}
\end{equation}
are measurable, and $\WinProj_{a,b}^x$ is a finite-rank projection for every
$x\in\Xspace$.  On $\Omega_{\rm op}$,
\begin{equation}\label{eq:measurable-data-window-identification}
 \GreenOp_x=A_x^{-1},
 \qquad
 \WinProj_{a,b}^x=\one_{(a,b)}(A_x).
\end{equation}
\end{corollary}

\begin{proof}
Proposition~\ref{prop:canonical-measurable-liouville-data} and
Lemma~\ref{lem:measurable-operator-field-calculus} give the measurability in
\eqref{eq:canonical-borel-spectral-maps}.  Compactness of $\GreenOp_x$ makes
$\WinProj_{a,b}^x$ finite-rank, and spectral inversion gives
\eqref{eq:measurable-data-window-identification} on $\Omega_{\rm op}$.
\end{proof}

The remaining measurability required in
Assumption~\ref{ass:spectral-borel} concerns finite-rank compressions by
multiplication operators.

\begin{corollary}
\label{cor:canonical-borel-compression-traces}
For every $0<a<b$, bounded real Borel function $v$, and $k\ge1$, the map
\begin{equation}\label{eq:measurable-compression-traces}
 x\longmapsto
 \Tr\!\left[(\WinProj_{a,b}^x\Mult_v\WinProj_{a,b}^x)^k\right]
 \in\mathbb R
\end{equation}
is measurable.  In particular, for rational $0<a<b$ and $v=q\in\mathcal Q_D$,
these are the power traces appearing in
\eqref{eq:minimal-borel-interface}.
\end{corollary}

\begin{proof}
For bounded real Borel $v$,
\[
 \langle s_i(M_x),\Mult_v s_j(M_x)\rangle
 =\int_D\chi_i v\chi_j\,dM_x,
\]
so $\Mult_v$ is a measurable bounded operator field.  By
Corollary~\ref{cor:canonical-borel-liouville-data} and
Lemma~\ref{lem:measurable-operator-field-calculus},
$\WinProj_{a,b}^x\Mult_v\WinProj_{a,b}^x$ and all its powers are measurable
finite-rank operator fields.  Lemma~\ref{lem:measurable-finite-rank-invariants}
then gives \eqref{eq:measurable-compression-traces}.
\end{proof}

\subsubsection{Compact resolvent}

The measurable Green operator now gives the usual discrete spectral
realization pathwise.

\begin{corollary}
\label{cor:lbm-compact-resolvent}
For every $x\in\Omega_{\rm op}$,
$A_x^{-1}=\GreenOp_x$ is a positive Hilbert--Schmidt operator.  Consequently,
$A_x$ has compact resolvent and
\begin{equation}\label{eq:liouville-spectrum-list}
        0<\Lambda_1^x\le \Lambda_2^x\le\cdots\uparrow\infty,
\end{equation}
where the eigenvalues are repeated according to multiplicity.
\end{corollary}

\begin{proof}
Fix $x\in\Omega_{\rm op}$.
Proposition~\ref{prop:canonical-measurable-liouville-data} and
Lemma~\ref{lem:green-hilbert-schmidt} give
\[
 \GreenOp_x=A_x^{-1}=\GreenOp_x^*\ge0,\qquad
 \norm{\GreenOp_x}_{\rm HS}<\infty,\qquad
 r_n(x)>0\quad(n\ge1).
\]
Hence the compact spectral theorem and inversion yield
\[
 r_n(x)\downarrow0,\qquad
 \Lambda_n^x=r_n(x)^{-1}\uparrow\infty,
\]
which is \eqref{eq:liouville-spectrum-list}.
\end{proof}

For $x\in\Omega_{\rm op}$, choose the eigenfunctions used below real and
normalized so that
\begin{equation}\label{eq:eigen-normalization}
 A_x\phi_n^x=\Lambda_n^x\phi_n^x,
 \qquad
 \int_D(\phi_n^x)^2\,dM_x=1.
\end{equation}
They satisfy
\begin{equation}\label{eq:eigen-weak}
        \E(\phi_n^x,v)
        =\Lambda_n^x\int_D\phi_n^x v\,dM_x,
        \qquad v\in\Vd_x.
\end{equation}

\subsection{Green smoothing and spectral transversality}

Recall the sets $\Omega_{\rm Fr}$ and $\Omega_{\rm op}$ defined in
\eqref{eq:borel-frostman-event} and \eqref{eq:measurable-operator-event},
respectively.  For each $x\in\Omega_{\rm Fr}\cap\Omega_{\rm op}$, Green smoothing gives
\[
 \GreenOp_x:L^2(M_x)\longrightarrow C_b(D).
\]
The almost-sure singularity of $M_h$ and the distributional identity for
eigenfunction squares then give spectral transversality.

\subsubsection{Green smoothing and continuous eigenfunctions}

\begin{proposition}
\label{prop:green-potential-L2-Cb}
Fix $x\in\Omega_{\rm Fr}\cap\Omega_{\rm op}$.  For
$\theta\in L^2(M_x)$, define
\[
 u_\theta(z):=\int_DG_D(z,w)\theta(w)\,M_x(dw),
 \qquad z\in D.
\]
Then $u_\theta\in C_b(D)$ and
\begin{equation}\label{eq:green-potential-Linfty-bound}
 [u_\theta]_{M_x}=\GreenOp_x\theta,
 \qquad
 \|u_\theta\|_\infty
 \le K_G(x)^{1/2}\|\theta\|_{L^2(M_x)}.
\end{equation}
Moreover, the continuous representative is unique:
\[
 v\in C_b(D),\quad [v]_{M_x}=\GreenOp_x\theta
 \quad\Longrightarrow\quad
 v=u_\theta\quad\text{on }D.
\]
\end{proposition}

\begin{proof}
For every $z\in D$, Cauchy--Schwarz gives
\[
 |u_\theta(z)|
 \le\|G_D(z,\cdot)\|_{L^2(M_x)}\|\theta\|_{L^2(M_x)}
 \le K_G(x)^{1/2}\|\theta\|_{L^2(M_x)}.
\]
If $z_n\to z$ in $D$, then all $z_n$ eventually lie in a compact
$K\Subset D$, and
Lemma~\ref{lem:frostman-green-smoothing} gives
\[
 |u_\theta(z_n)-u_\theta(z)|
 \le\|G_D(z_n,\cdot)-G_D(z,\cdot)\|_{L^2(M_x)}
      \|\theta\|_{L^2(M_x)}
 \longrightarrow0.
\]
Thus $u_\theta\in C_b(D)$, and its $L^2(M_x)$ class is
$\GreenOp_x\theta$ by \eqref{eq:green-operator-kh}.
Finally, $x\in\Omega_{\rm op}\subseteq\Omega_{\rm tr}$, so
Lemma~\ref{lem:gmc-full-support} gives $\operatorname{supp}M_x=D$.  If
$v\in C_b(D)$ satisfies $[v]_{M_x}=\GreenOp_x\theta$, then, for every
$\varepsilon>0$,
\[
 \{|v-u_\theta|>\varepsilon\}\text{ is open and has }M_x\text{-mass }0,
\]
hence it is empty.  Therefore $v=u_\theta$ on $D$.
\end{proof}

The inverse representation yields a continuous version of every eigenfunction;
compare \cite[equation~(2.13)]{BerestyckiSpectralGeometry} in the standard LQG
setting.  Fix $x\in\Omega_{\rm Fr}\cap\Omega_{\rm op}$ and a real normalized
eigenpair
\[
 A_x\phi=\Lambda\phi,
 \qquad
 \|\phi\|_{L^2(M_x)}=1,
 \qquad \Lambda>0.
\]
Define its Green representative by
\[
 \phi^{\rm c}(z)
 :=\Lambda\int_DG_D(z,w)\phi(w)\,M_x(dw),
 \qquad z\in D.
\]

\begin{lemma}
\label{lem:green-potential-eigenfunction-regularity}
For the eigenpair above, the following properties hold.
\begin{enumerate}[label=\textnormal{(\roman*)},leftmargin=2.1em,itemsep=2pt,topsep=3pt]
 \item In $H_0^1(D)$,
 \begin{equation}\label{eq:eigenfunction-green-representative}
  \phi=\Lambda\Pot(\phi M_x).
 \end{equation}

 \item The Green representative satisfies
 \[
  \phi^{\rm c}\in C_b(D),
  \qquad
  [\phi^{\rm c}]_{M_x}=[\phi]_{M_x},
  \qquad
  \phi^{\rm c}=\widetilde\phi\quad\E\text{-q.e. on }D.
 \]

 \item Its sup norm obeys
 \begin{equation}\label{eq:eigenfunction-green-linfty-bound}
  \|\phi^{\rm c}\|_\infty\le\Lambda K_G(x)^{1/2}.
 \end{equation}
\end{enumerate}
\end{lemma}

\begin{proof}
Set $w:=\phi-\Lambda\Pot(\phi M_x)$.  By
Lemma~\ref{lem:green-form-inverse} and the eigenvalue equation,
\[
 \E(w,v)=0\qquad(v\in\Vd_x).
\]
Since $w\in\Vd_x$, taking $v=w$ gives $\E(w,w)=0$, hence $w=0$ in
$H_0^1(D)$.  This proves~\textnormal{(i)}.

For~\textnormal{(ii)}, Lemma~\ref{lem:finite-energy-potential} and
\textnormal{(i)} give
\[
 \widetilde\phi(z)
 =\Lambda\int_DG_D(z,w)\phi(w)\,M_x(dw)
 =\phi^{\rm c}(z)
 \qquad\text{for }\E\text{-q.e. }z\in D.
\]
Proposition~\ref{prop:green-potential-L2-Cb}, applied with $\theta=\phi$,
shows that $\phi^{\rm c}\in C_b(D)$ and represents the same $L^2(M_x)$ class
as $\phi$.  The same proposition and $\|\phi\|_{L^2(M_x)}=1$ give
\textnormal{(iii)}.
\end{proof}

Henceforth, eigenfunctions on $\Omega_{\rm Fr}\cap\Omega_{\rm op}$ are
identified with this bounded continuous representative.

\subsubsection{Square-transversality inputs}

For $u,v\in H_0^1(D)$, set
\begin{equation}\label{eq:carre-du-champ-definition}
 \Gamma_\E(u,v)=\frac1{2\pi}\nabla u\cdot\nabla v,
 \qquad \Gamma_\E(u)=\Gamma_\E(u,u).
\end{equation}

The first square-transversality input compares the canonical Liouville measure
with planar Lebesgue measure.  Recall that $M_h=M^{\rm can}(h)$ by the
convention following Lemma~\ref{lem:canonical-borel-gmc}, while $\mathcal L^2$
denotes planar Lebesgue measure.

\begin{lemma}
\label{lem:gmc-singular-lebesgue}
The two measures are almost surely mutually singular:
\begin{equation}\label{eq:gmc-singular-lebesgue}
 \mathbb P\!\left(M_h\perp\mathcal L^2\right)=1.
\end{equation}
\end{lemma}

\begin{proof}
It suffices to construct, on a probability-one event, a Borel set
$N_h\subset D$ such that
\[
 \mathcal L^2(N_h)=0,
 \qquad
 M_h(D\setminus N_h)=0.
\]
Choose balls
\[
 V_j\Subset W_j\Subset U_j\Subset D,
 \qquad D=\bigcup_{j\ge1}V_j,
\]
and let $h_{U_j}$ be the zero-boundary component in the domain-Markov coupling
of Lemma~\ref{lem:interior-dirichlet-gmc-comparison}.  On
$\overline W_j\times\overline W_j$ its covariance has the form
\begin{equation*}
 G_{U_j}(z,w)
 =\log^+\frac{T_j}{|z-w|}+g_j(z,w),
 \qquad g_j\in C_b(\overline W_j\times\overline W_j)
\end{equation*}
for some $T_j>\operatorname{diam}(U_j)$.  The exact-dimensionality theorem for
subcritical GMC \cite[Theorems~4.1 and~4.2]{RhodesVargasGMCReview}, applied on
$W_j$, together with approximation independence
\cite[Corollary~18 and Theorem~25]{ShamovGMC}, gives, for every $j\ge1$,
almost surely,
\[
 \exists\,N_j\in\mathcal B(W_j):\qquad
 M_{h_{U_j}}^{\rm can}(W_j\setminus N_j)=0,
 \qquad
 \dim_{\rm H}N_j=2-\frac{\gamma^2}{2}<2.
\]
In particular,
\[
 \dim_{\rm H}N_j<2
 \quad\Longrightarrow\quad
 \mathcal L^2(N_j)=0.
\]
Equations \eqref{eq:interior-dirichlet-gmc-comparison} and
\eqref{eq:interior-dirichlet-gmc-bounded-density} give a density $w_j$ with
$w_j,w_j^{-1}\in C_b(V_j)$, and hence
\[
 M_h|_{V_j}=w_jM_{h_{U_j}}^{\rm can}|_{V_j}
 \quad\Longrightarrow\quad
 M_h(V_j\setminus N_j)=0.
\]
After intersecting the countably many full-probability events, set
\[
 N_h:=\bigcup_{j\ge1}(N_j\cap V_j)\in\mathcal B(D).
\]
Then
\[
 D\setminus N_h\subseteq\bigcup_{j\ge1}(V_j\setminus N_j),
 \qquad
 M_h(D\setminus N_h)
 \le\sum_{j\ge1}M_h(V_j\setminus N_j)=0,
 \qquad \mathcal L^2(N_h)=0,
\]
which is \eqref{eq:gmc-singular-lebesgue}.
\end{proof}

For $u\in H_0^1(D)$, we use the distributional extension of the background
generator $\Lzero$, namely
\begin{equation}
 \langle\Lzero u,\zeta\rangle=\E(u,\zeta),
 \qquad \zeta\in C_c^\infty(D).
 \label{eq:lbm-energy-laplacian-convention}
\end{equation}

Fix $x\in\Omega_{\rm Fr}\cap\Omega_{\rm op}$ and a real normalized eigenpair
\begin{equation*}
 A_x\phi=\Lambda\phi,
 \qquad
 \|\phi\|_{L^2(M_x)}=1,
 \qquad
 \Lambda>0.
\end{equation*}

\begin{lemma}
\label{lem:eigenfunction-square-laplacian}
The square $\phi^2$ belongs to $H_0^1(D)$ and satisfies
\begin{equation}\label{eq:eigenfunction-square-laplacian}
 \Lzero(\phi^2)
 =2\Lambda\phi^2M_x-2\Gamma_\E(\phi)\,\mathcal L^2
 \qquad\text{in }\mathcal D'(D).
\end{equation}
Equivalently, for every $\zeta\in C_c^\infty(D)$,
\begin{equation}\label{eq:eigenfunction-square-laplacian-test}
 \E(\phi^2,\zeta)
 =2\Lambda\int_D\zeta\phi^2\,dM_x
  -2\int_D\zeta\Gamma_\E(\phi)\,dz.
\end{equation}
\end{lemma}

\begin{proof}
By Lemma~\ref{lem:green-potential-eigenfunction-regularity},
$\phi\in H_0^1(D)\cap L^\infty(D)$.  The Sobolev chain rule gives
\[
 \phi^2\in H_0^1(D),
 \qquad
 \nabla(\phi^2)=2\phi\nabla\phi.
\]
Moreover,
$\zeta\phi\in H_0^1(D)$ and
\[
 \|\zeta\phi\|_{L^2(M_x)}
 \le\|\zeta\|_\infty\|\phi\|_{L^2(M_x)},
\]
so $\zeta\phi\in\Vd_x$.  The chain rule, the product rule, and the eigenvalue
equation now give
\begin{align*}
 \E(\phi^2,\zeta)
 &=2\int_D\phi\Gamma_\E(\phi,\zeta)\,dz\\
 &=2\E(\phi,\zeta\phi)-2\int_D\zeta\Gamma_\E(\phi)\,dz\\
 &=2\Lambda\int_D\zeta\phi^2\,dM_x
   -2\int_D\zeta\Gamma_\E(\phi)\,dz,
\end{align*}
which is \eqref{eq:eigenfunction-square-laplacian-test} and hence
\eqref{eq:eigenfunction-square-laplacian}.
\end{proof}

Since $\mathbb P(h\in\Omega_{\rm Fr}\cap\Omega_{\rm op})=1$,
Lemma~\ref{lem:eigenfunction-square-laplacian} applies pathwise to every real
$L^2(M_h)$-normalized eigenfunction $\phi$.  For each such $\phi$,
\[
 \phi^2M_h\ll M_h,
 \qquad
 \Gamma_\E(\phi)\,\mathcal L^2\ll\mathcal L^2.
\]
Combining these two absolute-continuity relations with
Lemma~\ref{lem:gmc-singular-lebesgue} gives, almost surely,
\[
 \phi^2M_h\perp\Gamma_\E(\phi)\,\mathcal L^2.
\]
Thus the two measures on the right-hand side of
\eqref{eq:eigenfunction-square-laplacian} are mutually singular.  This is the
square-transversality input used in the next subsection.

\subsection{LBM realization and pathwise response}

\subsubsection{LBM realization of the abstract framework}

We first choose the countable direction space and the admissible LBM orbit.
We then identify them with the abstract framework of
Theorem~\ref{thm:intro-abstract-response}.

Choose once and for all a countable set
$\mathcal A_0\subset C_c^\infty(D)$ containing an $H_0^1(D)$-dense subset
and a uniformly dense subset of $C_0(D)$, and set
\[
 \mathcal Q_D:=\operatorname{span}_{\mathbb Q}\mathcal A_0.
\]

\begin{lemma}
\label{lem:lbm-abstract-direction-interface}
The space $\mathcal Q_D$ has the following three properties.
\begin{enumerate}[label=\textnormal{(\roman*)},leftmargin=*,itemsep=0pt,topsep=0.35em,parsep=0pt,partopsep=0pt]
\item $\mathcal Q_D$ is a nonzero countable rational vector space contained in
$C_c^\infty(D)=\Hreg$.
\item It is dense in the Cameron--Martin space:
\[
 \overline{\mathcal Q_D}^{\,H_0^1(D)}=H_0^1(D).
\]
\item Its annihilator among finite signed measures is trivial:
\begin{equation}
 \int_Dq\,d\sigma=0\quad(q\in\mathcal Q_D)
 \quad\Longrightarrow\quad
 \sigma=0,
 \qquad \sigma\in\Mc_{\rm fin}(D).
 \label{eq:lbm-direction-annihilator}
\end{equation}
\end{enumerate}
\end{lemma}

\begin{proof}
Properties \textnormal{(i)} and \textnormal{(ii)} follow directly from the
construction.  For \textnormal{(iii)}, suppose that
$\sigma\in\Mc_{\rm fin}(D)$ annihilates $\mathcal Q_D$.  For $f\in C_0(D)$,
choose $q_k\in\mathcal Q_D$ with $\norm{q_k-f}_\infty\to0$.  Then
\[
 \left|\int_D f\,d\sigma\right|
 \le
 \left|\int_D(f-q_k)\,d\sigma\right|
 +\left|\int_Dq_k\,d\sigma\right|
 \le \norm{f-q_k}_\infty\,|\sigma|(D)
 \longrightarrow0.
\]
Hence $\int_D f\,d\sigma=0$ for every $f\in C_0(D)$, and the Riesz
representation theorem yields $\sigma=0$.
\end{proof}

Fix a uniformly dense sequence $(\psi_j)_{j\ge1}$ in the unit ball of
$C_0(D)$.  Since each integral below is measurable,
\[
 \|M_x-\mathcal L^2\|_{\rm TV}
 =\sup_{j\ge1}
    \left|\int_D\psi_j\,d(M_x-\mathcal L^2)\right|
\]
shows that $x\mapsto\|M_x-\mathcal L^2\|_{\rm TV}$ is measurable.  Moreover,
singularity is characterized by
\[
 M_x\perp\mathcal L^2
 \quad\Longleftrightarrow\quad
 \|M_x-\mathcal L^2\|_{\rm TV}
 =M_x(D)+\mathcal L^2(D).
\]
Consequently,
\[
 \{x\in\Xspace:M_x\perp\mathcal L^2\}
 \in\mathcal B(\Xspace).
\]
Define
\begin{equation}
 \Omega_{\rm LBM}
 :=
   \Omega_{\rm Fr}\cap\Omega_{\rm op}
   \cap\{x:M_x\perp\mathcal L^2\}
 \label{eq:lbm-structural-event-abstract}
\end{equation}
\begin{equation}
 \Omega_{\rm LBM}\in\mathcal B(\Xspace),
 \qquad
 \mathbb P(\Omega_{\rm LBM})=1,
 \qquad
 \Omega_{\rm LBM}
 \subseteq\Omega_{\rm op}
 \subseteq\Omega_{\rm tr}
 \subseteq\Omega_{\rm GMC}.
 \label{eq:lbm-event-inclusions}
\end{equation}
The admissible class is the full regular-tilt orbit
\begin{equation}
\mathcal M_{\rm LBM}
:=\left\{e^{\gamma f}M_x:
 x\in\Omega_{\rm LBM},\ f\in C_c^\infty(D)\right\}.
\label{eq:lbm-admissible-orbit}
\end{equation}
For $\mu=e^{\gamma f}M_x$ in this class, set
$c_f:=e^{\gamma\|f\|_\infty}$.  Then
\begin{equation}
 \begin{gathered}
  c_f^{-1}M_x\le\mu\le c_fM_x,
  \qquad \Vd_\mu=\Vd_x,\\
  c_f^{-1/2}\|u\|_{L^2(M_x)}
  \le \|u\|_{L^2(\mu)}
  \le c_f^{1/2}\|u\|_{L^2(M_x)}.
 \end{gathered}
 \label{eq:lbm-admissible-comparison}
\end{equation}
Proposition~\ref{prop:borel-trace-form-event},
Corollary~\ref{cor:lbm-compact-resolvent}, and
\eqref{eq:lbm-admissible-comparison} transfer closedness, density, and the
compact form embedding from $M_x$ to every admissible $\mu$; full support
makes $L^2(\mu)$ infinite-dimensional.  Finally, for $g\in C_c^\infty(D)$,
\begin{equation}
 e^{\gamma g}\mu=e^{\gamma(f+g)}M_x
 \in\mathcal M_{\rm LBM}.
 \label{eq:lbm-admissible-tilt-closure}
\end{equation}

For $\mu\in\mathcal M_{\rm LBM}$, recall the form norm
\[
 \norm{u}_{\E,\mu}^2=\E(u,u)+\norm{u}_{L^2(\mu)}^2,
 \qquad u\in\Vd_\mu.
\]

The next lemma verifies the multiplier property for the regular directions
$\Hreg=C_c^\infty(D)$.

\begin{lemma}
\label{lem:lbm-regular-form-multipliers}
There exists $C_{\rm mult}(D)<\infty$ such that, for every
$\mu\in\mathcal M_{\rm LBM}$ and $g\in\Hreg$,
\[
 g\Vd_\mu\subset\Vd_\mu,
\]
and, for every $u\in\Vd_\mu$,
\begin{equation}
 \norm{gu}_{\E,\mu}
 \le C_{\rm mult}(D)
 \bigl(\norm{g}_\infty+\norm{\nabla g}_\infty\bigr)
 \norm{u}_{\E,\mu}.
 \label{eq:lbm-form-multiplier-bound}
\end{equation}
\end{lemma}

\begin{proof}
For $u\in\Vd_\mu$, multiplication by $g\in C_c^\infty(D)$ gives
$gu\in H_0^1(D)$, while
\[
 \norm{gu}_{L^2(\mu)}\le\norm{g}_\infty\norm{u}_{L^2(\mu)}.
\]
Hence $gu\in\Vd_\mu$.  The Dirichlet Poincar\'e inequality gives
\begin{align*}
 \E(gu,gu)^{1/2}
 &\le
 \norm{g}_\infty\,\E(u,u)^{1/2}
 +(2\pi)^{-1/2}\norm{\nabla g}_\infty
   \norm{u}_{L^2(D,\mathcal L^2)}\\
 &\le
 C_{\rm mult}(D)\bigl(\norm{g}_\infty+\norm{\nabla g}_\infty\bigr)
 \E(u,u)^{1/2}.
\end{align*}
Combining this estimate with the preceding $L^2(\mu)$ bound proves
\eqref{eq:lbm-form-multiplier-bound}.
\end{proof}

The GFF realization in
\eqref{eq:standing-abstract-wiener-space}--
\eqref{eq:standing-gff-series} is the abstract Wiener space used here.  For LBM,
set
\begin{equation}
 \Omega_{\rm can}:=\Omega_{\rm GMC},
 \qquad
 \Omega_{\rm str}:=\Omega_{\rm LBM},
 \label{eq:lbm-abstract-events}
\end{equation}
and use the identifications
\begin{equation}
 \State=D,\qquad \mu_x=M_x,\qquad
 (\Hc,\Hreg,\kappa,\rho,\mathcal Q)
 =\bigl(H_0^1(D),C_c^\infty(D),\gamma,0,\mathcal Q_D\bigr).
 \label{eq:lbm-abstract-dictionary}
\end{equation}
Under \eqref{eq:lbm-abstract-dictionary}, the abstract operator and spectral
symbols agree with the LBM notation already used in this section.

\subsubsection{LBM pathwise response and transversality}
\label{subsec:lbm-response-consequences}

The constructions above provide the LBM inputs for
Assumptions~\ref{ass:admissible-forms}--\ref{ass:spectral-borel}.  We now turn
to the additional pathwise condition in
Definition~\ref{def:abstract-response-transversality}.  Its verification uses
the square criterion of Theorem~\ref{thm:abstract-square-transversality} and
the LBM response formulas below.  Fix $x\in\Omega_{\rm LBM}$ throughout this
subsubsection.

First fix real $L^2(M_x)$-normalized eigenpairs
$(\Lambda_i,\phi_i)_{i=1}^M$ of $A_x$ with pairwise distinct positive
eigenvalues.  Define
\[
 L:=\Lzero\big|_{\operatorname{span}_{\R}\{\phi_1^2,\ldots,\phi_M^2\}},
 \qquad
 b_i:=2\Gamma_\E(\phi_i)\,\mathcal L^2,\quad 1\le i\le M,
\]
where $L$ is initially understood distributionally.

\begin{corollary}
\label{cor:lbm-square-interface}
The map $L$ takes values in $\Mc_{\rm fin}(D)$, and each $b_i$ is a finite
positive Radon measure.  Moreover,
\begin{equation}
 M_x\perp\mathcal L^2,\qquad \operatorname{supp}M_x=D,\qquad b_i\ll\mathcal L^2,\qquad L(\phi_i^2)=2\Lambda_i\phi_i^2M_x-b_i,\quad 1\le i\le M.
 \label{eq:lbm-square-interface}
\end{equation}
\end{corollary}

\begin{proof}
Since $x\in\Omega_{\rm LBM}$,
\eqref{eq:lbm-structural-event-abstract} gives
$M_x\perp\mathcal L^2$, while
$x\in\Omega_{\rm op}\subseteq\Omega_{\rm tr}$ by
\eqref{eq:lbm-event-inclusions}.  Hence
Lemma~\ref{lem:gmc-full-support} gives
$\operatorname{supp}M_x=D$.

By Lemma~\ref{lem:eigenfunction-square-laplacian},
\[
 \Lzero(\phi_i^2)
 =2\Lambda_i\phi_i^2M_x
  -2\Gamma_\E(\phi_i)\,\mathcal L^2.
\]
Moreover, $\Gamma_\E(\phi_i)\in L^1(D,\mathcal L^2)$, so
$b_i=2\Gamma_\E(\phi_i)\mathcal L^2$ is a finite positive Radon measure with
$b_i\ll\mathcal L^2$.  Thus $L(\phi_i^2)\in\Mc_{\rm fin}(D)$ for every $i$,
which proves \eqref{eq:lbm-square-interface} and the asserted range of $L$.
\end{proof}

We next record the response formulas in LBM notation.  For
$f\in C_c^\infty(D)$ and $\tau\in\R$, set
\[
 M_{x,\tau}^{f}:=e^{\gamma\tau f}M_x,
 \qquad A_{x,\tau}^{f}:=\text{the associated trace-form generator},
\]
and let $\Lambda_n^{x;f}(\tau)$ be the $n$th positive ordered eigenvalue of
$A_{x,\tau}^{f}$.  Define
\begin{equation}
 \partial_f^\pm\Lambda_n^x
 :=
 \lim_{\tau\to0^\pm}
 \frac{\Lambda_n^{x;f}(\tau)-\Lambda_n^x}{\tau}.
 \label{eq:lbm-ordered-directional-derivatives}
\end{equation}
For a positive eigenvalue $\Lambda$ of $A_x$, set
$E_x(\Lambda):=\ker(A_x-\Lambda)$ and
\begin{equation}
 \Response_f^\Lambda(M_x)
 :=-\gamma\Lambda\EigProj_\Lambda^x
       \Mult_f\big|_{E_x(\Lambda)}.
 \label{eq:lbm-response-operator}
\end{equation}

Fix a positive eigenvalue $\Lambda$ of multiplicity $r$, choose a real
$L^2(M_x)$-orthonormal basis $(\phi_i)_{i=1}^r$ of $E_x(\Lambda)$, and fix
$f\in C_c^\infty(D)$.  Proposition~\ref{prop:abstract-cluster-response}
provides $\varepsilon>0$ and real-analytic branches
\[
 \lambda_j:(-\varepsilon,\varepsilon)\to\R,
 \qquad 1\le j\le r,
\]
whose values list, with multiplicity, the eigenvalue cluster of
$A_{x,\tau}^{f}$ issuing from $\Lambda$ for $|\tau|<\varepsilon$.

\begin{proposition}
\label{prop:lbm-cluster-response}
For these branches, the response operator has matrix
\begin{equation}
 \left(\inner{\Response_f^\Lambda(M_x)\phi_j}{\phi_i}
              _{L^2(M_x)}\right)_{i,j=1}^r
 =-\gamma\Lambda
 \left(\int_Df\phi_i\phi_j\,dM_x\right)_{i,j=1}^r,
 \label{eq:lbm-cluster-response-matrix}
\end{equation}
and its eigenvalues are precisely the branch derivatives:
\begin{equation}
 \{\lambda_j'(0):1\le j\le r\}
 =\Spec\bigl(\Response_f^\Lambda(M_x)\bigr)
 \quad\text{as multisets}.
 \label{eq:lbm-cluster-branch-derivatives}
\end{equation}
\end{proposition}

\begin{proof}
The comparison \eqref{eq:lbm-admissible-comparison}, the orbit closure
\eqref{eq:lbm-admissible-tilt-closure}, and
Lemma~\ref{lem:lbm-regular-form-multipliers} place this tilt in the setting of
Proposition~\ref{prop:abstract-cluster-response}.  Since
$\phi_i,\phi_j\in E_x(\Lambda)$,
\[
 \inner{\Response_f^\Lambda(M_x)\phi_j}{\phi_i}_{L^2(M_x)}
 =-\gamma\Lambda\inner{f\phi_j}{\phi_i}_{L^2(M_x)}
 =-\gamma\Lambda\int_D f\phi_i\phi_j\,dM_x.
\]
The branch-derivative conclusion of Proposition~\ref{prop:abstract-cluster-response}
then gives
\[
 \{\lambda_j'(0):1\le j\le r\}
 =\Spec\bigl(\Response_f^\Lambda(M_x)\bigr),
\]
which proves both assertions.
\end{proof}

Assume now that $\Lambda$ occupies the ordered labels
$n,\ldots,n+r-1$, and write
\[
 \beta_1(f)\le\cdots\le\beta_r(f)
\]
for the eigenvalues of $\Response_f^\Lambda(M_x)$.

\begin{corollary}
\label{cor:lbm-ordered-cluster-response}
The ordered one-sided derivatives satisfy
\begin{equation}
 \partial_f^+\Lambda_{n+j-1}^x=\beta_j(f),
 \qquad
 \partial_f^-\Lambda_{n+j-1}^x=\beta_{r-j+1}(f),
 \qquad 1\le j\le r.
 \label{eq:lbm-ordered-cluster-response}
\end{equation}
\end{corollary}

\begin{proof}
This is Proposition~\ref{prop:abstract-ordered-cluster-response} under the
dictionary \eqref{eq:lbm-abstract-dictionary} and the response matrix
\eqref{eq:lbm-cluster-response-matrix}.
\end{proof}

For a simple positive eigenvalue $\Lambda_n^x$, choose a real
$L^2(M_x)$-normalized eigenfunction $\phi_n^x$ and define
\begin{equation}
 \eta_n^x:=-\gamma\Lambda_n^x(\phi_n^x)^2M_x.
 \label{eq:lbm-spectral-response-measure}
\end{equation}
This measure is independent of the sign of $\phi_n^x$.  When $x=h$, we write
$\eta_n^h$.

\begin{corollary}
\label{cor:lbm-simple-eigenvalue-response}
For every $f\in C_c^\infty(D)$,
\begin{equation}
 \partial_f^-\Lambda_n^x
 =\partial_f^+\Lambda_n^x
 =:\partial_f\Lambda_n^x
 =-\gamma\Lambda_n^x\int_Df(\phi_n^x)^2\,dM_x
 =\int_Df\,d\eta_n^x.
 \label{eq:lbm-simple-eigenvalue-response}
\end{equation}
\end{corollary}

\begin{proof}
With $r=1$, Corollary~\ref{cor:lbm-ordered-cluster-response} gives
\[
 \partial_f^-\Lambda_n^x=\partial_f^+\Lambda_n^x
 =\inner{\Response_f^{\Lambda_n^x}(M_x)\phi_n^x}{\phi_n^x}_{L^2(M_x)}.
\]
Using \eqref{eq:lbm-response-operator} and
\eqref{eq:lbm-spectral-response-measure},
\[
 \inner{\Response_f^{\Lambda_n^x}(M_x)\phi_n^x}{\phi_n^x}_{L^2(M_x)}
 =-\gamma\Lambda_n^x\int_Df(\phi_n^x)^2\,dM_x
 =\int_Df\,d\eta_n^x.
\]
\end{proof}

For the square-independence step, fix $N\ge1$ and real
$L^2(M_x)$-normalized eigenpairs $(\Lambda_i,\phi_i)_{i=1}^N$ of $A_x$ with
pairwise distinct positive eigenvalues.

\begin{theorem}
\label{thm:lbm-square-transversality}
For every $a=(a_1,\ldots,a_N)\in\R^N$,
\begin{equation}
 \sum_{i=1}^Na_i\phi_i^2M_x=0
 \quad\Longrightarrow\quad
 a_1=\cdots=a_N=0,
 \label{eq:lbm-square-transversality}
\end{equation}
and
\begin{equation}
 \sum_{i=1}^Na_i\Lambda_i\phi_i^2M_x=0
 \quad\Longrightarrow\quad
 a_1=\cdots=a_N=0.
 \label{eq:lbm-weighted-square-transversality}
\end{equation}
\end{theorem}

\begin{proof}
Apply Corollary~\ref{cor:lbm-square-interface} to the present eigenpairs and
then Theorem~\ref{thm:abstract-square-transversality}.  Since $\rho=0$,
\[
 \Theta_{M_x}(\phi_i,\phi_i)=\phi_i^2M_x.
\]
Thus \eqref{eq:abstract-centered-square-independence} and
\eqref{eq:abstract-weighted-square-independence} are exactly
\eqref{eq:lbm-square-transversality} and
\eqref{eq:lbm-weighted-square-transversality}.
\end{proof}

Finally, choose $1\le n_1<\cdots<n_N$ such that
$\Lambda_{n_1}^x,\ldots,\Lambda_{n_N}^x$ are simple, and use the response
measures from \eqref{eq:lbm-spectral-response-measure}.

\begin{corollary}
\label{cor:lbm-response-transversality}
The measures $\eta_{n_1}^x,\ldots,\eta_{n_N}^x$ are linearly independent.
Consequently, the functionals
\begin{equation*}
  q\longmapsto\int_Dq\,d\eta_{n_i}^x,
  \qquad q\in\mathcal Q_D,\quad 1\le i\le N,
\end{equation*}
are linearly independent, and the family
$(\Lambda_{n_i}^x)_{i=1}^N$ is response-transverse in the sense of
Definition~\ref{def:abstract-response-transversality}.
\end{corollary}

\begin{proof}
For $a=(a_1,\ldots,a_N)\in\R^N$,
\[
 \sum_{i=1}^Na_i\eta_{n_i}^x=0
 \quad\Longrightarrow\quad
 \sum_{i=1}^Na_i\Lambda_{n_i}^x(\phi_{n_i}^x)^2M_x=0,
\]
so \eqref{eq:lbm-weighted-square-transversality} gives $a=0$.  If
\[
 \sum_{i=1}^Na_i\int_Dq\,d\eta_{n_i}^x=0
 \qquad(q\in\mathcal Q_D),
\]
then \eqref{eq:lbm-direction-annihilator} gives
$\sum_i a_i\eta_{n_i}^x=0$, and hence again $a=0$.
\end{proof}

\subsection{Final proof of Theorem~\ref{thm:intro-lbm-spectrum}}

Use the identifications
\eqref{eq:lbm-abstract-events}--\eqref{eq:lbm-abstract-dictionary}.  The five
assumptions of Theorem~\ref{thm:intro-abstract-response} are verified as
follows.
\begin{itemize}[label=$\bullet$]
 \item Assumption~\ref{ass:admissible-forms}.  Given
 $\mu\in\mathcal M_{\rm LBM}$, choose $x\in\Omega_{\rm LBM}$ and
 $f\in C_c^\infty(D)$ such that $\mu=e^{\gamma f}M_x$.  Then
 \eqref{eq:lbm-admissible-comparison} gives
 \[
  \Vd_\mu=\Vd_x,
  \qquad
  \|\cdot\|_{L^2(\mu)}\asymp
  \|\cdot\|_{L^2(M_x)}.
 \]
 By \eqref{eq:lbm-event-inclusions},
 $x\in\Omega_{\rm op}\subseteq\Omega_{\rm tr}$.
 Proposition~\ref{prop:borel-trace-form-event} therefore gives a closed densely
 defined trace form for $M_x$, while
 Corollary~\ref{cor:lbm-compact-resolvent} gives
 \[
  \|u\|_{\E,M_x}^2
  =\sum_{n\ge1}(1+\Lambda_n^x)
    |\langle u,\phi_n^x\rangle_{L^2(M_x)}|^2,
  \qquad \Lambda_n^x\uparrow\infty.
 \]
 Hence $\Vd_x\hookrightarrow L^2(M_x)$ is compact.  The comparison in
 \eqref{eq:lbm-admissible-comparison} transfers this compact embedding to
 $\mu$.  Full support makes $L^2(\mu)$ infinite-dimensional.

 \item Assumption~\ref{ass:regular-form-multipliers}.  By
 Lemma~\ref{lem:lbm-regular-form-multipliers}, for
 $g\in\Hreg=C_c^\infty(D)$,
 \[
  g\Vd_\mu\subset\Vd_\mu,
  \qquad
  \|gu\|_{\E,\mu}
  \le C_{\rm mult}(D)
  \bigl(\|g\|_\infty+\|\nabla g\|_\infty\bigr)\|u\|_{\E,\mu}.
 \]

 \item Assumption~\ref{ass:coherent-orbit}.
 Lemma~\ref{lem:canonical-borel-gmc},
 \eqref{eq:lbm-event-inclusions}, and
 \eqref{eq:lbm-abstract-events} give Borel events satisfying
 \[
  \Omega_{\rm str}=\Omega_{\rm LBM}
  \subseteq\Omega_{\rm GMC}=\Omega_{\rm can},
  \qquad \mathbb P(\Omega_{\rm str})=1.
 \]
 By \eqref{eq:lbm-admissible-orbit},
 $M_x\in\mathcal M_{\rm LBM}$ for $x\in\Omega_{\rm str}$, so
 \eqref{eq:abstract-structural-data} holds.  Moreover,
 Theorem~\ref{thm:gmc-cm-shift} verifies the orbit identity
 \eqref{eq:pointwise-abstract-orbit}:
 \[
  M_{x+g}=e^{\gamma g}M_x,
  \qquad g\in\Hreg,\quad x,x+g\in\Omega_{\rm GMC}.
 \]
 Finally, \eqref{eq:lbm-admissible-tilt-closure} is the closure clause
 \eqref{eq:abstract-admissible-orbit-closure} of this assumption.

 \item Assumption~\ref{ass:direction-separation}.  Lemma
 \ref{lem:lbm-abstract-direction-interface} provides a nonzero countable
 rational space
 $\mathcal Q_D\subset C_c^\infty(D)=\Hreg$, dense in $H_0^1(D)$.  Together
 with $\rho=0$ from \eqref{eq:lbm-abstract-dictionary}, this verifies
 \eqref{eq:countable-direction-space}.  The same lemma gives
 \[
  \int_Dq\,d\sigma=0\quad(q\in\mathcal Q_D)
  \quad\Longrightarrow\quad \sigma=0,
  \qquad \sigma\in\Mc_{\rm fin}(D).
 \]
 Since $\rho=0$, this is \eqref{eq:direction-annihilator}.

 \item Assumption~\ref{ass:spectral-borel}.
 Corollaries~\ref{cor:canonical-borel-liouville-data} and
 \ref{cor:canonical-borel-compression-traces} verify
 \eqref{eq:minimal-borel-interface} on
 $\Omega_{\rm str}=\Omega_{\rm LBM}$ for
 $n,k\ge1$, rational $0<a<b$, and $q\in\mathcal Q_D$.
\end{itemize}
Corollary~\ref{cor:lbm-compact-resolvent} and
Theorem~\ref{thm:intro-abstract-response}\textnormal{(i)} give, almost surely,
\[
 0<\Lambda_1^h<\Lambda_2^h<\cdots\uparrow\infty,
\]
which is \eqref{eq:intro-lbm-simple-spectrum}.

Fix $1\le n_1<\cdots<n_N$.  On this event, Corollaries
\ref{cor:lbm-simple-eigenvalue-response} and
\ref{cor:lbm-response-transversality}, evaluated at $x=h$, give
\begin{align*}
 \ell_{n_i}^{M_h}(q)
 &=-\gamma\Lambda_{n_i}^h\int_Dq(\phi_{n_i}^h)^2\,dM_h
   =\int_Dq\,d\eta_{n_i}^h,
   \qquad q\in\mathcal Q_D,\\
 \sum_{i=1}^N a_i\ell_{n_i}^{M_h}(q)=0\quad(q\in\mathcal Q_D)
 &\Longrightarrow a_1=\cdots=a_N=0.
\end{align*}
Thus Definition~\ref{def:abstract-response-transversality} holds, and
Theorem~\ref{thm:intro-abstract-response}\textnormal{(ii)} yields
\eqref{eq:intro-lbm-density}.  This proves
Theorem~\ref{thm:intro-lbm-spectrum}.  The stronger potential realization in
Definition~\ref{def:response-realization} is established separately in
Appendix~\ref{app:lbm-finite-energy-response}.

\section{The Liouville--Cauchy operator}
\label{sec:lcp-realization}

This section verifies the abstract response framework for the
Liouville--Cauchy operator and establishes the nonlocal eigenfunction-square
identity used for response transversality.

\subsection{The circle Cauchy form and boundary chaos}
Recall that $\LCircle=\R/(2\pi\mathbb Z)$ and
$m(d\theta)=d\theta/(2\pi)$ is normalized Haar measure.
Write $d_{\LCircle}$ for the geodesic distance and
\[
 B_{\LCircle}(\theta,r)
 :=\{\varphi\in\LCircle:d_{\LCircle}(\theta,\varphi)<r\}.
\]

For $u\in L^2(m)$ and for a finite signed Radon measure $\sigma$ on
$\LCircle$, use the Fourier conventions
\[
 \widehat u(n)=\int_{\LCircle}u(\theta)e^{-in\theta}\,m(d\theta),
 \qquad
 \widehat\sigma(n)=\int_{\LCircle}e^{-in\theta}\,\sigma(d\theta).
\]

All forms and Hilbert spaces below are taken over $\mathbb R$; complex Fourier
notation is used only to express their standard sesquilinear complexifications.
The mean-zero Cameron--Martin space is
\begin{equation}\label{eq:H-half-definition}
 \Hc=\dot H^{1/2}(\LCircle)
 :=\left\{u\in L^2(m):\widehat u(0)=0,\ 
       \sum_{n\ne0}|n|\,|\widehat u(n)|^2<\infty\right\},
 \qquad
 \norm{u}_{\Hc}^2=\sum_{n\ne0}|n|\,|\widehat u(n)|^2.
\end{equation}

On $H^{1/2}(\LCircle)=\Hc\oplus\R\one$, the circle Cauchy form is
\begin{align}
 \E(u,v)
 &=\sum_{n\ne0}|n|\widehat u(n)\overline{\widehat v(n)}
 \label{eq:cauchy-form-fourier}\\
 &=\frac12\iint_{\LCircle\times\LCircle}
   (u(\theta)-u(\varphi))(v(\theta)-v(\varphi))
   J(\theta,\varphi)\,m(d\theta)m(d\varphi),
 \label{eq:cauchy-form-jump}
\end{align}
where
\begin{equation}\label{eq:cauchy-jump-kernel}
 J(\theta,\varphi)=\frac{2}{|e^{i\theta}-e^{i\varphi}|^2}.
\end{equation}
For $u\in H^{1/2}(\LCircle)$, $\widetilde u$ denotes a fixed
quasi-continuous representative for this Dirichlet form.

The associated un-time-changed generator is the Fourier multiplier
\begin{equation}\label{eq:cauchy-generator-fourier}
 A_0e^{in\theta}=|n|e^{in\theta},
 \qquad n\in\mathbb Z.
\end{equation}

The inverse of the Fourier multiplier $|n|$ on mean-zero functions is
represented off the diagonal by
\begin{equation}\label{eq:circle-pseudo-green}
 G(\theta,\varphi)
 =\sum_{n\ne0}\frac{e^{in(\theta-\varphi)}}{|n|}
 =-2\log\left(2\left|\sin\frac{\theta-\varphi}{2}\right|\right),
 \qquad \theta\ne\varphi.
\end{equation}
Set $G(\theta,\theta)=+\infty$; the resulting extended kernel is measurable
with respect to $\mathcal B(\LCircle^2)$.  It is
normalized by
\begin{equation}\label{eq:pseudo-green-zero-average}
 \int_{\LCircle}G(\theta,\varphi)m(d\varphi)=0.
\end{equation}

\paragraph{The canonical boundary chaos.}

Fix $s>0$ and realize the mean-zero circle field as the coordinate map
\[
 h:\Xspace_{\rm C}\longrightarrow\Xspace_{\rm C},
 \qquad h(x)=x,
 \qquad \Xspace_{\rm C}=\dot H^{-s}(\LCircle),
\]
with Gaussian law $\mathbb P_{\rm C}$.  Its covariance is the kernel
\eqref{eq:circle-pseudo-green}, its Cameron--Martin space is
\eqref{eq:H-half-definition}, and
\begin{equation}
 \kappa:=\frac\gamma2,
 \qquad 0<\gamma<2.
 \label{eq:boundary-beta}
\end{equation}
We retain $\gamma$ in the GMC estimates and use $\kappa=\gamma/2$ in the
abstract tilt and response formulas.
Let $P_\eps=e^{-\eps A_0}$ be the Poisson semigroup on mean-zero
distributions and set
\begin{equation}
 M_{x,\eps}^{\partial}(d\theta)
 :=\exp\left(
   \kappa P_\eps x(\theta)
   -\frac{\kappa^2}{2}
       \mathbb E_{\rm C}[(P_\eps h(\theta))^2]
  \right)m(d\theta).
 \label{eq:lcp-regularized-gmc}
\end{equation}

\begin{lemma}
\label{lem:lcp-canonical-boundary-gmc}
There exist a deterministic sequence $\eps_j\downarrow0$, a Borel set
\[
 \Omega_{\rm GMC}^{\partial}\subset\Xspace_{\rm C},
 \qquad
 \mathbb P_{\rm C}(\Omega_{\rm GMC}^{\partial})=1,
\]
and a measurable map
\begin{equation}
 \Xspace_{\rm C}\ni x\longmapsto M_x^{\partial}
 \in\Mc_{\rm fin}^+(\LCircle).
 \label{eq:lcp-canonical-borel-gmc}
\end{equation}
They satisfy the following properties.
\begin{enumerate}[label=\textnormal{(\roman*)},leftmargin=*,itemsep=0pt,topsep=0.35em,parsep=0pt,partopsep=0pt]
\item For every $x\in\Omega_{\rm GMC}^{\partial}$,
\[
 M_{x,\eps_j}^{\partial}\Longrightarrow M_x^{\partial},
\]
where $\Longrightarrow$ denotes narrow convergence of finite measures.
\item For the coordinate field $h$, $M_h^\partial$ agrees almost surely with
subcritical boundary GMC.
\item If $f\in C^\infty(\LCircle)\cap\Hc$ and
$x,x+f\in\Omega_{\rm GMC}^{\partial}$, then
\begin{equation}
 M_{x+f}^{\partial}=e^{\kappa f}M_x^{\partial}.
 \label{eq:boundary-gmc-coherent-shift}
\end{equation}
\end{enumerate}
\end{lemma}

\begin{proof}
For every fixed $\eps>0$, the Poisson semigroup is smoothing:
\[
 P_\eps:\dot H^{-s}(\LCircle)\longrightarrow C^\infty(\LCircle)
\]
is continuous.  Hence the map $x\mapsto M_{x,\eps}^{\partial}$ is Borel when
$\Mc_{\rm fin}^+(\LCircle)$ is equipped with the narrow topology.

Write
\[
 K_\eps^\partial(\theta,\varphi)
 :=\Cov(P_\eps h(\theta),P_\eps h(\varphi))
 =\sum_{n\ne0}\frac{e^{-2\eps|n|}e^{in(\theta-\varphi)}}{|n|}.
\]
Then $K_\eps^\partial\to G$ off the diagonal and in
$L^1(\LCircle^2,m\otimes m)$.  Moreover, for $0<\eps\le1$,
\[
 K_\eps^\partial(\theta,\varphi)
 \le C+2\log^+\frac1{d_{\LCircle}(\theta,\varphi)\vee\eps},
 \qquad \theta,\varphi\in\LCircle,
\]
and $P_\eps f\to f$ in $\Hc$ for every $f\in\Hc$.

It remains to record the uniform integrability required by the approximation
theorem.  Choose $q\in(1,4/\gamma^2)$.  Cover $\LCircle$ by finitely many
arcs $I_1,\ldots,I_R$ of length less than $1/2$.  On each arc, Kahane's
convexity inequality and the one-dimensional subcritical positive-moment
estimate, applied with effective GMC parameter $\gamma/\sqrt2$, give a
bound uniform in the exact-scale cutoff (the cutoff masses form an
$L^q$-bounded martingale):
\[
 \sup_{0<\eps\le1}
 \mathbb E_{\rm C}\!\left[M_{h,\eps}^{\partial}(I_r)^q\right]<\infty,
 \qquad 1\le r\le R;
\]
see \cite[Theorems~2.1 and~2.11]{RhodesVargasGMCReview}.  Since the cover is
finite,
\[
 \sup_{0<\eps\le1}
 \mathbb E_{\rm C}\!\left[
   M_{h,\eps}^{\partial}(\LCircle)^q\right]<\infty.
\]
Thus the total masses are uniformly integrable.  The covariance convergence
in measure, Cameron--Martin convergence, and uniform integrability now verify
the hypotheses of Shamov's approximation theorem
\cite[Theorem~25]{ShamovGMC}; uniqueness
\cite[Corollary~18]{ShamovGMC} identifies the limit with the subcritical
boundary GMC.  Consequently $M_{h,\eps}^{\partial}$ converges in probability,
for the narrow topology, to that GMC.

Let $d_{\rm w}$ be a complete compatible metric on
$\Mc_{\rm fin}^+(\LCircle)$ and choose $\eps_j\downarrow0$ so that
\[
 \mathbb P_{\rm C}\!\left(
 d_{\rm w}(M_{h,\eps_{j+1}}^\partial,M_{h,\eps_j}^\partial)>2^{-j}
 \right)\le2^{-j}.
\]
By Borel--Cantelli, for all sufficiently large $k>j$,
\[
 d_{\rm w}(M_{h,\eps_k}^\partial,M_{h,\eps_j}^\partial)
 \le\sum_{\ell=j}^{k-1}2^{-\ell}
 \le2^{1-j}\longrightarrow0.
\]
Therefore
\[
 \Omega_{\rm GMC}^{\partial}
 :=\left\{x\in\Xspace_{\rm C}:
 (M_{x,\eps_j}^{\partial})_{j\ge1}
 \text{ is Cauchy in }d_{\rm w}\right\}
 \in\mathcal B(\Xspace_{\rm C}),
 \qquad
 \mathbb P_{\rm C}(\Omega_{\rm GMC}^{\partial})=1.
\]
Completeness of $d_{\rm w}$ permits the definition
\[
 M_x^{\partial}
 :=\begin{cases}
 \displaystyle\lim_{j\to\infty}M_{x,\eps_j}^{\partial},
   &x\in\Omega_{\rm GMC}^{\partial},\\[1ex]
 0,&x\notin\Omega_{\rm GMC}^{\partial}.
 \end{cases}
\]
The resulting map is measurable, being a pointwise limit of measurable maps
on the Borel convergence event and constant on its complement.  This proves
\eqref{eq:lcp-canonical-borel-gmc}.

For smooth $f$, $P_{\eps_j}f\to f$ uniformly and
\[
 M_{x+f,\eps_j}^{\partial}
 =e^{\kappa P_{\eps_j}f}M_{x,\eps_j}^{\partial}.
\]
Hence, for every $\psi\in C(\LCircle)$,
\[
 \int_{\LCircle}\psi\,dM_{x+f}^{\partial}
 =\lim_{j\to\infty}\int_{\LCircle}\psi\,dM_{x+f,\eps_j}^{\partial}
 =\lim_{j\to\infty}\int_{\LCircle}
   \psi e^{\kappa P_{\eps_j}f}\,dM_{x,\eps_j}^{\partial}
 =\int_{\LCircle}\psi e^{\kappa f}\,dM_x^{\partial}.
\]
This proves \eqref{eq:boundary-gmc-coherent-shift} whenever both endpoints
belong to $\Omega_{\rm GMC}^{\partial}$.
\end{proof}

For the coordinate field $h$, we henceforth write $M_h^\partial$ for the
evaluation at $x=h$ of the map \eqref{eq:lcp-canonical-borel-gmc}.

The canonical boundary chaos admits a single structural event on which all
properties needed below hold simultaneously.

\begin{proposition}
\label{prop:lcp-canonical-structural-event}
There exists a Borel set
\[
 \Omega_{\rm str}^{\partial}\subset\Omega_{\rm GMC}^{\partial},
 \qquad
 \mathbb P_{\rm C}(\Omega_{\rm str}^{\partial})=1,
\]
with the following properties for every $x\in\Omega_{\rm str}^{\partial}$.
\begin{enumerate}[label=\textnormal{(\roman*)},leftmargin=*,itemsep=0pt,topsep=0.35em,parsep=0pt,partopsep=0pt]
\item The measure $M_x^\partial$ is finite, nonzero, has full support, is
singular with respect to $m$, and satisfies
\begin{equation}
 M_x^{\partial}(B_{\LCircle}(\theta,r))
 \le C_xr^{\alpha_x},
 \qquad \theta\in\LCircle,\quad 0<r\le1,
 \label{eq:standing-boundary-gmc-event}
\end{equation}
for some $C_x<\infty$ and $\alpha_x>0$.
\item The measure $M_x^\partial$ is smooth with full quasi-support for the
Cauchy form.  Its time-change form is $(\E,\Vd_x)$ with
\begin{equation}
 \Vd_x
 =\{u\in H^{1/2}(\LCircle):\widetilde u\in L^2(M_x^{\partial})\}.
 \label{eq:lcp-trace-form}
\end{equation}
\item The conclusions in \textnormal{(i)}--\textnormal{(ii)} remain valid
with $M_x^\partial$ replaced by $e^{\kappa f}M_x^\partial$ for every
$f\in C^\infty(\LCircle)\cap\Hc$.
\end{enumerate}
\end{proposition}

\begin{proof}
For the Frostman bound in \eqref{eq:standing-boundary-gmc-event}, Kahane
comparison and the subcritical positive-moment estimate imply that, for
$1<q<4/\gamma^2$ and every arc $I$ of length at most $1/2$,
\begin{equation*}
 \mathbb E_{\rm C}[M_h^{\partial}(I)^q]
 \le C_q|I|^{\zeta_\gamma(q)},
 \qquad
 \zeta_\gamma(q)
 =\left(1+\frac{\gamma^2}{4}\right)q
  -\frac{\gamma^2}{4}q^2.
\end{equation*}
Equivalently, the logarithmic strength of $\kappa h$ is that of the usual
one-dimensional GMC with parameter $\gamma/\sqrt2$; see
\cite[Theorem~2.11]{RhodesVargasGMCReview}.  Since
$\zeta_\gamma(1)=1$ and
$\zeta_\gamma'(1)=1-\gamma^2/4>0$, choose $q>1$ and
\[
 0<\alpha<\frac{\zeta_\gamma(q)-1}{q}.
\]
If $\mathcal D_n$ is the family of $2^n$ dyadic arcs, Markov's inequality
gives
\[
 \mathbb P_{\rm C}\left(
  \max_{I\in\mathcal D_n}M_h^{\partial}(I)>2^{-\alpha n}
 \right)
 \le C_q2^{-n(\zeta_\gamma(q)-1-\alpha q)}.
\]
The right-hand side is summable.  Hence, almost surely for all large $n$,
$M_h^{\partial}(I)\le2^{-\alpha n}$ for every $I\in\mathcal D_n$.  If
$2^{-(n+1)}<r\le2^{-n}$, each $B_{\LCircle}(\theta,r)$ is covered by only a bounded number
of arcs in $\mathcal D_n$, and therefore
\[
 M_h^{\partial}(B_{\LCircle}(\theta,r))
 \lesssim 2^{-\alpha n}
 \le 2^{\alpha}r^{\alpha}.
\]
This proves the Frostman bound in \eqref{eq:standing-boundary-gmc-event}.

For every deterministic non-empty closed arc $I$ and every $p>0$, the
subcritical negative-moment theorem
\cite[Theorem~2.12]{RhodesVargasGMCReview}, in a local circle chart with
parameter $\gamma/\sqrt2$, gives
\[
 \mathbb E_{\rm C}\!\left[M_h^{\partial}(I)^{-p}\right]<\infty
 \quad\Longrightarrow\quad
 \mathbb P_{\rm C}\!\left(M_h^{\partial}(I)=0\right)=0.
\]
Intersecting these events over a countable family of closed arcs such that
every non-empty open arc contains one proves
$\supp M_h^{\partial}=\LCircle$.  The thick-point description
\cite[Theorems~4.1--4.2]{RhodesVargasGMCReview}, again with parameter
$\gamma/\sqrt2$, places full $M_h^{\partial}$-mass on a set of Hausdorff
dimension $1-\gamma^2/4<1$, and hence $M_h^{\partial}\perp m$.

Baverez's PCAF construction
\cite[Theorem~4.1.1, Definition~4.1.1, and
Theorem~4.3.14]{BaverezThesis} shows that boundary GMC is a smooth Revuz
measure of full quasi-support for the Cauchy form.  The time-change theorem
\cite[Theorem~6.2.1 and equation~(6.2.22)]{FukushimaOshimaTakeda} then gives
\eqref{eq:lcp-trace-form}.

By subcritical GMC uniqueness, this chaos agrees almost surely with the
canonical version of Lemma~\ref{lem:lcp-canonical-boundary-gmc}.  Let $E_*$ be
the intersection of the completed full events for the preceding conclusions
and this identification.  Choose a Borel full core
\[
 \Omega_{\rm str}^{\partial}
 \in\mathcal B(\Xspace_{\rm C}),
 \qquad
 \Omega_{\rm str}^{\partial}\subset
 \Omega_{\rm GMC}^{\partial}\cap E_*,
 \qquad
 \mathbb P_{\rm C}(\Omega_{\rm str}^{\partial})=1.
\]
For smooth mean-zero $f$, the density $e^{\kappa f}$ is bounded above and below
by positive constants.  Thus all the preceding properties persist under
tilting, and the same time-change theorem gives the tilted trace form.
\end{proof}

For $x\in\Omega_{\rm str}^{\partial}$, let $A_x$ denote the
non-negative self-adjoint operator on $L^2(M_x^\partial)$ associated with
$(\E,\Vd_x)$.

\paragraph{All-parameter coherent Cameron--Martin lines.}
For $x\in\Omega_{\rm str}^{\partial}$ and
$f\in C^\infty(\LCircle)\cap\Hc$, define
\begin{equation}
 M_{x,\tau}^{\partial,f}
 :=e^{\kappa\tau f}M_x^\partial,
 \qquad \tau\in\mathbb R.
 \label{eq:lcp-coherent-line-gmc}
\end{equation}
This is an all-parameter family of finite, mutually equivalent measures and
\begin{equation}
 M_{x,\tau}^{\partial,f}
 =e^{\kappa(\tau-r)f}
  M_{x,r}^{\partial,f},
 \qquad r,\tau\in\mathbb R.
 \label{eq:lcp-coherent-line-cocycle}
\end{equation}
Every member retains the structural properties in
\eqref{eq:standing-boundary-gmc-event}--\eqref{eq:lcp-trace-form}.  For the
coordinate field $h$ and each fixed $\tau\in\mathbb R$, Cameron--Martin
equivalence and \eqref{eq:boundary-gmc-coherent-shift} give
\begin{equation}
 M_{h,\tau}^{\partial,f}=M_{h+\tau f}^{\partial}
 \qquad\mathbb P_{\rm C}\text{-almost surely}.
 \label{eq:lcp-canonical-coherent-fixed-parameter}
\end{equation}
The last identity is not asserted on one event simultaneously for all
$\tau$; the family \eqref{eq:lcp-coherent-line-gmc} is the version used in
all pathwise analytic arguments.

\begin{remark}
\label{rem:lcp-no-unit-volume-normalization}
No unit-volume normalization is imposed here.  The measure $M_x^{\partial}$
is the unnormalized GMC of the centered trace field; we never divide by
$M_x^{\partial}(\LCircle)$.  Indeed, the probability
normalization
$\overline M_x^{\partial}=M_x^{\partial}/M_x^{\partial}(\LCircle)$ would obey
\[
 \overline M_{x+f}^{\partial}
 =\frac{e^{\kappa f}}
        {\int_{\LCircle}e^{\kappa f}
          \,d\overline M_x^{\partial}}
   \overline M_x^{\partial},
\]
not the coherent orbit \eqref{eq:boundary-gmc-coherent-shift}.  References
which impose unit boundary length may still be used for smoothness and full
quasi-support, because multiplication by a positive pathwise scalar changes
only the time scale and preserves null sets and quasi-support.
\end{remark}

For every $x\in\Omega_{\rm str}^{\partial}$, the Fourier representation of the
form, smoothness, and full quasi-support give the kernel identity
\begin{equation}
 u\in\ker A_x
 \quad\Longleftrightarrow\quad
 u\in\Vd_x\ \text{and}\ \E(u,u)=0
 \quad\Longleftrightarrow\quad
 u=c\one\quad M_x^\partial\text{-almost everywhere}.
 \label{eq:lcp-kernel-constants}
\end{equation}
Indeed, vanishing energy identifies $u$ with a constant in
$H^{1/2}(\LCircle)$, hence quasi-everywhere for the chosen quasi-continuous
representatives.  Smoothness then transfers this equality to
$M_x^\partial$-almost everywhere, while the converse is immediate.

For any finite positive measure $\mu$ on $\LCircle$ with
$0<\mu(\LCircle)<\infty$, write
\begin{equation}
 L_0^2(\mu)
 :=\left\{f\in L^2(\mu):
 \int_{\LCircle}f\,d\mu=0\right\}.
 \label{eq:l2-zero-mean-M}
\end{equation}
Two elementary estimates will be used repeatedly.  First, multiplication by
a smooth function preserves $H^{1/2}(\LCircle)$.  For
$a\in C^1(\LCircle)$, the jump representation and
$|a(\theta)-a(\varphi)|\le C\norm{a'}_\infty
 |e^{i\theta}-e^{i\varphi}|$ give
\begin{equation}\label{eq:H-half-multiplier-bound}
 \norm{au}_{H^{1/2}(\LCircle)}
 \le C\norm{a}_{C^1}\norm{u}_{H^{1/2}(\LCircle)}.
\end{equation}
Second, we use the pathwise trace--Poincar\'e bound
\begin{equation}\label{eq:lcp-trace-poincare-bound}
 \norm{u}_{L^2(m)}^2
 \le C_x^{\rm P}\left(\E(u,u)+\norm{u}_{L^2(M_x^\partial)}^2\right),
 \qquad u\in\Vd_x.
\end{equation}
Suppose, to the contrary, that \eqref{eq:lcp-trace-poincare-bound} fails.
Then there are $u_k$ with $\norm{u_k}_{L^2(m)}=1$ and the right-hand side
tending to zero.  If
$c_k=\int_{\LCircle}u_k\,dm$, Fourier--Poincar\'e gives
$u_k-c_k\to0$ in $H^{1/2}(\LCircle)$.  Moreover,
\[
 1=\|u_k\|_{L^2(m)}^2
 =|c_k|^2+\|u_k-c_k\|_{L^2(m)}^2,
\]
so, after a subsequence, $c_k\to c$ with $|c|=1$.  A further subsequence
converges quasi-everywhere to $c$, hence
$M_x^\partial$-almost everywhere because $M_x^\partial$ is smooth.  Fatou's
lemma and $M_x^\partial(\LCircle)>0$ then give
\[
 |c|^2M_x^\partial(\LCircle)
 \le\liminf_{k\to\infty}\norm{u_k}_{L^2(M_x^\partial)}^2=0,
\]
contradicting $|c|=1$.

Equations \eqref{eq:H-half-multiplier-bound}--\eqref{eq:lcp-trace-poincare-bound}
show that every smooth multiplier is bounded for the form norm.  Consequently
$e^{-\kappa f/2}$ is an invertible multiplier and depends real analytically on
$f$ in every finite-dimensional smooth slice.  Hence the usual fixed-space
conjugation applies to \eqref{eq:lcp-trace-form}.  The remaining model-specific
work concerns the nonlocal inputs.

\subsection{Centered Green compactness and smoothing}

Fix $x\in\Omega_{\rm str}^{\partial}$ throughout this subsection.  Put
$Z_x^\partial:=M_x^\partial(\LCircle)$ and define
\begin{align}
 \bar G_x(\theta)
 &:=\frac1{Z_x^\partial}\int_{\LCircle}G(\theta,z)\,M_x^\partial(dz),
 \label{eq:centered-green-gh}\\
 c_x
 &:=\frac1{(Z_x^\partial)^2}
   \iint_{\LCircle\times\LCircle}G(z,w)\,M_x^\partial(dz)M_x^\partial(dw),
 \label{eq:centered-green-ch}\\
 G_x^\circ(\theta,\varphi)
 &:=G(\theta,\varphi)-\bar G_x(\theta)-\bar G_x(\varphi)+c_x.
 \label{eq:centered-green-kernel}
\end{align}
Then $G_x^\circ$ is symmetric and
\begin{equation}\label{eq:centered-green-kills-constants}
 \int_{\LCircle}G_x^\circ(\theta,\varphi)M_x^\partial(d\varphi)=0.
\end{equation}

Define the centered Green operator by
\begin{equation}\label{eq:lcp-centered-green-operator}
 \GreenOp_x f(\theta)
 :=\int_{\LCircle}G_x^\circ(\theta,\varphi)f(\varphi)M_x^\partial(d\varphi).
\end{equation}
Set
\[
 K_G^\partial(x)
 :=\sup_{\theta\in\LCircle}
 \int_{\LCircle}|G_x^\circ(\theta,\varphi)|^2M_x^\partial(d\varphi)
 \in[0,\infty].
\]

\begin{lemma}
\label{lem:lcp-centered-green-smoothing}
The centered kernel and Green operator have the following properties.
\begin{enumerate}[label=\textnormal{(\roman*)},leftmargin=*,itemsep=0pt,topsep=0.35em,parsep=0pt,partopsep=0pt]
\item $G_x^\circ\in L^2(M_x^\partial\otimes M_x^\partial)$ and
\begin{equation}
 \|\GreenOp_x\|_{\rm HS}^2
 =\iint_{\LCircle^2}|G_x^\circ(\theta,\varphi)|^2
   M_x^\partial(d\theta)M_x^\partial(d\varphi)<\infty.
 \label{eq:lcp-green-HS}
\end{equation}
\item The operator $\GreenOp_x$ is self-adjoint,
\[
 \GreenOp_x\one=0,
 \qquad
 \GreenOp_x\bigl(L_0^2(M_x^\partial)\bigr)
 \subseteq L_0^2(M_x^\partial).
\]
\item $K_G^\partial(x)<\infty$, and for every $f\in L^2(M_x^\partial)$,
\begin{equation}\label{eq:lcp-green-C-bound}
 \GreenOp_xf\in C(\LCircle),
 \qquad
 \norm{\GreenOp_x f}_\infty
 \le K_G^\partial(x)^{1/2}\norm{f}_{L^2(M_x^\partial)}.
\end{equation}
\end{enumerate}
\end{lemma}

\begin{proof}
The explicit kernel \eqref{eq:circle-pseudo-green} gives
\begin{equation}\label{eq:circle-green-log-bound}
 |G(\theta,\varphi)|
 \le C\left(1+\log^+\frac1{d_{\LCircle}(\theta,\varphi)}\right).
\end{equation}
Write $C_x$ and $\alpha_x$ for the Frostman constants in
\eqref{eq:standing-boundary-gmc-event}.  For
$L_\theta(\varphi)=\log^+d_{\LCircle}(\theta,\varphi)^{-1}$, layer cake and
\eqref{eq:standing-boundary-gmc-event} give
\begin{equation}\label{eq:uniform-log-square-bound}
 \sup_{\theta\in\LCircle}\int_{\LCircle}L_\theta(\varphi)^2M_x^\partial(d\varphi)
 \le 2C_x\int_0^\infty t e^{-\alpha_x t}\,dt
    +M_x^\partial(\LCircle)<\infty.
\end{equation}
The same calculation, beginning at $\log(1/r)$, gives
\begin{equation}\label{eq:uniform-local-log-square-tail}
 \lim_{r\downarrow0}\sup_{\theta\in\LCircle}
 \int_{B_{\LCircle}(\theta,r)}
 \left(1+L_\theta(\varphi)^2\right)M_x^\partial(d\varphi)=0.
\end{equation}

If $\theta_k\to\theta$, fix $r>0$ and, for all sufficiently large $k$, put
\[
 N_k=B_{\LCircle}(\theta,2r)\cup B_{\LCircle}(\theta_k,2r).
\]
Since $d_{\LCircle}(\theta_k,\theta)<r$, the set $N_k$ is contained both in
$B_{\LCircle}(\theta,3r)$ and in $B_{\LCircle}(\theta_k,3r)$.  Hence
\begin{align*}
 &\int_{N_k}|G(\theta_k,\varphi)-G(\theta,\varphi)|^2
       M_x^\partial(d\varphi)\\
 &\le
 2\int_{B_{\LCircle}(\theta_k,3r)}
     |G(\theta_k,\varphi)|^2M_x^\partial(d\varphi)
 +2\int_{B_{\LCircle}(\theta,3r)}
     |G(\theta,\varphi)|^2M_x^\partial(d\varphi),
\end{align*}
which tends to zero uniformly in $k$ as $r\downarrow0$ by
\eqref{eq:circle-green-log-bound} and
\eqref{eq:uniform-local-log-square-tail}.  On $N_k^c$, both variables stay at distance at least
$2r$ from the diagonal, so uniform continuity of $G$ gives convergence to
zero for fixed $r$.  Letting first $k\to\infty$ and then $r\downarrow0$
proves
\begin{equation}\label{eq:green-row-L2-continuity}
 \norm{G(\theta_k,\cdot)-G(\theta,\cdot)}_{L^2(M_x^\partial)}\longrightarrow0.
\end{equation}

In particular,
\[
 \bar G_x(\theta)
 =\frac1{Z_x^\partial}
   \inner{G(\theta,\cdot)}{\one}_{L^2(M_x^\partial)}
\]
is continuous, and hence bounded on the compact circle.  Subtracting the
bounded continuous one-variable terms in
\eqref{eq:centered-green-kernel} preserves both the uniform square bound and
the row-$L^2$ continuity.  Thus
\eqref{eq:uniform-log-square-bound}--\eqref{eq:green-row-L2-continuity}
hold with $G_x^\circ$ in place of $G$.  This proves
$G_x^\circ\in L^2(M_x^\partial\otimes M_x^\partial)$, the Hilbert--Schmidt
assertion, and \eqref{eq:lcp-green-C-bound}.  For
$f\in L^2(M_x^\partial)$ and
$\theta_k\to\theta$,
\[
 |\GreenOp_x f(\theta_k)-\GreenOp_x f(\theta)|
 \le
 \|G_x^\circ(\theta_k,\cdot)-G_x^\circ(\theta,\cdot)\|_{L^2(M_x^\partial)}
 \|f\|_{L^2(M_x^\partial)}\longrightarrow0,
\]
so the image is continuous.

Finally, symmetry of $G_x^\circ$, Fubini's theorem, and
\eqref{eq:centered-green-kills-constants} give
\[
 \inner{\GreenOp_x f}{g}_{L^2(M_x^\partial)}
 =\inner{f}{\GreenOp_x g}_{L^2(M_x^\partial)},
 \qquad
 \GreenOp_x\one=0,
\]
and hence, for $f\in L_0^2(M_x^\partial)$,
\[
 \int_{\LCircle}\GreenOp_x f\,dM_x^\partial
 =\inner{f}{\GreenOp_x\one}_{L^2(M_x^\partial)}=0.
\]
Thus the operator is self-adjoint, kills constants, and preserves
$L_0^2(M_x^\partial)$.
\end{proof}

\begin{proposition}
\label{prop:lcp-centered-green-inverse}
For every $f\in L_0^2(M_x^\partial)$, the function
$u:=\GreenOp_xf$ belongs to $\Vd_x$ and satisfies
\begin{equation}\label{eq:lcp-green-form-inverse}
 \E(u,v)=\int_{\LCircle}f\widetilde v\,dM_x^\partial,
 \qquad v\in\Vd_x.
\end{equation}
It is the unique solution of \eqref{eq:lcp-green-form-inverse} with zero
$M_x^\partial$-mean.  Moreover,
\[
 \GreenOp_x:L_0^2(M_x^\partial)\longrightarrow L_0^2(M_x^\partial)
\]
is positive and injective, and
\[
 \GreenOp_x
 =\left(A_x\big|_{L_0^2(M_x^\partial)}\right)^{-1}.
\]
Consequently, the positive spectrum of $A_x$ is discrete.
\end{proposition}

\begin{proof}
Let $f\in L_0^2(M_x^\partial)$ and set $\sigma=fM_x^\partial$.  Then
$\sigma(\LCircle)=0$.  The logarithmic kernel estimate
\eqref{eq:circle-green-log-bound} and the uniform square bound
\eqref{eq:uniform-log-square-bound}, followed by Cauchy--Schwarz, imply
\begin{equation}\label{eq:lcp-finite-green-energy-fM}
 \iint_{\LCircle\times\LCircle}|G(\theta,\varphi)|
       |f(\theta)f(\varphi)|M_x^\partial(d\theta)M_x^\partial(d\varphi)<\infty.
\end{equation}

Using the Abel-regularized kernels
\[
 G_r(\theta,\varphi)
 =\sum_{n\ne0}\frac{r^{|n|}e^{in(\theta-\varphi)}}{|n|},
 \qquad 0<r<1,
\]
we have, uniformly for $1/2\le r<1$,
\[
 |G_r(\theta,\varphi)|
 \le C\left(1+\log^+\frac1{d_{\LCircle}(\theta,\varphi)}\right).
\]
Thus \eqref{eq:lcp-finite-green-energy-fM} permits dominated convergence on
the integral side, whereas Parseval's identity gives the monotone scalar
series
\[
 \iint_{\LCircle\times\LCircle}G_r(\theta,\varphi)
       \sigma(d\theta)\sigma(d\varphi)
 =\sum_{n\ne0}\frac{r^{|n|}|\widehat\sigma(n)|^2}{|n|}
 \uparrow
 \sum_{n\ne0}\frac{|\widehat\sigma(n)|^2}{|n|}.
\]
Letting $r\uparrow1$ therefore gives the finite-energy identity
\begin{equation}\label{eq:lcp-fourier-energy-measure}
 \iint_{\LCircle\times\LCircle}G(\theta,\varphi)\sigma(d\theta)\sigma(d\varphi)
 =\sum_{n\ne0}\frac{|\widehat\sigma(n)|^2}{|n|}.
\end{equation}

Set
\[
 \Potential\sigma(\theta)
 :=\int_{\LCircle}G(\theta,\varphi)\sigma(d\varphi),
 \qquad
 u_r(\theta):=\int_{\LCircle}G_r(\theta,\varphi)\sigma(d\varphi).
\]
Then $u_r\in C^\infty(\LCircle)$, while
\begin{align}
 \norm{u_r-\Potential\sigma}_{L^2(M_x^\partial)}
 &\le
 \norm{G_r-G}_{L^2(M_x^\partial\otimes M_x^\partial)}\norm{f}_{L^2(M_x^\partial)}
 \longrightarrow0,
 \label{eq:lcp-abel-potential-L2}\\
 \E(u_r-u_s,u_r-u_s)
 &=\sum_{n\ne0}
   \frac{(r^{|n|}-s^{|n|})^2}{|n|}|\widehat\sigma(n)|^2
 \longrightarrow0,
 \qquad r,s\uparrow1.
 \label{eq:lcp-abel-potential-energy}
\end{align}
Dominated convergence gives both limits, using the uniform logarithmic square bound
for the first and \eqref{eq:lcp-fourier-energy-measure} for the second.  Thus
$(u_r)$ is Cauchy in the closed form norm; \eqref{eq:lcp-abel-potential-L2}
identifies its limit with $\Potential\sigma$, so $\Potential\sigma\in\Vd_x$.  Finally,
\eqref{eq:green-row-L2-continuity} gives $\Potential\sigma\in C(\LCircle)$, hence
$\Potential\sigma\in\Hc\cap C(\LCircle)\cap\Vd_x$.

For a trigonometric polynomial $v$, Fourier inversion gives
\begin{equation}\label{eq:lcp-potential-core-identity}
 \E(\Potential\sigma,v)=\int_{\LCircle}v\,d\sigma,
 \qquad
 \left|\int_{\LCircle}v\,d\sigma\right|
 \le \left(\sum_{n\ne0}\frac{|\widehat\sigma(n)|^2}{|n|}\right)^{1/2}
       \E(v,v)^{1/2}.
\end{equation}
If $v\in C^\infty(\LCircle)$ and
\[
 S_Nv:=\sum_{|n|\le N}\widehat v(n)e^{in\theta},
\]
then
\[
 \norm{S_Nv-v}_{H^{1/2}(\LCircle)}+\norm{S_Nv-v}_\infty
 \longrightarrow0.
\]
Since $M_x^\partial$ and $|\sigma|$ are finite, this convergence is also in
the time-changed form norm and in $L^1(|\sigma|)$, respectively.  Passing to
the limit in \eqref{eq:lcp-potential-core-identity} therefore proves that
identity for every smooth $v$.

The full-quasi-support time-change theorem
\cite[Theorem~6.2.1 and equation~(6.2.22)]{FukushimaOshimaTakeda} makes the
time-changed form regular on $\LCircle$, so
$\Vd_x\cap C(\LCircle)$ is a form core.  If
$w\in\Vd_x\cap C(\LCircle)$, its Fej\'er polynomials converge to $w$
both uniformly and in $H^{1/2}(\LCircle)$; since $M_x^\partial$ is finite,
they therefore converge in the time-changed form norm.  Thus trigonometric
polynomials are a form core.  Accordingly, for $v\in\Vd_x$, choose such
polynomials
$v_k\to v$ in the form norm.  Then
\[
 \left|\int_{\LCircle}(\widetilde v_k-\widetilde v)f\,dM_x^\partial\right|
 \le\norm{f}_{L^2(M_x^\partial)}
      \norm{v_k-v}_{L^2(M_x^\partial)}\longrightarrow0,
\]
while the left side of \eqref{eq:lcp-potential-core-identity} converges by
energy.
Thus the quasi-continuous representative is $|\sigma|$-integrable and
\begin{equation*}
 \E(\Potential\sigma,v)=\int_{\LCircle}\widetilde v\,d\sigma,
 \qquad v\in\Vd_x.
\end{equation*}

The centered operator differs from $\Potential(fM_x^\partial)$ only by the constant which
makes its $M_x^\partial$-mean zero.  Hence, for every $v\in\Vd_x$,
\[
 \E(\GreenOp_x f,v)
 =\E(\Potential(fM_x^\partial),v)
 =\int_{\LCircle}f\widetilde v\,dM_x^\partial,
\]
because constants have zero $\E$-energy.  This is
\eqref{eq:lcp-green-form-inverse}.  By the definition of the form generator,
\[
 A_x(\GreenOp_x f)=f,
 \qquad f\in L_0^2(M_x^\partial).
\]

Conversely, if $u\in\Dom(A_x)\cap L_0^2(M_x^\partial)$, then
$A_xu\in L_0^2(M_x^\partial)$ because
$\inner{A_xu}{\one}_{L^2(M_x^\partial)}=\E(u,\one)=0$.  The weak identity gives
\[
 \E(u-\GreenOp_x(A_xu),v)=0,
 \qquad v\in\Vd_x.
\]
Put $w=u-\GreenOp_x(A_xu)$.  Taking $v=w$ in the preceding weak
identity and using $w\in L_0^2(M_x^\partial)$ gives
\[
 \E(w,w)=0,
 \qquad \int_{\LCircle}w\,dM_x^\partial=0
 \quad\Longrightarrow\quad w=0,
\]
so $\GreenOp_x(A_xu)=u$.

For positivity and injectivity,
\[
 \inner{f}{\GreenOp_x f}_{L^2(M_x^\partial)}
 =\sum_{n\ne0}\frac{|\widehat\sigma(n)|^2}{|n|}\ge0,
 \qquad
 \inner{f}{\GreenOp_x f}=0
 \Longrightarrow \widehat\sigma\equiv0
 \Longrightarrow \sigma=0\Longrightarrow f=0,
\]
where $\widehat\sigma(0)=\sigma(\LCircle)=0$.  Thus $\GreenOp_x$ is
positive and injective.  The inverse identification and compactness give
discreteness of the positive spectrum.
\end{proof}

Fix a real normalized positive eigenpair
\[
 A_x\phi=\Lambda\phi,
 \qquad
 \|\phi\|_{L^2(M_x^\partial)}=1,
 \qquad
 \Lambda>0,
\]
and define
\[
 \phi^{\rm c}:=\Lambda\GreenOp_x\phi.
\]

\begin{corollary}
\label{cor:lcp-continuous-eigenfunctions}
The eigenfunction is orthogonal to constants, and $\phi^{\rm c}$ is its
unique continuous representative:
\begin{equation}\label{eq:lcp-eigenfunction-green-representation}
 \int_{\LCircle}\phi\,dM_x^\partial=0,
 \qquad
 \phi^{\rm c}\in C(\LCircle),
 \qquad
 [\phi^{\rm c}]_{M_x^\partial}=[\phi]_{M_x^\partial}.
\end{equation}
Moreover,
\[
 \|\phi^{\rm c}\|_\infty
 \le\Lambda K_G^\partial(x)^{1/2}.
\]
\end{corollary}

\begin{proof}
The eigenvalue equation gives
\[
 0=\E(\phi,\one)
 =\Lambda\inner{\phi}{\one}_{L^2(M_x^\partial)}.
\]
Hence $\phi\in L_0^2(M_x^\partial)$, and
Proposition~\ref{prop:lcp-centered-green-inverse} gives
$[\phi^{\rm c}]_{M_x^\partial}=[\phi]_{M_x^\partial}$.
Lemma~\ref{lem:lcp-centered-green-smoothing} gives continuity and the stated
sup-norm bound.  Full support of $M_x^\partial$ makes the continuous
representative unique.
\end{proof}

Henceforth, positive eigenfunctions on $\Omega_{\rm str}^\partial$ are
identified with this continuous representative.

\subsection{Measurable spectrum and compression traces}

For the all-space measurable completion, fix a countable $\mathbb Q$-algebra
$\{\psi_j:j\ge0\}\subset C(\LCircle)$ which is uniformly dense in
$C(\LCircle)$ and contains $\psi_0=\one$.  Set
\[
 \mu_x:=
 \begin{cases}
  M_x^{\partial},&x\in\Omega_{\rm GMC}^{\partial},\\
  m,&x\notin\Omega_{\rm GMC}^{\partial},
 \end{cases}
 \qquad
 H_x:=L^2(\mu_x).
\]
For $x\in\Omega_{\rm str}^{\partial}$, first record the total mass and the
one-variable Green average:
\begin{align}
 Z_x&:=\mu_x(\LCircle),\\
 \bar G_x(\theta)
 &:=\frac1{Z_x}\int_{\LCircle}G(\theta,z)\,\mu_x(dz).
\end{align}
The remaining scalar average completes the centering of the Green kernel:
\begin{align}
 c_x
 &:=\frac1{Z_x^2}\iint_{\LCircle^2}G(z,w)\,\mu_x(dz)\mu_x(dw),\\
 G_x^\circ(\theta,\varphi)
 &:=G(\theta,\varphi)-\bar G_x(\theta)-\bar G_x(\varphi)+c_x.
 \label{eq:lcp-borel-centered-green-kernel}
\end{align}
The corresponding centered Green operator is
\begin{equation}
 \GreenOp_xu(\theta)
 :=\int_{\LCircle}G_x^\circ(\theta,\varphi)u(\varphi)\,\mu_x(d\varphi).
 \label{eq:lcp-borel-centered-green-operator}
\end{equation}
On $\Omega_{\rm str}^{\partial}$ these agree with the pathwise centered Green
objects in \eqref{eq:centered-green-gh}--\eqref{eq:lcp-centered-green-operator}.
Off $\Omega_{\rm str}^{\partial}$ set $G_x^\circ=0$ and $\GreenOp_x=0$ on $H_x$.
For $n\ge1$, let $r_n(x)$ be the $n$th positive eigenvalue of $\GreenOp_x$ on
$\Omega_{\rm str}^{\partial}$, counted with multiplicity, and set
$r_n(x)=0$ otherwise.

\begin{proposition}
\label{lem:lcp-borel-spectral-interface}
The family $(H_x,\GreenOp_x)_x$ is a measurable field of separable Hilbert
spaces and compact self-adjoint operators, and each
\[
 r_n:\Xspace_{\rm C}\longrightarrow[0,\infty)
\]
is measurable.  On $\Omega_{\rm str}^{\partial}$,
\[
 H_x=\operatorname{span}\{\one\}\oplus L_0^2(M_x^\partial),
 \qquad
 \ker\GreenOp_x=\operatorname{span}\{\one\},
\]
and the restriction of $\GreenOp_x$ to $L_0^2(M_x^\partial)$ is positive,
injective, and equal to
$\bigl(A_x|_{L_0^2(M_x^\partial)}\bigr)^{-1}$.  In particular,
$r_n(x)>0$ for every $n\ge1$ on this event.
\end{proposition}

\begin{proof}
The map
\[
 x\longmapsto\mu_x:
 (\Xspace_{\rm C},\mathcal B(\Xspace_{\rm C}))
 \longrightarrow
 (\Mc_{\rm fin}^+(\LCircle),\mathcal B(\Mc_{\rm fin}^+(\LCircle)))
\]
is measurable.  Parameterized integration, first for the truncations
$G_N=(-N)\vee(G\wedge N)$ and then by passage to the limit using the
logarithmic estimate and the Frostman bound, makes
$(x,\theta,\varphi)\mapsto G_x^\circ(\theta,\varphi)$ jointly measurable.
Moreover,
\[
 \|\GreenOp_x\|_{\rm HS}^2
 =\iint_{\LCircle^2}|G_x^\circ(\theta,\varphi)|^2
   \mu_x(d\theta)\mu_x(d\varphi)<\infty
\]
on $\Omega_{\rm str}^{\partial}$.

The functions $s_j(x):=\psi_j\in H_x$ generate a measurable Hilbert field,
since
\[
 x\longmapsto\langle s_i(x),s_j(x)\rangle_{H_x}
 =\int_{\LCircle}\psi_i\psi_j\,d\mu_x
\]
is measurable.  Likewise,
\[
 x\longmapsto
 \langle s_i(x),\GreenOp_xs_j(x)\rangle_{H_x}
 =\iint_{\LCircle^2}
 \psi_i(\theta)G_x^\circ(\theta,\varphi)\psi_j(\varphi)
 \,\mu_x(d\theta)\mu_x(d\varphi)
\]
is measurable.  Hence $\GreenOp_x$ is a measurable operator field.

On $\Omega_{\rm str}^{\partial}$,
Lemma~\ref{lem:lcp-centered-green-smoothing} and
Proposition~\ref{prop:lcp-centered-green-inverse} give the stated
self-adjointness, kernel, positivity, injectivity, and inverse relation.
Full support makes $L_0^2(M_x^\partial)$ infinite-dimensional.  Therefore the
restriction has infinitely many positive eigenvalues.  The measurable-field
tools of Subsubsection~\ref{subsubsec:measurable-field-tools}, applied with
$D$ replaced by $\LCircle$ and $(\chi_j)$ by $(\psi_j)$, give the
measurability of every $r_n$ by
Lemma~\ref{lem:measurable-operator-field-calculus}.
\end{proof}

For Borel completion on all of $\Xspace_{\rm C}$, define
\begin{equation}
 \Lambda_n^x
 :=\begin{cases}
  r_n(x)^{-1},&x\in\Omega_{\rm str}^{\partial},\\
  0,&x\notin\Omega_{\rm str}^{\partial},
 \end{cases}
 \label{eq:lcp-borel-positive-eigenvalues}
\end{equation}
and, for $0<a<b$,
\[
 \WinProj_{a,b}^x
 :=\one_{(1/b,1/a)}(\GreenOp_x).
\]

\begin{corollary}
\label{cor:lcp-borel-positive-spectrum}
For every $n\ge1$ and $0<a<b$, the maps
\[
 x\longmapsto\Lambda_n^x,
 \qquad
 x\longmapsto\WinProj_{a,b}^x
\]
are measurable, and $\WinProj_{a,b}^x$ is finite-rank.  On
$\Omega_{\rm str}^{\partial}$,
\[
 \Lambda_n^x\text{ is the $n$th positive eigenvalue of }A_x,
 \qquad
 \WinProj_{a,b}^x=\one_{(a,b)}(A_x).
\]
\end{corollary}

\begin{proof}
Proposition~\ref{lem:lcp-borel-spectral-interface} and
Lemma~\ref{lem:measurable-operator-field-calculus} give the measurability.
The inverse relation in Proposition~\ref{prop:lcp-centered-green-inverse}
gives the two pathwise identifications on $\Omega_{\rm str}^{\partial}$.
\end{proof}

The remaining measurability required in
Assumption~\ref{ass:spectral-borel} concerns finite-rank compressions by
multiplication operators.

\begin{corollary}
\label{cor:lcp-borel-compression-traces}
For every $0<a<b$, bounded real Borel function $v$ on $\LCircle$, and
$k\ge1$, the map
\begin{equation}
 x\longmapsto
 \Tr\!\left[(\WinProj_{a,b}^x\Mult_v\WinProj_{a,b}^x)^k\right]
 \in\mathbb R
 \label{eq:lcp-minimal-borel-spectral-interface}
\end{equation}
is measurable.
\end{corollary}

\begin{proof}
For bounded real Borel $v$,
\[
 \langle s_i(x),\Mult_vs_j(x)\rangle_{H_x}
 =\int_{\LCircle}\psi_i v\psi_j\,d\mu_x,
\]
so $\Mult_v$ is a measurable bounded operator field.  By
Corollary~\ref{cor:lcp-borel-positive-spectrum} and
Lemma~\ref{lem:measurable-operator-field-calculus},
$\WinProj_{a,b}^x\Mult_v\WinProj_{a,b}^x$ and its powers are measurable
finite-rank operator fields.  Lemma~\ref{lem:measurable-finite-rank-invariants}
then gives \eqref{eq:lcp-minimal-borel-spectral-interface}.
\end{proof}

For the coordinate field, $r_n(h)$ and $\Lambda_n^h$ denote the evaluations
at $x=h$ of these measurable functions.

\subsection{The nonlocal eigenfunction-square identity}

For real $u\in H^{1/2}(\LCircle)$, define its jump energy density by
\begin{equation}\label{eq:lcp-jump-energy-density}
 \Gamma_J(u)(\theta)
 :=\int_{\LCircle}(u(\theta)-u(\varphi))^2
       J(\theta,\varphi)m(d\varphi).
\end{equation}
Then $\Gamma_J(u)\in L^1(m)$ and
\begin{equation}\label{eq:lcp-jump-energy-total}
 \int_{\LCircle}\Gamma_J(u)\,dm=2\E(u,u).
\end{equation}
For $u\in H^{1/2}(\LCircle)$, the distribution $A_0u$ is understood through
\begin{equation}
 \langle A_0u,\zeta\rangle:=\E(u,\zeta),
 \qquad \zeta\in C^\infty(\LCircle).
 \label{eq:lcp-A0-distribution-convention}
\end{equation}
If $g\in L^1(m)$, the notation $gm$ denotes the measure whose action is
$\int_{\LCircle}\zeta g\,dm$.

Fix $x\in\Omega_{\rm str}^{\partial}$ and a real normalized eigenpair
\[
 A_x\phi=\Lambda\phi,
 \qquad
 \|\phi\|_{L^2(M_x^\partial)}=1,
 \qquad
 \Lambda>0.
\]

\begin{lemma}
\label{lem:lcp-nonlocal-square-identity}
The square $\phi^2$ belongs to $\Vd_x$ and satisfies
\begin{equation}\label{eq:lcp-nonlocal-square-measure-decomposition}
 A_0(\phi^2)
 =2\Lambda\phi^2M_x^\partial-\Gamma_J(\phi)m
 \qquad\text{in }\mathcal D'(\LCircle).
\end{equation}
Both terms on the right are finite Radon measures, and
$\Gamma_J(\phi)m\ll m$.  Equivalently, for every real
$\zeta\in C^\infty(\LCircle)$,
\begin{equation}\label{eq:lcp-nonlocal-square-identity}
 \E(\phi^2,\zeta)
 =2\Lambda\int_{\LCircle}\zeta\phi^2\,dM_x^\partial
  -\int_{\LCircle}\zeta\,\Gamma_J(\phi)\,dm.
\end{equation}
\end{lemma}

\begin{proof}
By Corollary~\ref{cor:lcp-continuous-eigenfunctions}, $\phi$ is bounded.  The
standard bounded product property of the Cauchy Dirichlet form (equivalently,
the Lipschitz functional calculus applied on the bounded range of $\phi$)
gives $\phi^2\in H^{1/2}(\LCircle)$; boundedness and finiteness of
$M_x^\partial$ give
$\phi^2\in L^2(M_x^\partial)$.  Hence $\phi^2\in\Vd_x$.  Likewise,
$\zeta\phi\in\Vd_x$.

For real numbers $a,b,c,d$,
\begin{equation*}
 2(a-b)(ac-bd)-(a^2-b^2)(c-d)=(a-b)^2(c+d).
\end{equation*}
Apply this with
$(a,b,c,d)=(\phi(\theta),\phi(\varphi),
\zeta(\theta),\zeta(\varphi))$, integrate against
$\frac12J(\theta,\varphi)m(d\theta)m(d\varphi)$, and use symmetry in
$(\theta,\varphi)$.  The result is
\begin{equation}\label{eq:lcp-square-form-product-rule}
 \E(\phi^2,\zeta)
 =2\E(\phi,\phi\zeta)
  -\int_{\LCircle}\zeta\,\Gamma_J(\phi)\,dm.
\end{equation}
The weak eigenvalue equation with test function $\phi\zeta$ gives
\[
 \E(\phi,\phi\zeta)
 =\Lambda\int_{\LCircle}\zeta\phi^2\,dM_x^\partial.
\]
Substitution in \eqref{eq:lcp-square-form-product-rule} proves
\eqref{eq:lcp-nonlocal-square-identity}.  The two terms on the right are finite
Radon measures by boundedness of $\phi$, finiteness of $M_x^\partial$, and
\eqref{eq:lcp-jump-energy-total}; this also proves the measure formulation.
\end{proof}

\subsection{Liouville--Cauchy realization and transversality}

\subsubsection{Liouville--Cauchy realization of the abstract framework}

Choose the countable rational direction space
\begin{equation}
 \mathcal Q_{\rm C}
 :=\operatorname{span}_{\mathbb Q}
 \{\cos(n\theta),\sin(n\theta):n\ge1\}.
 \label{eq:lcp-concrete-Q}
\end{equation}

\begin{lemma}
\label{lem:lcp-abstract-direction-interface}
The space $\mathcal Q_{\rm C}$ has the following three properties.
\begin{enumerate}[label=\textnormal{(\roman*)},leftmargin=*,itemsep=0pt,topsep=0.35em,parsep=0pt,partopsep=0pt]
\item $\mathcal Q_{\rm C}$ is a nonzero countable rational vector space contained
in $\Hreg=C^\infty(\LCircle)\cap\Hc$.
\item It is dense in the Cameron--Martin space:
\[
 \overline{\mathcal Q_{\rm C}}^{\,\Hc}=\Hc.
\]
\item Its annihilator among finite signed measures is exactly
\begin{equation}
 \left\{\sigma\in\Mc_{\rm fin}(\LCircle):
 \int_{\LCircle}q\,d\sigma=0\text{ for every }q\in\mathcal Q_{\rm C}\right\}
 =\operatorname{span}\{m\}.
 \label{eq:lcp-direction-annihilator}
\end{equation}
\end{enumerate}
\end{lemma}

\begin{proof}
Properties \textnormal{(i)} and \textnormal{(ii)} follow from the Fourier
basis of $\dot H^{1/2}(\LCircle)$.  For \textnormal{(iii)}, if $\sigma$
annihilates $\mathcal Q_{\rm C}$, then all nonzero Fourier coefficients of
$\sigma$ vanish.  Hence $\sigma-\sigma(\LCircle)m$ has every Fourier
coefficient equal to zero, and Fourier uniqueness for finite measures gives
$\sigma=\sigma(\LCircle)m$.
\end{proof}

Define the admissible Liouville--Cauchy orbit by
\begin{equation}
 \mathcal M_{\rm LCP}
 :=\left\{e^{\kappa f}M_x^\partial:
 x\in\Omega_{\rm str}^\partial,\ f\in\Hreg\right\}.
 \label{eq:lcp-admissible-orbit}
\end{equation}
For $\mu=e^{\kappa f}M_x^\partial$ in this class, set
$ c_f:=e^{\kappa\|f\|_\infty}$.  Then
\begin{equation}
 \begin{gathered}
  c_f^{-1}M_x^\partial\le\mu\le c_fM_x^\partial,
  \qquad \Vd_\mu=\Vd_x,\\
  c_f^{-1/2}\|u\|_{L^2(M_x^\partial)}
  \le\|u\|_{L^2(\mu)}
  \le c_f^{1/2}\|u\|_{L^2(M_x^\partial)}.
 \end{gathered}
 \label{eq:lcp-admissible-comparison}
\end{equation}
Write $A_\mu$ for the non-negative self-adjoint operator associated with
$(\E,\Vd_\mu)$ in $L^2(\mu)$.  For $g\in\Hreg$,
\begin{equation}
 e^{\kappa g}\mu=e^{\kappa(f+g)}M_x^\partial\in\mathcal M_{\rm LCP}.
 \label{eq:lcp-admissible-tilt-closure}
\end{equation}

The abstract and Liouville--Cauchy symbols are related by
\begin{equation}
 \Xspace=\Xspace_{\rm C},\qquad \mathbb P=\mathbb P_{\rm C},
 \qquad
 (\Omega_{\rm can},\Omega_{\rm str})
 =(\Omega_{\rm GMC}^{\partial},\Omega_{\rm str}^{\partial}).
 \label{eq:lcp-abstract-gaussian-dictionary}
\end{equation}
and
\begin{equation}
 \State=\LCircle,\qquad \mu_x=M_x^\partial,
 \qquad
 (\Hc,\Hreg,\kappa,\rho,\mathcal Q)
 =\left(\dot H^{1/2}(\LCircle),
 C^\infty(\LCircle)\cap\Hc,\frac\gamma2,m,\mathcal Q_{\rm C}\right).
 \label{eq:lcp-abstract-dictionary}
\end{equation}
At the coordinate environment $h$,
\[
 (\Vd_h^{\rm C},A_h^{\rm C},\Lambda_n^{{\rm C},h})
 =(\Vd_h,A_h,\Lambda_n^h).
\]

\subsubsection{Liouville--Cauchy response transversality}

The preceding constructions provide the model inputs for
Assumptions~\ref{ass:admissible-forms}--\ref{ass:spectral-borel}.  We now
verify the additional pathwise condition in
Definition~\ref{def:abstract-response-transversality}.  Fix
$x\in\Omega_{\rm str}^{\partial}$ and real normalized eigenpairs
$(\Lambda_i,\phi_i)_{i=1}^M$ of $A_x$ with pairwise distinct positive
eigenvalues.  Define
\[
 L:=A_0\big|_{\operatorname{span}_{\mathbb R}
   \{\phi_1^2,\ldots,\phi_M^2\}},
 \qquad
 b_i:=\Gamma_J(\phi_i)m,
 \quad 1\le i\le M,
\]
where $L$ is initially understood distributionally.

\begin{corollary}
\label{cor:lcp-square-interface}
The map $L$ takes values in $\Mc_{\rm fin}(\LCircle)$, and each $b_i$ is a
finite positive Radon measure.  Moreover,
\begin{equation}
 M_x^\partial\perp m,\qquad
 \operatorname{supp}M_x^\partial=\LCircle,\qquad
 b_i\ll m,\qquad
 L(\phi_i^2)=2\Lambda_i\phi_i^2M_x^\partial-b_i,
 \quad 1\le i\le M.
 \label{eq:lcp-square-interface-identity}
\end{equation}
\end{corollary}

\begin{proof}
Proposition~\ref{prop:lcp-canonical-structural-event} gives
$M_x^\partial\perp m$ and full support.  By
Lemma~\ref{lem:lcp-nonlocal-square-identity},
\[
 A_0(\phi_i^2)
 =2\Lambda_i\phi_i^2M_x^\partial-\Gamma_J(\phi_i)m,
 \qquad \Gamma_J(\phi_i)m\ll m.
\]
The total-energy identity \eqref{eq:lcp-jump-energy-total} makes each $b_i$
finite.  Hence $L(\phi_i^2)\in\Mc_{\rm fin}(\LCircle)$ and the displayed
properties follow.
\end{proof}

Now fix $1\le n_1<\cdots<n_N$ such that
$\Lambda_{n_1}^x,\ldots,\Lambda_{n_N}^x$ are simple, and choose real normalized
eigenfunctions $\phi_{n_1}^x,\ldots,\phi_{n_N}^x$.

\begin{corollary}
\label{cor:lcp-response-transversality}
The response functionals
\[
 \ell_{n_i}^{M_x^\partial}(q)
 =-\kappa\Lambda_{n_i}^x
   \int_{\LCircle}q(\phi_{n_i}^x)^2\,dM_x^\partial,
 \qquad q\in\mathcal Q_{\rm C},
\]
are linearly independent.  Hence the selected eigenvalue family is
response-transverse in the sense of
Definition~\ref{def:abstract-response-transversality}.
\end{corollary}

\begin{proof}
Apply Corollary~\ref{cor:lcp-square-interface} to the eigenpairs
$(\Lambda_{n_i}^x,\phi_{n_i}^x)_{i=1}^N$.  Together with
Lemma~\ref{lem:lcp-abstract-direction-interface}, the hypotheses of
Corollary~\ref{cor:square-implies-RT} are satisfied.  That corollary gives the
stated linear independence.
\end{proof}

\subsection{Final proof of Theorem~\ref{thm:intro-lcp-spectrum}}

Use the identifications
\eqref{eq:lcp-abstract-gaussian-dictionary}--
\eqref{eq:lcp-abstract-dictionary}.  The five assumptions of
Theorem~\ref{thm:intro-abstract-response} are verified as follows.

\begin{itemize}[label=$\bullet$]
\item Assumption~\ref{ass:admissible-forms}.  For
$\mu=e^{\kappa f}M_x^\partial\in\mathcal M_{\rm LCP}$,
\eqref{eq:lcp-admissible-comparison} gives
\[
 \Vd_\mu=\Vd_x,
 \qquad
 \|\cdot\|_{L^2(\mu)}\asymp\|\cdot\|_{L^2(M_x^\partial)}.
\]
Proposition~\ref{prop:lcp-canonical-structural-event} gives a closed, dense
trace form at $M_x^\partial$.  Lemma~\ref{lem:lcp-centered-green-smoothing},
Proposition~\ref{prop:lcp-centered-green-inverse}, and
Corollary~\ref{cor:lcp-borel-positive-spectrum} give
\[
 \ker A_x=\operatorname{span}\{\one\},
 \qquad
 0<\Lambda_1^x\le\Lambda_2^x\le\cdots\uparrow\infty,
\]
with compact inverse on the orthogonal complement of constants.  Hence the
form embedding is compact.  The comparison transfers closedness,
density, and compactness to $\mu$, while full support makes $L^2(\mu)$
infinite-dimensional.

\item Assumption~\ref{ass:regular-form-multipliers}.  For $g\in\Hreg$,
\eqref{eq:H-half-multiplier-bound}--\eqref{eq:lcp-trace-poincare-bound} and
\eqref{eq:lcp-admissible-comparison} give
\[
 g\Vd_\mu\subset\Vd_\mu,
 \qquad
 \|gu\|_{\E,\mu}\le C_{\mu,g}\|u\|_{\E,\mu}.
\]

\item Assumption~\ref{ass:coherent-orbit}.
Lemma~\ref{lem:lcp-canonical-boundary-gmc} gives
\[
 M_{x+g}^\partial=e^{\kappa g}M_x^\partial,
 \qquad x,x+g\in\Omega_{\rm can},\quad g\in\Hreg.
\]
Proposition~\ref{prop:lcp-canonical-structural-event} gives
$\Omega_{\rm str}\subset\Omega_{\rm can}$ with
$\mathbb P_{\rm C}(\Omega_{\rm str})=1$, while
\eqref{eq:lcp-admissible-orbit} and
\eqref{eq:lcp-admissible-tilt-closure} give the structural and closure clauses.

\item Assumption~\ref{ass:direction-separation}.  Lemma
\ref{lem:lcp-abstract-direction-interface} gives
\eqref{eq:countable-direction-space}, and
\eqref{eq:lcp-direction-annihilator} is
\eqref{eq:direction-annihilator} with $\rho=m$.

\item Assumption~\ref{ass:spectral-borel}.  Corollaries
\ref{cor:lcp-borel-positive-spectrum} and
\ref{cor:lcp-borel-compression-traces} give
\eqref{eq:minimal-borel-interface} for rational $0<a<b$,
$q\in\mathcal Q_{\rm C}$, and $n,k\ge1$.
\end{itemize}

The preceding Green inverse gives
$\ker A_h^{\rm C}=\operatorname{span}\{\one\}$ almost surely, while
Theorem~\ref{thm:intro-abstract-response}\textnormal{(i)} makes the positive
spectrum simple:
\[
 0<\Lambda_1^h<\Lambda_2^h<\cdots\uparrow\infty.
\]
Together these statements are \eqref{eq:intro-lcp-simple-spectrum}.

Fix $1\le n_1<\cdots<n_N$.  On the same full-probability simplicity event,
Corollary~\ref{cor:lcp-response-transversality}, evaluated at $x=h$, gives
\[
 \sum_{i=1}^Na_i\ell_{n_i}^{M_h^\partial}(q)=0
 \quad(q\in\mathcal Q_{\rm C})
 \quad\Longrightarrow\quad
 a_1=\cdots=a_N=0.
\]
Thus Definition~\ref{def:abstract-response-transversality} holds, and
Theorem~\ref{thm:intro-abstract-response}\textnormal{(ii)} yields
\eqref{eq:intro-lcp-density}.  This proves
Theorem~\ref{thm:intro-lcp-spectrum}.  The stronger potential realization in
Definition~\ref{def:response-realization} is established separately in
Appendix~\ref{app:lcp-response}.

\appendix

\section{Further response results for LBM}
\label{app:lbm-finite-energy-response}

Theorem~\ref{thm:intro-lbm-spectrum} was proved entirely in
Section~\ref{sec:lbm-realization}.  This appendix develops further pathwise
response information that is not needed for that theorem.  We first represent
the response measures by finite-energy Dirichlet Green potentials, in a
canonical environment and along the regular tilt orbit, and then extend the
response to all of $H_0^1(D)$.  In particular, we verify the stronger
realization property in Definition~\ref{def:response-realization}.

Throughout the appendix we return to the LBM conventions of
Section~\ref{sec:lbm-realization}:
\[
 \Hc=H_0^1(D),\qquad \Hreg=C_c^\infty(D),\qquad
 \kappa=\gamma,\qquad \rho=0,
\]
$\E$ is the Dirichlet energy in
\eqref{eq:standing-energy-normalization}, and hence
$\Theta_\mu(u,v)=uv\,\mu$.
Recall also from \eqref{eq:lbm-structural-event-abstract} and
\eqref{eq:lbm-admissible-orbit} that
\[
 \Omega_{\rm LBM}\subseteq\Omega_{\rm Fr}\cap\Omega_{\rm op},
 \qquad
 \mathcal M_{\rm LBM}
 =\{e^{\gamma f}M_x:
      x\in\Omega_{\rm LBM},\ f\in C_c^\infty(D)\}.
\]

\subsection{Finite-energy spectral products}

For $x\in\Omega_{\rm Fr}\cap\Omega_{\rm op}$ and
$\theta\in L^2(M_x)$, Lemmas~\ref{lem:green-hilbert-schmidt},
\ref{lem:finite-energy-potential}, and~\ref{lem:green-form-inverse}, together
with Proposition~\ref{prop:green-potential-L2-Cb}, give
\[
 \theta M_x\in\ME,
 \qquad
 \widetilde{\Pot(\theta M_x)}=\GreenOp_x\theta
 \quad\text{$\E$-q.e.}.
\]
The two sides represent the same element of $L^2(M_x)$, and
$\GreenOp_x\theta$ is its unique continuous representative.
\begin{proposition}
\label{prop:pathwise-cluster-product-green-energy}
Fix $x\in\Omega_{\rm Fr}\cap\Omega_{\rm op}$, and let
$(\Lambda_i^x,\phi_i)$ and $(\Lambda_j^x,\phi_j)$ be positive eigenpairs of
$A_x$ with real $L^2(M_x)$-normalized eigenfunctions.  Then
\begin{equation}
 \phi_i\phi_jM_x\in\ME,
 \label{eq:pathwise-cluster-product-membership}
\end{equation}
and
\begin{equation}
 I_D(|\phi_i\phi_j|M_x)
 \le M_x(D)^{1/2}K_G(x)^{3/2}
      \min\{(\Lambda_i^x)^2,(\Lambda_j^x)^2\}<\infty.
 \label{eq:pathwise-cluster-product-green-energy-bound}
\end{equation}
Consequently, for every $g\in\Hc$,
\begin{equation}
 \big\langle\Pot(\phi_i\phi_jM_x),g\big\rangle_{\Hc}
 =\int_D\widetilde g\,\phi_i\phi_j\,dM_x.
 \label{eq:eigenfunction-product-response-pairing}
\end{equation}
\end{proposition}

\begin{proof}
For $\theta:=|\phi_i\phi_j|$, Lemma~\ref{lem:green-potential-eigenfunction-regularity}
gives
\[
 \|\theta\|_{L^2(M_x)}^2
 \le\min\{\|\phi_i\|_\infty^2,\|\phi_j\|_\infty^2\}
 \le K_G(x)\min\{(\Lambda_i^x)^2,(\Lambda_j^x)^2\}.
\]
By \eqref{eq:green-HS-from-smoothing} and
\eqref{eq:green-form-inverse-energy},
\begin{align*}
 I_D(\theta M_x)
 &=\langle\theta,\GreenOp_x\theta\rangle_{L^2(M_x)}
 \le\|\GreenOp_x\|_{\rm op}\|\theta\|_{L^2(M_x)}^2\\
 &\le M_x(D)^{1/2}K_G(x)^{1/2}
       \|\theta\|_{L^2(M_x)}^2,
\end{align*}
which proves \eqref{eq:pathwise-cluster-product-green-energy-bound}.  Since
$0\le(\phi_i\phi_j)_\pm\le\theta$ and $G_D\ge0$, both Jordan parts have finite
energy, proving \eqref{eq:pathwise-cluster-product-membership}.  Finally,
\eqref{eq:eigenfunction-product-response-pairing} follows from
Lemma~\ref{lem:finite-energy-potential}.
\end{proof}

The Green estimates from Section~\ref{sec:lbm-realization} persist along the
regular-tilt orbit.

\begin{lemma}
\label{lem:lbm-tilted-product-energy}
Let $x\in\Omega_{\rm LBM}$, $f\in C_c^\infty(D)$, and
$\mu=e^{\gamma f}M_x$.  Put
\[
 K_G(\mu)
 :=\sup_{z\in D}\int_DG_D(z,y)^2\,\mu(dy).
\]
For $w\in L^2(\mu)$, write
\[
 \GreenOp_\mu w(z):=\int_DG_D(z,y)w(y)\,\mu(dy).
\]
Then
\begin{align}
 K_G(\mu)
 &\le e^{\gamma\|f\|_\infty}K_G(x)<\infty,
 \label{eq:lbm-tilted-smoothing-bound}\\
 \iint_{D^2}G_D(z,y)^2\,\mu(dz)\mu(dy)
 &\le e^{2\gamma\|f\|_\infty}
 \iint_{D^2}G_D(z,y)^2\,
 M_x(dz)M_x(dy)<\infty.
 \label{eq:lbm-tilted-HS-bound}
\end{align}
Every real $L^2(\mu)$-normalized eigenfunction of $A_\mu$ has a bounded
continuous quasi-continuous representative.  Any two positive eigenpairs
$(\Lambda_i^\mu,\phi_i)$ and $(\Lambda_j^\mu,\phi_j)$ with real normalized
eigenfunctions satisfy
\begin{equation}
 \phi_i\phi_j\,\mu\in\ME,
 \qquad
 I_D(|\phi_i\phi_j|\mu)
 \le
 \mu(D)^{1/2}K_G(\mu)^{3/2}
 \min\{(\Lambda_i^\mu)^2,(\Lambda_j^\mu)^2\}.
 \label{eq:lbm-tilted-product-energy}
\end{equation}
\end{lemma}

\begin{proof}
Set $c_f:=e^{\gamma\|f\|_\infty}$.  The comparison
\[
 c_f^{-1}M_x\le\mu\le c_fM_x
\]
gives \eqref{eq:lbm-tilted-smoothing-bound}--\eqref{eq:lbm-tilted-HS-bound};
it also gives $\Vd_\mu=\Vd_x$ and preserves the Frostman and full-support
properties.

For a normalized eigenpair $(\Lambda,\phi)$, set
\[
 \theta_f:=e^{\gamma f}\phi,
 \qquad
 \phi\mu=\theta_fM_x,
 \qquad
 \|\theta_f\|_{L^2(M_x)}^2
 \le c_f\|\phi\|_{L^2(\mu)}^2=c_f.
\]
Lemmas~\ref{lem:green-hilbert-schmidt} and
\ref{lem:green-form-inverse}, applied to $M_x$, show that
\[
 u:=\Pot(\phi\mu)=\Pot(\theta_fM_x)
 \in\Vd_x=\Vd_\mu.
\]
The finite-energy pairing and the weak eigenvalue equation therefore give
\[
 \E(\phi-\Lambda u,v)=0\quad(v\in\Vd_\mu)
 \quad\Longrightarrow\quad
 \phi=\Lambda u\quad\text{in }H_0^1(D).
\]
Lemma~\ref{lem:finite-energy-potential} identifies the quasi-continuous
representative of $u=\Pot(\phi\mu)$ with the Green integral
$\GreenOp_\mu\phi$.  Consequently,
\[
 \widetilde\phi=\Lambda\GreenOp_\mu\phi
 \quad\text{quasi-everywhere and in }L^2(\mu).
\]
The proof of Lemma~\ref{lem:green-hilbert-schmidt}, with $M_x$ replaced by
$\mu$, also gives, for every $w\in L^2(\mu)$,
\[
 w\mu\in\ME,
 \qquad
 I_D(w\mu)
 =\langle w,\GreenOp_\mu w\rangle_{L^2(\mu)}.
\]
Green smoothing therefore yields a bounded continuous quasi-continuous
representative and
\[
 \|\phi\|_\infty\le\Lambda K_G(\mu)^{1/2}.
\]
For $\theta:=|\phi_i\phi_j|$, it follows that
\begin{align*}
 \|\theta\|_{L^2(\mu)}^2
 &\le K_G(\mu)
       \min\{(\Lambda_i^\mu)^2,(\Lambda_j^\mu)^2\},\\
 I_D(\theta\mu)
 &=\langle\theta,\GreenOp_\mu\theta\rangle_{L^2(\mu)}
 \le \mu(D)^{1/2}K_G(\mu)^{1/2}
       \|\theta\|_{L^2(\mu)}^2.
\end{align*}
This proves \eqref{eq:lbm-tilted-product-energy}; the positive and negative
parts of $\phi_i\phi_j\mu$ are dominated by $\theta\mu$, so
$\phi_i\phi_j\mu\in\ME$.
\end{proof}

\subsection{Potential realization and Cameron--Martin response}

For every $\mu\in\mathcal M_{\rm LBM}$, take
\begin{equation}
 \mathcal M_{\rm resp}(\mu):=\ME,
 \qquad
 \Potential_\mu\sigma:=\Pot\sigma.
 \label{eq:lbm-response-space}
\end{equation}

\begin{proposition}
\label{prop:lbm-response-realization}
For every $\mu\in\mathcal M_{\rm LBM}$ and every $u,v$ in a common
positive eigenspace of $A_\mu$,
\begin{equation}
 \Theta_\mu(u,v)=uv\,\mu\in\ME.
 \label{eq:lbm-response-product-space}
\end{equation}
The map $\Potential_\mu=\Pot$ in \eqref{eq:lbm-response-space} is linear and
injective, and
\begin{equation}
 \inner{\Potential_\mu\sigma}{f}_{\Hc}
 =\int_Df\,d\sigma,
 \qquad
 \sigma\in\ME,\quad f\in C_c^\infty(D).
 \label{eq:lbm-response-potential-identification}
\end{equation}
Thus the LBM response is faithfully potential-realizable in the sense of
Definition~\ref{def:response-realization}.
\end{proposition}

\begin{proof}
For $\nu,\eta\in\ME$, the positive mutual-energy estimate
\eqref{eq:positive-mutual-energy-pairing}, applied to their Jordan parts, gives
$I_D(|\nu|),I_D(|\eta|)<\infty$.  Positivity of $G_D$ and
\eqref{eq:green-energy-cs} then give
\[
 I_D(|\nu+\eta|)^{1/2}
 \le I_D(|\nu|+|\eta|)^{1/2}
 \le I_D(|\nu|)^{1/2}
     +I_D(|\eta|)^{1/2}<\infty.
\]
Since $(\nu+\eta)^\pm\le|\nu+\eta|$, this proves closure under addition.
The same argument applies to scalar multiples, so $\ME$ is a real vector
space.  Lemma~\ref{lem:finite-energy-potential} yields
\[
 \Pot(a\nu+b\eta)=a\Pot\nu+b\Pot\eta,
 \qquad a,b\in\mathbb R,
\]
and \eqref{eq:lbm-response-potential-identification};
Lemma~\ref{lem:green-energy-strict-positive} gives injectivity.

If $u,v$ belong to a common positive eigenspace of $A_\mu$, then
Lemma~\ref{lem:lbm-tilted-product-energy} and homogeneity give
\[
 \Theta_\mu(u,v)=uv\,\mu\in\ME,
\]
with the zero case immediate.  These are exactly the requirements of
Definition~\ref{def:response-realization}.
\end{proof}

\begin{proposition}
\label{prop:lbm-full-cm-response}
Let $x\in\Omega_{\rm LBM}$ and let $\Lambda>0$ have multiplicity
$r$ for $A_x$, with real $L^2(M_x)$-orthonormal eigenbasis
$(\phi_i)_{i=1}^r$.  Set
\begin{equation}
 \sigma_{ij}^x
 :=\Theta_{M_x}(\phi_i,\phi_j)=\phi_i\phi_jM_x
 \in\ME,
 \qquad 1\le i,j\le r.
 \label{eq:lbm-cluster-product-measures}
\end{equation}
The response in \eqref{eq:lbm-cluster-response-matrix}, initially defined
for $f\in C_c^\infty(D)$, extends uniquely to a continuous linear map
$f\mapsto\Response_f^\Lambda(M_x)$ on $\Hc$, given by
\begin{equation}
 \begin{aligned}
 \left(\inner{\Response_f^\Lambda(M_x)\phi_j}{\phi_i}
              _{L^2(M_x)}\right)_{i,j=1}^r
 &=-\gamma\Lambda
 \left(\inner{\Pot\sigma_{ij}^x}{f}_{\Hc}\right)_{i,j=1}^r\\
 &=-\gamma\Lambda
 \left(\int_D\widetilde f\,\phi_i\phi_j\,dM_x\right)_{i,j=1}^r.
 \end{aligned}
 \label{eq:lbm-full-cm-cluster-response}
\end{equation}
For $f\in H_0^1(D)\setminus C_c^\infty(D)$, this notation denotes only the
continuous extension; no tilted operator or analytic eigenvalue branch is
asserted.
\end{proposition}

\begin{proof}
Proposition~\ref{prop:lbm-response-realization}, the density of
$C_c^\infty(D)$ in $H_0^1(D)$, and $\rho=0$ verify both conditions in
\eqref{eq:abstract-full-cm-direction-conditions}.  Hence
Proposition~\ref{prop:full-cm-extension} applies.  Since
$\Potential_{M_x}=\Pot$, $\kappa=\gamma$, and
$\Theta_{M_x}(\phi_i,\phi_j)=\sigma_{ij}^x$, it gives the first equality in
\eqref{eq:lbm-full-cm-cluster-response}.  The second follows from
Lemma~\ref{lem:finite-energy-potential}:
\[
 \inner{\Pot\sigma_{ij}^x}{f}_{\Hc}
 =\int_D\widetilde f\,d\sigma_{ij}^x
 =\int_D\widetilde f\,\phi_i\phi_j\,dM_x,
 \qquad f\in H_0^1(D).
\]
\end{proof}

\begin{corollary}
\label{cor:lbm-finite-energy-simple-response}
Let $x\in\Omega_{\rm LBM}$, let $\Lambda_n^x$ be a simple eigenvalue
of $A_x$, and choose a real $L^2(M_x)$-normalized eigenfunction $\phi_n^x$.
The response measure $\eta_n^x$ in
\eqref{eq:lbm-spectral-response-measure} belongs to $\ME$ and satisfies
\begin{equation}
 I_D(\eta_n^x)
 \le \gamma^2(\Lambda_n^x)^4M_x(D)^{1/2}
       K_G(x)^{3/2}<\infty.
 \label{eq:lbm-response-energy-bound}
\end{equation}
The simple-eigenvalue response functional extends continuously to $\Hc$ as
\begin{equation}
 f\longmapsto
 -\gamma\Lambda_n^x\int_D\widetilde f(\phi_n^x)^2\,dM_x
 =\inner{\Pot\eta_n^x}{f}_{\Hc},
 \qquad f\in\Hc.
 \label{eq:lbm-full-cm-simple-response}
\end{equation}
\end{corollary}

\begin{proof}
Proposition~\ref{prop:pathwise-cluster-product-green-energy} with $i=j=n$
gives $(\phi_n^x)^2M_x\in\ME$ and
\[
 I_D((\phi_n^x)^2M_x)
 \le (\Lambda_n^x)^2M_x(D)^{1/2}
      K_G(x)^{3/2}.
\]
Since $\eta_n^x=-\gamma\Lambda_n^x(\phi_n^x)^2M_x$,
\[
 I_D(\eta_n^x)
 =\gamma^2(\Lambda_n^x)^2I_D((\phi_n^x)^2M_x),
\]
which gives \eqref{eq:lbm-response-energy-bound}.  Finally, for
$f\in\Hc$,
\[
 \inner{\Pot\eta_n^x}{f}_{\Hc}
 =\int_D\widetilde f\,d\eta_n^x
 =-\gamma\Lambda_n^x\int_D\widetilde f(\phi_n^x)^2\,dM_x,
\]
proving \eqref{eq:lbm-full-cm-simple-response}.
\end{proof}

\begin{corollary}
\label{cor:lbm-positive-response-gram}
Let $x\in\Omega_{\rm LBM}$ and $1\le n_1<\cdots<n_N$ satisfy the hypotheses
of Corollary~\ref{cor:lbm-response-transversality}.  Then the matrix
\begin{equation}
 \left(\inner{\Pot\eta_{n_i}^x}{\Pot\eta_{n_j}^x}_{\Hc}\right)_{i,j=1}^N
 \label{eq:lbm-response-gram}
\end{equation}
is positive definite.
\end{corollary}

\begin{proof}
This is Corollary~\ref{cor:abstract-positive-response-gram}, using
Proposition~\ref{prop:lbm-response-realization} and
Corollary~\ref{cor:lbm-response-transversality}.
\end{proof}

\subsection{Cluster envelopes and first-order splitting}

For $x\in\Omega_{\rm LBM}$ and an eigenvalue $\Lambda>0$, define
\begin{equation}
 \mathcal N_x(\Lambda)
 :=\left\{-\gamma\Lambda\psi^2M_x:
 \psi\in E_x(\Lambda),\
 \norm{\psi}_{L^2(M_x)}=1\right\}.
 \label{eq:lbm-active-response-family}
\end{equation}

\begin{corollary}
\label{cor:lbm-cluster-envelope}
Let $x\in\Omega_{\rm LBM}$ and let $\Lambda>0$ be an eigenvalue of
$A_x$ occupying the ordered labels $n,\ldots,n+r-1$.  Then, for every
$f\in C_c^\infty(D)$,
\begin{equation}
 \partial_f^+\Lambda_n^x
 =\min_{\eta\in\mathcal N_x(\Lambda)}\int_Df\,d\eta,
 \qquad
 \partial_f^-\Lambda_n^x
 =\max_{\eta\in\mathcal N_x(\Lambda)}\int_Df\,d\eta.
 \label{eq:lbm-cluster-lower-edge-envelope}
\end{equation}
\end{corollary}

\begin{proof}
For a unit $\psi\in E_x(\Lambda)$,
\[
 \int_D f\,d(-\gamma\Lambda\psi^2M_x)
 =\inner{\Response_f^\Lambda(M_x)\psi}{\psi}_{L^2(M_x)}.
\]
Rayleigh--Ritz and Corollary~\ref{cor:lbm-ordered-cluster-response} give
\eqref{eq:lbm-cluster-lower-edge-envelope}.
\end{proof}

\begin{lemma}
\label{lem:lbm-active-response-multiplicity}
For $x\in\Omega_{\rm LBM}$ and an eigenvalue $\Lambda>0$ of multiplicity $r$,
\begin{equation}
 \mathcal N_x(\Lambda)\text{ is a singleton}
 \quad\Longleftrightarrow\quad r=1.
 \label{eq:lbm-active-family-singleton}
\end{equation}
\end{lemma}

\begin{proof}
If $r=1$, unit eigenfunctions differ only by sign.  Conversely, if $r\ge2$
and the family were a singleton, orthonormal $u,v\in E_x(\Lambda)$ would
satisfy
\[
 u^2M_x=v^2M_x=\frac{(u+v)^2}{2}M_x.
\]
Thus $uv=0$ and $u^2=v^2$ $M_x$-almost everywhere, so
$u^4=u^2v^2=0$ $M_x$-almost everywhere, contrary to normalization.
\end{proof}

\begin{corollary}
\label{cor:lbm-first-order-splitting}
Let $x\in\Omega_{\rm LBM}$ and let $\Lambda>0$ be a multiple eigenvalue of
$A_x$.  There exists $q\in\mathcal Q_D$ such that
\begin{equation}
 \Response_q^\Lambda(M_x)
 =-\gamma\Lambda\EigProj_\Lambda^x
   \Mult_q\big|_{E_x(\Lambda)}
 \label{eq:lbm-simple-splitting-matrix}
\end{equation}
has simple spectrum.  Hence the analytic branches issuing from $\Lambda$
along $M_{x,\tau}^{q}=e^{\gamma\tau q}M_x$ have pairwise distinct first
derivatives at zero.
\end{corollary}

\begin{proof}
The admissible-form, multiplier, and tilt-closure hypotheses used in
Section~\ref{sec:pathwise-response} were verified in
Section~\ref{sec:lbm-realization}.  Lemma~\ref{lem:lbm-abstract-direction-interface} and
\eqref{eq:lbm-abstract-dictionary} verify
\eqref{eq:countable-direction-space}--\eqref{eq:direction-annihilator} with
$(\kappa,\rho,\mathcal Q)=(\gamma,0,\mathcal Q_D)$.  Hence
Theorem~\ref{thm:abstract-simple-splitting} applies.  Proposition
\ref{prop:lbm-cluster-response} identifies the resulting simple response
spectrum with the branch derivatives.
\end{proof}

\section{Further response results for the Liouville--Cauchy operator}
\label{app:lcp-response}

Theorem~\ref{thm:intro-lcp-spectrum} was proved entirely in
Section~\ref{sec:lcp-realization}.  This appendix develops the centered
finite-energy and full Cameron--Martin response calculus, which is not needed
for that theorem.

Throughout this appendix we use the Liouville--Cauchy conventions of
Section~\ref{sec:lcp-realization}:
\[
 \Hc=\dot H^{1/2}(\LCircle),\qquad
 \Hreg=C^\infty(\LCircle)\cap\Hc,
 \qquad \kappa=\frac\gamma2,\qquad \rho=m.
\]
Here $G$ is the Cauchy pseudo-Green kernel in
\eqref{eq:circle-pseudo-green}, $\mathcal M_{\rm LCP}$ is the orbit
in \eqref{eq:lcp-admissible-orbit}, and $\Theta_\mu$ is the centered product
measure in \eqref{eq:centered-product-measure}.
All pathwise statements are made for a deterministic
$x\in\Omega_{\rm str}^{\partial}$; evaluation at the coordinate field is
obtained by setting $x=h$.

\subsection{Finite-energy centered spectral products}

For any finite signed measure $\zeta$ on $\LCircle$ with
$\zeta(\LCircle)=0$ and
$\sum_{n\ne0}|\widehat\zeta(n)|^2/|n|<\infty$, define its Fourier
potential $\Potential\zeta\in\Hc$ by
\begin{equation}
 \widehat{\Potential\zeta}(0)=0,
 \qquad
 \widehat{\Potential\zeta}(n)=\frac{\widehat\zeta(n)}{|n|},
 \quad n\ne0.
 \label{eq:lcp-response-potential-fourier}
\end{equation}
When the Green integral is absolutely convergent, this agrees with the kernel
potential $\theta\mapsto\int_{\LCircle}G(\theta,\varphi)\,\zeta(d\varphi)$.

Let $\sigma_1$ and $\sigma_2$ be finite signed measures of total mass zero satisfying
\begin{equation}
 \iint_{\LCircle^2}|G(\theta,\varphi)|
 \,d(|\sigma_1|+|\sigma_2|)(\theta)\,d(|\sigma_1|+|\sigma_2|)(\varphi)<\infty.
 \label{eq:lcp-green-energy-joint-domain}
\end{equation}
Define
\begin{equation}\label{eq:lcp-green-energy-definition}
 I_{\LCircle}(\sigma_1,\sigma_2)
 :=\iint_{\LCircle\times\LCircle}G(\theta,\varphi)\sigma_1(d\theta)\sigma_2(d\varphi),
 \qquad
 I_{\LCircle}(\sigma_1):=I_{\LCircle}(\sigma_1,\sigma_1),
\end{equation}
and similarly for $\sigma_2$ and $\sigma_1+\sigma_2$.  The Abel self-energy argument in
\eqref{eq:lcp-fourier-energy-measure} gives
\[
 I_{\LCircle}(\zeta)
 =\sum_{n\ne0}\frac{|\widehat\zeta(n)|^2}{|n|}
 =\norm{\Potential\zeta}_{\Hc}^2,
 \qquad \zeta\in\{\sigma_1,\sigma_2,\sigma_1+\sigma_2\}.
\]
Polarization therefore gives
\begin{equation}\label{eq:lcp-energy-potential-isometry}
 I_{\LCircle}(\sigma_1,\sigma_2)
 =\frac12\left[
   I_{\LCircle}(\sigma_1+\sigma_2)
   -I_{\LCircle}(\sigma_1)
   -I_{\LCircle}(\sigma_2)\right]
 =\inner{\Potential\sigma_1}{\Potential\sigma_2}_{\Hc}.
\end{equation}

\begin{lemma}
\label{lem:lcp-finite-energy-centered-products}
Let $x\in\Omega_{\rm str}^{\partial}$, and let
$\phi_1,\ldots,\phi_r$ be real, $L^2(M_x^\partial)$-orthonormal
eigenfunctions of $A_x$ with positive eigenvalues, and define
\begin{equation}
 \sigma_{ij}^x
 :=\Theta_{M_x^\partial}(\phi_i,\phi_j)
 =\phi_i\phi_jM_x^\partial-\delta_{ij}m,
 \qquad 1\le i,j\le r.
 \label{eq:lcp-centered-cluster-products}
\end{equation}
For $1\le i,j\le r$,
\begin{equation}
 \sigma_{ij}^x(\LCircle)=0,
 \qquad
 \iint_{\LCircle\times\LCircle}|G(\theta,\varphi)|
 \dd|\sigma_{ij}^x|(\theta)
 \dd|\sigma_{ij}^x|(\varphi)<\infty.
 \label{eq:lcp-centered-products-finite-energy}
\end{equation}
Moreover, $\Potential\sigma_{ij}^x\in\Hc$ and
\begin{equation}
 \inner{\Potential\sigma_{ij}^x}{f}_{\Hc}
 =\int_{\LCircle}f\phi_i\phi_j\,dM_x^\partial,
 \qquad f\in\Hreg.
 \label{eq:lcp-centered-product-pairing}
\end{equation}
\end{lemma}

\begin{proof}
Corollary~\ref{cor:lcp-continuous-eigenfunctions} gives bounded continuous
representatives of all $\phi_i$, hence
\begin{equation*}
 |\sigma_{ij}^x|
 \le \norm{\phi_i}_\infty\norm{\phi_j}_\infty M_x^\partial+\delta_{ij}m.
\end{equation*}
By \eqref{eq:circle-green-log-bound}, \eqref{eq:uniform-log-square-bound}, its
$m$-analogue, and Cauchy--Schwarz,
\begin{align*}
 &\iint|G|\,d|\sigma_{ij}^x|\,d|\sigma_{ij}^x|\\
 &\le \norm{\phi_i}_\infty^2\norm{\phi_j}_\infty^2
       \iint|G|\,dM_x^\partial dM_x^\partial\\
 &\quad+2\delta_{ij}\norm{\phi_i}_\infty\norm{\phi_j}_\infty
       \iint|G|\,dM_x^\partial dm
 +\delta_{ij}\iint|G|\,dm\,dm<\infty.
 \end{align*}
 This proves \eqref{eq:lcp-centered-products-finite-energy}, and orthonormality
 gives
 $\sigma_{ij}^x(\LCircle)
 =\int\phi_i\phi_j\,dM_x^\partial-\delta_{ij}=0$.
 The Abel identity preceding \eqref{eq:lcp-energy-potential-isometry} now gives
 $\Potential\sigma_{ij}^x\in\Hc$.  For $f\in\Hreg$, Cauchy--Schwarz in
 Fourier space and Fourier inversion give
 \begin{align*}
  \inner{\Potential\sigma_{ij}^x}{f}_{\Hc}
  &=\sum_{n\ne0}
     \widehat{\sigma_{ij}^x}(n)
     \overline{\widehat f(n)}
    =\int f\,d\sigma_{ij}^x\\
  &=\int f\phi_i\phi_j\,dM_x^\partial
    -\delta_{ij}\int f\,dm
   =\int f\phi_i\phi_j\,dM_x^\partial,
 \end{align*}
 because $f\in\Hc$ has zero $m$-mean.  This is
 \eqref{eq:lcp-centered-product-pairing}.
The left side extends continuously to all of $\Hc$.  We use it as the
finite-energy pairing with $\sigma_{ij}^x$ for non-smooth
$f\in\Hc$; no ordinary integral against a quasi-continuous representative is
asserted there.
\end{proof}

Let $\EigProj_\Lambda^x$ be the
$L^2(M_x^\partial)$-orthogonal projection onto
$\ker(A_x-\Lambda)$.  For
$f\in\Hreg$, define
\begin{equation}
 \Response_f^\Lambda(M_x^\partial)
 :=-\kappa\Lambda
 \EigProj_\Lambda^x
 \Mult_f\big|_{\ker(A_x-\Lambda)}.
 \label{eq:lcp-response-compression}
\end{equation}

\begin{proposition}
\label{prop:lcp-cluster-response}
Let $x\in\Omega_{\rm str}^{\partial}$ and let $\Lambda>0$ have multiplicity
$r$ for $A_x$, with real $L^2(M_x^\partial)$-orthonormal eigenbasis
$(\phi_i)_{i=1}^r$.  Put
$\sigma_{ij}^x:=\Theta_{M_x^\partial}(\phi_i,\phi_j)$ as in
\eqref{eq:lcp-centered-cluster-products}.  For $f\in\Hreg$, the first variation along
$M_{x,\tau}^{\partial,f}
:=e^{\kappa\tau f}M_x^\partial$ is
\begin{equation}
 \left(\inner{\Response_f^\Lambda(M_x^\partial)\phi_j}{\phi_i}
              _{L^2(M_x^\partial)}\right)_{i,j=1}^r
 =-\kappa\Lambda
 \left(\int_{\LCircle}f\phi_i\phi_j\,dM_x^\partial\right)_{i,j=1}^r
 =-\kappa\Lambda
 \left(\inner{\Potential\sigma_{ij}^x}{f}_{\Hc}
 \right)_{i,j=1}^r.
 \label{eq:lcp-cluster-response-matrix}
\end{equation}
The last expression defines the unique continuous linear extension
$f\mapsto\Response_f^\Lambda(M_x^\partial)$ from
$\Hreg$ to $\Hc$.
\end{proposition}

\begin{proof}
Along the coherent line, the fixed-space unitary is
\[
 U_{\tau,f}:L^2(e^{\kappa\tau f}M_x^\partial)
 \longrightarrow L^2(M_x^\partial),
 \qquad U_{\tau,f}u=e^{\kappa\tau f/2}u.
\]
The transported form is
\[
 a_{\tau,f}(u,v)
 =\E(e^{-\kappa\tau f/2}u,e^{-\kappa\tau f/2}v),
 \qquad \Dom a_{\tau,f}=\Vd_x,
\]
and \eqref{eq:H-half-multiplier-bound}--
\eqref{eq:lcp-trace-poincare-bound} justify differentiation in the form norm.
Thus
\[
 \dot a_{0,f}(u,v)
 =-\frac\kappa2\E(fu,v)
  -\frac\kappa2\E(u,fv).
\]
For $u,v$ in the $\Lambda$-eigenspace, the weak eigenvalue equation gives
\[
 \E(fu,v)=\E(u,fv)
 =\Lambda\int_{\LCircle}fuv\,dM_x^\partial.
\]
Therefore
\[
 \dot a_{0,f}(u,v)
 =-\kappa\Lambda\int_{\LCircle}fuv\,dM_x^\partial.
 \]
 The form Hellmann--Feynman formula gives
 \eqref{eq:lcp-cluster-response-matrix}.  Continuity on $\Hc$ follows from
 \eqref{eq:lcp-centered-product-pairing}, and uniqueness follows from the
 density of $\Hreg$ in $\Hc$.
\end{proof}

For $f\in\Hc\setminus\Hreg$ this notation denotes only the
continuous response extension supplied by
\eqref{eq:lcp-centered-product-pairing}; no differentiable form path is
asserted.

\subsection{Potential realization and centered response}

For every $\mu\in\mathcal M_{\rm LCP}$, set
\begin{equation}
 \mathcal M_{\rm resp}(\mu)
 :=\left\{\sigma\in\Mc_{\rm fin}(\LCircle):
 \sigma(\LCircle)=0,\quad
 \sum_{k\ne0}\frac{|\widehat\sigma(k)|^2}{|k|}<\infty\right\}.
 \label{eq:lcp-response-space}
\end{equation}
For $\sigma\in\mathcal M_{\rm resp}(\mu)$, set
$\Potential_\mu\sigma:=\Potential\sigma$ by
\eqref{eq:lcp-response-potential-fourier}.

\begin{proposition}
\label{prop:lcp-response-realization}
The Liouville--Cauchy response is faithfully potential-realizable in the sense
of Definition~\ref{def:response-realization}, with
\eqref{eq:lcp-response-space} and \eqref{eq:lcp-response-potential-fourier}.
\end{proposition}

\begin{proof}
Fix
\[
 \mu=e^{\kappa g}M_x^\partial
 \in\mathcal M_{\rm LCP}.
\]
Set $Z_\mu:=\mu(\LCircle)$ and define
$\bar G_\mu,c_\mu,G_\mu^\circ,$ and $\GreenOp_\mu$ by the centered
formulas \eqref{eq:centered-green-gh}--
\eqref{eq:lcp-centered-green-operator}, with
$M_x^\partial,Z_x^\partial$ replaced by $\mu,Z_\mu$.  The bounded density comparison with
$M_x^\partial$ preserves the Frostman bound, smoothness, full quasi-support,
and full support.  The proofs of
Lemma~\ref{lem:lcp-centered-green-smoothing},
Proposition~\ref{prop:lcp-centered-green-inverse}, and
Corollary~\ref{cor:lcp-continuous-eigenfunctions} therefore give
\begin{equation}
 \GreenOp_\mu
 =\left(A_\mu\big|_{L_0^2(\mu)}\right)^{-1},
 \qquad
 \GreenOp_\mu:L^2(\mu)\longrightarrow C(\LCircle).
 \label{eq:lcp-tilted-green-interface}
\end{equation}

For $\sigma\in\mathcal M_{\rm resp}(\mu)$,
\begin{align*}
 \norm{\Potential_\mu\sigma}_{\Hc}^2
 &=\sum_{k\ne0}\frac{|\widehat\sigma(k)|^2}{|k|}<\infty,\\
 \inner{\Potential_\mu\sigma}{f}_{\Hc}
 &=\sum_{k\ne0}\widehat\sigma(k)\overline{\widehat f(k)}
 =\int_{\LCircle}f\,d\sigma,
 \qquad f\in\Hreg.
\end{align*}
If $\Potential_\mu\sigma=0$, all nonzero Fourier coefficients of $\sigma$
vanish; the zero-mass condition also gives $\widehat\sigma(0)=0$, so Fourier
uniqueness gives $\sigma=0$.  Thus $\Potential_\mu$ is injective.  Moreover,
for $a,b\in\mathbb R$,
\[
 \left(\sum_{k\ne0}
 \frac{|a\widehat\sigma_1(k)+b\widehat\sigma_2(k)|^2}{|k|}\right)^{1/2}
 \le |a|\norm{\Potential_\mu\sigma_1}_{\Hc}
     +|b|\norm{\Potential_\mu\sigma_2}_{\Hc},
\]
so $\mathcal M_{\rm resp}(\mu)$ is a vector space and
$\Potential_\mu$ is linear.

Let $\phi_1,\ldots,\phi_r$ be a real $L^2(\mu)$-orthonormal basis of a common
positive eigenspace.  Applying the proof of
Lemma~\ref{lem:lcp-finite-energy-centered-products} with $\mu$ in place of
$M_x^\partial$ gives
\[
 \phi_i\phi_j\mu-\delta_{ij}m
 \in\mathcal M_{\rm resp}(\mu),
 \qquad 1\le i,j\le r.
\]
Consequently, for $u=\sum_i a_i\phi_i$ and $v=\sum_jb_j\phi_j$,
\[
 \Theta_\mu(u,v)
 =\sum_{i,j=1}^ra_ib_j
  (\phi_i\phi_j\mu-\delta_{ij}m)
 \in\mathcal M_{\rm resp}(\mu).
\]
 This proves potential realizability and faithfulness.
\end{proof}

For $x\in\Omega_{\rm str}^{\partial}$ and a real normalized simple positive
eigenpair $(\Lambda_n^x,\phi_n^x)$, define
\begin{equation}
 \eta_n^x
 :=-\kappa\Lambda_n^x
 \left((\phi_n^x)^2M_x^\partial-m\right).
 \label{eq:lcp-centered-response-measure}
\end{equation}
Lemma~\ref{lem:lcp-finite-energy-centered-products} gives
\[
 \eta_n^x\in\mathcal M_{\rm resp}(M_x^\partial).
\]
For $f\in\Hc$, define the extended response functional by
\begin{equation}
 \ell_{n,\mathrm{ext}}^x(f)
 :=\inner{\Potential\eta_n^x}{f}_{\Hc}.
 \label{eq:lcp-centered-response-functional}
\end{equation}
For $f\in\Hreg$, Corollary~\ref{cor:abstract-simple-response} gives the
derivative formula, while equation~\eqref{eq:lcp-centered-product-pairing}
identifies it with the potential pairing:
\[
 \ell_{n,\mathrm{ext}}^x(f)
 =-\kappa\Lambda_n^x
   \int_{\LCircle}f(\phi_n^x)^2\,dM_x^\partial
 =\partial_f\Lambda_n^x.
\]
For $f\in\Hc\setminus\Hreg$, only the continuous extension
$\ell_{n,\mathrm{ext}}^x(f)$ is asserted; no analytic perturbation path is
attached to $f$.

For $1\le n_1<\cdots<n_N$ with simple eigenvalues, set
\begin{equation}
 C_{ij}(x)
 :=\inner{\Potential\eta_{n_i}^x}
          {\Potential\eta_{n_j}^x}_{\Hc},
 \qquad 1\le i,j\le N.
 \label{eq:lcp-centered-response-gram}
\end{equation}

\begin{corollary}
\label{cor:lcp-centered-response-gram}
For the data in \eqref{eq:lcp-centered-response-gram}, the centered response
potentials are linearly independent, and $C(x)$ is positive definite.
\end{corollary}

\begin{proof}
For $a=(a_1,\ldots,a_N)\in\mathbb R^N$,
\[
 \sum_{i=1}^Na_i\eta_{n_i}^x
 =-\kappa\sum_{i=1}^N
 a_i\Lambda_{n_i}^x
 \left((\phi_{n_i}^x)^2M_x^\partial-m\right).
\]
Corollary~\ref{cor:lcp-square-interface} and the weighted conclusion of
Theorem~\ref{thm:abstract-square-transversality} imply that this measure can
vanish only when $a=0$.  Proposition~\ref{prop:lcp-response-realization}
then gives linear independence of the potentials.  Hence, for $a\ne0$,
\begin{align*}
 \sum_{i,j=1}^Na_iC_{ij}(x)a_j
 &=\inner{\sum_{i=1}^Na_i\Potential\eta_{n_i}^x}
          {\sum_{j=1}^Na_j\Potential\eta_{n_j}^x}_{\Hc}\\
 &=\norm{\sum_{i=1}^Na_i\Potential\eta_{n_i}^x}_{\Hc}^2>0.
\end{align*}
\end{proof}

\small
\raggedright

\end{document}